\documentclass[10pt,a4paper,
oneside,%
final%
]{amsart}
\usepackage[T1]{fontenc}
\usepackage[utf8]{inputenc}
\usepackage{lmodern}
\usepackage{amsthm,amsfonts,amsmath,amssymb,latexsym,amscd}
\usepackage{mathtools}
\usepackage[final]{microtype} %
\usepackage{epsfig,graphics,color}
\usepackage{psfrag} %
\usepackage{pinlabel} %
\usepackage{graphicx}  %
\usepackage{tikz} %
\usetikzlibrary{decorations.markings} %
\usetikzlibrary{cd} %
\usetikzlibrary{calc} %
\usepackage[style=trad-abbrv,backend=biber,giveninits=true]{biblatex}
\usepackage{hyperref}
\usepackage{xurl}
\hypersetup{breaklinks=true}
\usepackage{mathrsfs}
\usepackage{verbatim}
\usepackage{bm}
\usepackage{enumerate}

\bibliography{./000_simpconn}
\renewbibmacro*{volume+number+eid}{
  \printfield{volume}
  \setunit*{\addnbspace}
  \printfield{number}
  \setunit{\addcomma\space}
  \printfield{eid}}
\DeclareFieldFormat[article]{number}{\mkbibparens{#1}}
\DeclareFieldFormat[article]{volume}{{\em{#1}}}
\title[Chain-level equivariant string topology: algebra versus analysis]{Chain-level equivariant string topology: algebra versus analysis}
\author{K.~Cieliebak, P.~H\'ajek and E.~Volkov}
\theoremstyle{plain}
\newtheorem{theorem}{Theorem}[section]
\newtheorem{thm}[theorem]{Theorem}
\newtheorem{corollary}[theorem]{Corollary}
\newtheorem{cor}[theorem]{Corollary}
\newtheorem{proposition}[theorem]{Proposition}
\newtheorem{prop}[theorem]{Proposition}
\newtheorem{lemma}[theorem]{Lemma}
\newtheorem{lem}[theorem]{Lemma}
\newtheorem{definition}[theorem]{Definition}

\theoremstyle{remark}

\newtheorem{question}[theorem]{Question}
\newtheorem{remark}[theorem]{Remark}
\newtheorem{rem}[theorem]{Remark}

\newcommand{\Id}{{{\mathchoice {\rm 1\mskip-4mu l} {\rm 1\mskip-4mu l}
{\rm 1\mskip-4.5mu l} {\rm 1\mskip-5mu l}}}}

\newcommand{\Y}{{\underbar{Y}}}
\newcommand{\T}{\mathbb{T}}
\newcommand{\D}{\mathbb{D}}
\newcommand{\ol}{\overline}

\newcommand{\p}{\partial}

\newcommand{\Om}{\Omega}

\newcommand{\into}{\hookrightarrow}
\newcommand{\onto}{\twoheadrightarrow}
\newcommand{\la}{\langle}
\newcommand{\ra}{\rangle}
\newcommand{\wt}{\widetilde}
\newcommand{\wh}{\widehat}
\newcommand{\N}{{\mathbb{N}}}
\newcommand{\Z}{{\mathbb{Z}}}
\newcommand{\R}{{\mathbb{R}}}
\newcommand{\C}{{\mathbb{C}}}
\newcommand{\Q}{{\mathbb{Q}}}
\newcommand{\K}{{\mathbb{K}}}
\renewcommand{\P}{{\mathbb{P}}}
\renewcommand{\r}{{\bf r}}
\renewcommand{\v}{{\bf v}}
\newcommand{\w}{{\bf w}}
\renewcommand{\a}{{\bf a}}
\newcommand{\e}{{\bf e}}
\newcommand{\m}{{\bf m}}
\newcommand{\F}{{\bf F}}
\newcommand{\G}{{\bf G}}
\newcommand{\M}{{\bf M}}
\newcommand{\A}{{\bf A}}
\newcommand{\B}{{\bf B}}
\newcommand{\E}{{\bf E}}
\newcommand{\X}{{\bf X}}
\newcommand{\J}{{\bf J}}

\newcommand{\bg}{{\bf g}}
\newcommand{\im}{{\rm im\,}}        %

\newcommand{\const}{{\rm const}}

\newcommand{\inn}{{\rm int\,}}

\newcommand{\can}{{\rm can}}
\newcommand{\ks}{{\rm KS}}

\newcommand{\pt}{{\rm pt}}

\newcommand{\IBL}{{\rm IBL}}
\newcommand{\dIBL}{{\rm dIBL}}

\newcommand{\DGA}{{\rm DGA}}
\newcommand{\CDGA}{{\rm CDGA}}
\newcommand{\dPD}{{\rm dPD}}
\newcommand{\PDGA}{{\rm PDGA}}
\newcommand{\ext}{{\rm ext}}

\newcommand{\ev}{{\rm ev}}
\newcommand{\cyc}{{\rm cyc}}

\newcommand{\Hom}{{\rm Hom}}

\newcommand{\Bl}{{\rm Bl}}
\newcommand{\alg}{{\rm alg}}
\newcommand{\ana}{{\rm ana}}

\newcommand{\MM}{\mathcal{M}}

\newcommand{\HH}{\mathcal{H}}

\newcommand{\PP}{\mathcal{P}}
\newcommand{\QQ}{\mathcal{Q}}

\renewcommand{\SS}{\mathcal{S}}
\newcommand{\TT}{\mathcal{T}}
\newcommand{\g}{{\mathfrak g}}         %
\newcommand{\h}{{\mathfrak h}}
\renewcommand{\t}{{\mathfrak t}}

\renewcommand{\o}{{\mathfrak o}}

\renewcommand{\u}{{\mathfrak u}}

\newcommand{\fp}{{\mathfrak p}}
\newcommand{\fq}{{\mathfrak q}}
\newcommand{\ff}{{\mathfrak f}}
\newcommand{\fm}{{\mathfrak m}}
\newcommand{\fn}{{\mathfrak n}}

\newcommand{\fe}{{\mathfrak e}}
\newcommand{\Flag}{{\rm Flag}}
\newcommand{\rmint}{{\rm int}}
\newcommand{\rmext}{{\rm ext}}

\newcommand{\SqueezeMath}[1]{\resizebox{!}{.7\baselineskip}{$#1$}}
\newcommand{\SOneEquiv}[1]{{#1}^{S^1}\!}
\newcommand{\dcbc}{B^{\text{\rm cyc}*}}
\newcommand{\rdcbc}{\ol{\dcbc}}
\newcommand{\cbc}{B^{\text{\rm cyc}}}
\newcommand{\Alg}{A}
\newcommand{\dprime}{{\prime\prime}}
\newcommand{\IBLinfty}{\IBL_\infty}
\newcommand{\Ainfty}{\mathrm{A}_\infty}
\newcommand{\fIBL}{\SqueezeMath{\wh{\IBL}}}
\newcommand{\fdcbc}{\SqueezeMath{\wh{\dcbc}}}
\newcommand{\rmbin}{\mathrm{bin}}
\newcommand{\dd}{\mathrm{d}}
\newcommand{\bb}{\mathrm{b}}
\newcommand{\HdR}{H_{\mathrm{dR}}}
\tikzset{
	point/.style = {draw, circle, fill=black, minimum size=4pt,inner sep=0pt},
	leaf/.style = {draw, circle, fill=white, minimum size=4pt,inner sep=0pt},
	root/.style = {draw, circle, fill=lightgray, minimum size=4pt,inner sep=0pt},
	}
\def\edgelen{3cm}
\def\leaflen{1cm}
\def\brancheangle{60}

\begin{document}

\begin{abstract}
We prove that on $2$-connected closed oriented manifolds, the analytic and algebraic constructions of an $\IBL_\infty$ structure associated to a closed oriented manifold coincide. The corresponding structure is invariant under orientation preserving homotopy equivalences and induces on homology the involutive Lie bialgebra structure of Chas and Sullivan.
\end{abstract}

\maketitle
\tableofcontents

\section{Introduction}
Let $M$ be a closed oriented manifold of dimension $n$ and $LM=C^\infty(S^1,M)$ its free loop space, equipped with the $S^1$-action by reparametrisation.
In their seminal 1999 paper~\cite{Chas-Sullivan99}, Chas and Sullivan described algebraic structures on the non-equivariant and equivariant homology of $LM$ which are commonly known under the name {\em string topology}:
\begin{itemize}
\item The degree shifted homology $H_{*+n}(LM)$ carries the structure of a Batalin-Vilkovisky algebra.
Moreover, the homology relative to the constant loops $H_*(LM,\const)$ carries a coproduct of degree $1-n$.
\item The degree shifted equivariant homology relative to the constant loops $H_{*+n-2}^{S^1}(LM,\const)$ carries the structure of an involutive Lie bialgebra (short {\em $\IBL$ algebra}), where the bracket has degree $0$ and the cobracket has degree $2(2-n)$.
\end{itemize}
The products on $H_{*+n}(LM)$ and $H_{*+n-2}^{S^1}(LM)$ are homotopy invariants~\cite{Cohen-Klein-Sullivan,Gruher-Salvatore,Crabb}, whereas the coproduct on $H_*(LM,\const)$ with $\Z$-coefficients is not~\cite{Naef}.

Since the work of Stasheff in the 1960s~\cite{Stasheff}, topologists have recognized the importance of studying the chain-level structures underlying operations on homology.
So it is natural to ask for the chain-level structures underlying string topology operations.
These structures are also important for applications in symplectic topology described by Fukaya and others, see, e.g., \cite{Fukaya-application,Cieliebak-Latschev}.
In the non-equivariant case, the chain level structure underlying the Batalin-Vilkovisky structure on $H_{*+n}(LM)$ has been described and constructed by Irie~\cite{Irie}. Here and in the sequel we adhere to the

{\bf Convention. }{\em The manifold $M$ is closed and oriented of dimension $n$, and all homology groups are with $\R$-coefficients}. 

In this paper we focus on the $S^1$-equivariant case.
Here the chain level structure underlying an involutive Lie bialgebra structure on homology is that of an {\em $\IBLinfty$ algebra} introduced in~\cite{Cieliebak-Fukaya-Latschev}.
We recall this notion in Section~\ref{sec:IBLinfty} below.
See also~\cite{Hoffbeck-Leray-Vallette} for further discussion of $\IBLinfty$ algebras from a properadic perspective.

If the manifold $M$ is simply connected, then it is known from Sullivan's groundbreaking work in the 1970s (see, e.g., \cite{Sullivan-infinitesimal}) that the real homotopy type of $M$ is determined by the minimal model of its de Rham complex $(\Om^*(M),\dd,\wedge)$.
Work of Chen, Jones and others (see, e.g., \cite{Chen73,Jones,Cieliebak-Volkov-cyclic})
expresses the homology of $LM$ in terms of the Hochschild and reduced cyclic cohomology of the de Rham complex,
\begin{equation}\label{eq:cychom-loophom}
   HH^*(\Om^*(M))\cong H_*(LM),\qquad  \ol{CH}^*(\Om^*(M))\cong H_*^{S^1}(LM,\pt).
\end{equation}
In his thesis~\cite{Basu-thesis}, Basu uses Sullivan's minimal model to compute string topology operations. 

In view of these results, it is natural to ask whether we can construct the $\IBLinfty$ algebra underlying equivariant string topology from the de Rham complex.
It is argued in~\cite{Cieliebak-Fukaya-Latschev} that this should be possible if we enhance the de Rham complex by its intersection pairing $\la a,b\ra=\int_Ma\wedge b$ to an {\em oriented Poincar\'e $\DGA$} 
\[
   \bigl(\Om^*(M),\dd,\wedge,\la\cdot,\cdot\ra\bigr).
\]
See Section~\ref{sec:PDGA} for the precise definition of an oriented Poincar\'e $\DGA$; besides graded commutativity, the main requirement is that the pairing descends to a nondegenerate pairing on cohomology.
A special case of this is a {\em differential Poincar\'e duality algebra}
$(\Alg,\dd,\wedge,\la\cdot,\cdot\ra)$ for which $\Alg$ is finite dimensional and the pairing is nondegenerate on $\Alg$.
Without graded commutativity, this corresponds to a {\em cyclic $\DGA$} in the terminology of~\cite{Cieliebak-Fukaya-Latschev}. 
In this algebraic situation we have the following result:

\begin{thm}[\cite{Cieliebak-Fukaya-Latschev}]\label{thm:CFL-intro}
Let $(\Alg,\dd,\wedge,\la\cdot,\cdot\ra)$ be a cyclic $\DGA$ of degree $n$ with cohomology $H=H(\Alg,\dd)$.
Then $\dcbc H[2-n]$ carries an $\IBLinfty$ structure which is $\IBLinfty$ homotopy equivalent to the twisted $\dIBL$ structure $\fp^{\fm}$ on $\dcbc \Alg[2-n]$.
In particular, its homology equals the cyclic cohomology of $(\Alg,\dd,\wedge)$.
\end{thm}

See Section~\ref{sec:IBLinfty} for the notions appearing in this theorem such as the dual cyclic bar complex $\dcbc \Alg$. 
In view of the second isomorphism in~\eqref{eq:cychom-loophom},  Theorem~\ref{thm:CFL-intro} suggests two approaches to chain level string topology.

{\em Algebraic approach: }Apply Theorem~\ref{thm:CFL-intro} to a suitable finite dimensional model of the de Rham complex 
$\Om^*(M)$. 

{\em Analytic approach: }Extend Theorem~\ref{thm:CFL-intro} to the de Rham complex $\Om^*(M)$, replacing finite sums by integrals over configuration spaces. 

The algebraic approach is based on the following theorem of Lambrechts and 
Stanley.

\begin{thm}[\cite{Lambrechts-Stanley}]\label{thm:Lambrechts-Stanley} 
For every simply connected Poincar\'e $\DGA$ $\Alg$ there exists a differential Poincar\'e duality algebra $\Alg^{\prime}$ which is connected to $\Alg$ by a zigzag of quasi-isomorphisms of commutative $\DGA$s.
\end{thm}

Applying Theorem~\ref{thm:Lambrechts-Stanley} to $A=\Om^*(M)$ for simply connected $M$, and then Theorem~\ref{thm:CFL-intro} to the differential Poincar\'e duality algebra $A'$, one obtains an $\IBLinfty$ structure on $\dcbc H_{\rm dR}^*(M)[2-n]$ whose reduced homology equals the reduced cyclic cohomology of $\Om^*(M)$ and therefore $H_*^{S^1}(LM,\pt)$. 
The algebraic approach is now completed by the following theorem of Naef and Willwacher.

\begin{thm}[\cite{Naef-Willwacher}]\label{thm:Naef-Willwacher}
Let $M$ be a simply connected closed manifold.
Then the involutive Lie bialgebra structure
on the reduced cyclic homology of $H=H_{\rm dR}^*(M)$
induced by the $\IBL_\infty$ structure 
on $\dcbc H[2-n]$ obtained {\em algebraically} 
is isomorphic to the involutive Lie bialgebra structure on $H_*^{S^1}(LM,\pt)$ due to Chas and Sullivan.
\end{thm}

However, the algebraic approach raises the issue that the resulting $\IBLinfty$ structure may depend on the chosen model and thus not be canonical.
This issue has been addressed by the second author in~\cite{Pavel-Hodge}.
To describe the results,
let us abbreviate ``differential Poincar\'e duality algebra'' by ``$\dPD$ algebra''.
Two Poincar\'e $\DGA$s are called \emph{weakly equivalent} if they are connected by a zigzag of Poincar\'e $\DGA$ quasi-isomorphisms (see Section~\ref{ss:PDGA} for more details). 
Recall that a graded $\R$-vector space $\Alg=\bigoplus_{k\geq 0}\Alg^k$ is called {\em $m$-connected} if $\Alg^0=\R$ and $\Alg^1=\cdots=\Alg^m=0$. 
Then we have the following improvement of Theorem~\ref{thm:Lambrechts-Stanley}.

\begin{thm}[\cite{Pavel-Hodge}]\label{thm:Pavel-Hodge-intro}
(a) Every Poincar\'e $\DGA$ $\Alg$ is weakly equivalent to a $\dPD$ algebra $\Alg_1$.
If the homology of $\Alg$ is $m$-connected, then 
$\Alg_1$ can be chosen to be $m$-connected. 

(b) Let $\Alg_1,\Alg_2$ be $2$-connected $\dPD$ algebras which are weakly equivalent as Poincar\'e $\DGA$s.  
Then there exists a $1$-connected $\dPD$ algebra $\Alg_3$ and quasi-isomorphisms of $\dPD$ algebras $\Alg_1\longleftarrow \Alg_3\longrightarrow \Alg_2$. 
\end{thm}

The analytic approach has been outlined in~\cite{Cieliebak-Fukaya-Latschev}.
It is completed by the following result.

\begin{thm}[\cite{Cieliebak-Volkov-Chern-Simons}]\label{thm:CV-intro}
Let $M$ be a closed oriented manifold of dimension $n$ and $H=\HdR(M)$ its de Rham cohomology.
Then the following hold.

(a) Integrals over configuration spaces give rise to an  $\IBLinfty$ structure on $\dcbc H[2-n]$ whose homology equals the cyclic cohomology of the de Rham complex of $M$.

(b) This structure is independent of all choices up to $\IBLinfty$ homotopy equivalence.
\end{thm}

Note that Theorem~\ref{thm:CV-intro}(b) resolves the issue of potential dependence on choices for the analytic approach.

\medskip

{\bf Results of this paper. }
Our first new result is

\begin{thm}\label{thm:alg-intro}
Let $\Alg,\Alg^{\prime}$ be Poincar\'e $\DGA$s with $2$-connected homology and $\Alg_1,\Alg_1^{\prime}$ associated $2$-connected $\dPD$ algebras as in Theorem~\ref{thm:Pavel-Hodge-intro}(a).
Assume that $\Alg$ and $\Alg^{\prime}$ are weakly equivalent.
Then the $\IBLinfty$ structures on $\dcbc H[2-n]$ obtained by applying Theorem~\ref{thm:CFL-intro} to $\Alg_1$ and $\Alg_1^{\prime}$ are $\IBLinfty$ homotopy equivalent. 
\end{thm}

In particular, we can apply this theorem to the de Rham complexes of closed oriented connected $n$-dimensional manifolds $M,M^{\prime}$.
We call  a homotopy equivalence $f\colon M\to M^{\prime}$ {\em orientation preserving} if the induced isomorphism $f_*\colon H_n(M)\to H_n(M^{\prime})$ maps the fundamental class $[M]$ to $[M^{\prime}]$.
This ensures that the pullback $f^*\colon \Om^*(M^{\prime})\to\Om^*(M)$ is a quasi-isomorphism of Poincar\'e $\DGA$s.
Hence, Theorem~\ref{thm:alg-intro} implies

\begin{cor}\label{cor-alg-intro}
Let $M,M^{\prime}$ be closed oriented connected $n$-dimensional manifolds with vanishing first and second de Rham cohomology.
Let  $\Alg_1,\Alg_1^{\prime}$ be $2$-connected $\dPD$ algebras associated to the de Rham complexes $\Om^*(M),\Om^*(M^{\prime})$ as in Theorem~\ref{thm:Pavel-Hodge-intro}(a).
Assume that there exists an orientation preserving homotopy equivalence between $M$ and $M^{\prime}$.
Then the  $\IBLinfty$ structures on $\dcbc H[2-n]$ obtained  by applying Theorem~\ref{thm:CFL-intro} to $\Alg_1$ and $\Alg_1^{\prime}$ are $\IBLinfty$ homotopy equivalent.
\hfill$\square$
\end{cor}

This accomplishes the algebraic construction of a canonical (up to $\IBLinfty$ homotopy equivalence) $\IBLinfty$ algebra associated to a closed oriented connected manifold with $H^1_{\rm dR}(M)=H^2_{\rm dR}(M)=0$.
Our second result says that this agrees with the analytic construction in Theorem~\ref{thm:CV-intro}. 

\begin{thm}\label{thm:comparison-intro}
Let $M$ be a closed oriented connected manifold and $H=\HdR(M)$ its de Rham cohomology.
If $H^1_{\rm dR}(M)=H^2_{\rm dR}(M)=0$,  then the $\IBLinfty$ structures on $\dcbc H[2-n]$ arising from Theorem~\ref{thm:CV-intro} and Corollary~\ref{cor-alg-intro}
are $\IBLinfty$ homotopy equivalent.  
\end{thm}

This comparison result has the following immediate corollary.

\begin{cor}\label{cor:homotopy-invariance-intro}
In the class of closed oriented connected manifolds with $H^1_{\rm dR}=H^2_{\rm dR}=0$, the $\IBLinfty$ structure arising from the analytic construction in Theorem~\ref{thm:CV-intro} is invariant (up to $\IBLinfty$ homotopy equivalence) under orientation preserving homotopy equivalences. 
\hfill$\square$
\end{cor}

Combining this with Theorem~\ref{thm:Naef-Willwacher}, we obtain

\begin{cor}\label{cor:homotopy-invariance-simpconn-intro}
In the class of closed oriented {\em simply connected} manifolds with $H^2_{\rm dR}=0$, the $\IBLinfty$ structure arising from Theorem~\ref{thm:CV-intro} is invariant (up to $\IBLinfty$ homotopy equivalence) under orientation preserving homotopy equivalences. Moreover, the involutive Lie bialgebra structure on its reduced homology 
is isomorphic to the one on reduced equivariant loop space homology due to Chas and Sullivan.
\hfill$\square$
\end{cor}

Corollary~\ref{cor:homotopy-invariance-simpconn-intro} establishes, in the class of simply connected manifolds with $H^2_{\rm dR}=0$, the orientation preserving homotopy invariance of the $\IBLinfty$ structure underlying $S^1$-equivariant string topology. 
To our knowledge, this is the first result on homotopy invariance of string topology operations on the chain level. 
The homotopy invariance statement for the string cobracket appears to be new even on the level of homology. 

Another corollary of Theorem~\ref{thm:comparison-intro} is the description of the $\IBLinfty$ structures in the case that $M$ is formal in the sense of rational homotopy theory (see~\S\ref{ss:formality} for the relevant definitions):

\begin{cor}[Formality implies $\IBLinfty$ formality]\label{cor:formal-intro}
Let $M$ be a closed oriented connected manifold and $H=\HdR(M)$ its de Rham cohomology.
Assume that $M$ is formal and $H^1_{\rm dR}(M)=H^2_{\rm dR}(M)=0$.
Then the analytic $\IBLinfty$ structure on $\dcbc H[2-n]$ in Theorem~\ref{thm:CV-intro}
is $\IBLinfty$ homotopy equivalent to the canonical $\dIBL$ structure $\dIBL^\fm(H)$ which is only twisted by the triple intersection product. 
\end{cor}

\begin{rem}
Theorem~\ref{thm:CV-intro} is in fact a consequence of a more refined statement: $\dcbc H[2-n]$ carries a Maurer--Cartan element for its untwisted $\dIBL$ structure which is independent of all choices up to $\IBLinfty$ gauge equivalence.
Similarly, the $\IBLinfty$ structure in Theorem~\ref{thm:CFL-intro} is obtained by twisting with a Maurer--Cartan element.
Combining the proof of Theorem~\ref{thm:CV-intro} with Lemma~\ref{lem:relative-hodge-decomposition} below, Theorem~\ref{thm:alg-intro} can probably be upgraded to $\IBLinfty$ gauge equivalence of the corresponding Maurer--Cartan elements. 
\end{rem}

\begin{rem}
In the upcomping paper~\cite{Cieliebak-Volkov-string} it is proved {\em analytically} (using Chen's iterated integrals) that, for a simply connected closed manifold, 
the involutive Lie bialgebra structure on the reduced homology of the $\IBL_\infty$ structure from Theorem~\ref{thm:CV-intro} is isomorphic to the one on reduced equivariant loop space homology due to Chas and Sullivan. Thus, the last assertion in Corollary~\ref{cor:homotopy-invariance-simpconn-intro} is true without the hypothesis $H^2_{\rm dR}=0$. 
\end{rem}

\begin{rem}
In the new version of~\cite{Pavel-Hodge} it is proved that in all the results above we can drop the hypothesis $H^2_{\rm dR}=0$ if the manifold has {\em odd dimension}. It would be interesting to know whether this also holds in the even dimensional case. 
\end{rem}

{\bf Acknowledgements. }
We thank H{\^{o}}ng V{\^{a}}n L{\^{e}}, Thomas Willwacher and Florian Naef for interesting discussions.

\section{Poincar\'e DGAs}\label{sec:PDGA}
\subsection{Cochain complexes with pairing}\label{ss:cochain}

Let $\N=\{1,2,\dotsc\}$ be the set of natural numbers, $\N_0\coloneqq \N\cup\{0\}$, $\Z$ the ring of integers, $\Q$ the field of rational numbers, and~$\R$ the field of real numbers.

We will work in the category of $\Z$-graded vector spaces $\Alg=\bigoplus_{i\in\Z} \Alg^i$ over~$\R$ with morphisms
\begin{align*}
	\Hom(\Alg_1,\Alg_2)&\coloneqq \bigoplus_{i\in\Z}\Hom^i(\Alg_1,\Alg_2),\\
	\Hom^i(\Alg_1,\Alg_2)&\coloneqq\{f\colon\Alg_1\to\Alg_2 \text{ linear homogenous}, \deg f = i\}.
\end{align*}
Here we say that $f:\Alg_1\to\Alg_2$ is \emph{homogenous of degree} $\deg f\in \Z$ if $f(\Alg^i_1)\subset \Alg_2^{i+\deg f}$ for all $i\in\Z$, and that $x\in \Alg$ is \emph{homogenous of degree} $\deg x\in \Z$ if $x\in\Alg^{\deg x}$.

We say that $\Alg$ is \emph{nonnegatively graded} if $\Alg=\bigoplus_{i\in\N_0} \Alg^i$, and of {\em finite type} if $\dim \Alg^i < \infty$ for all $i\in \Z$.%

We denote by $\Alg[n]$ the \emph{degree shift} of $\Alg$ by $n\in \Z$, i.e.,  $\Alg[n]^i=\Alg^{i+n}$ for all $i\in\Z$.

A {\em pairing of degree $n\in \Z$} on $\Alg$ is a bilinear form $\la\cdot,\cdot\ra\colon \Alg\times \Alg\to \R$ which for all homogenous $x, y\in \Alg$ satisfies the degree condition%
\footnote{Note that this is equivalent to $\deg \la\cdot,\cdot\ra=-n$ in $\Hom(\Alg^{\otimes 2},\R)$.}
\begin{equation*}
	\la x,y\ra\neq 0\quad\Longrightarrow\quad\deg x+\deg y=n
\end{equation*}
and graded symmetry
\begin{align}\label{Eq:Symmetry}
   \la x, y \ra = (-1)^{\deg x \deg y} \la y, x \ra.
\end{align}
We write $x\perp y$ if $\la x, y \ra = 0$ and say that $x,y$ are \emph{orthogonal}.
The subspace of elements of $\Alg$ orthogonal to a given subspace $B\subset\Alg$ will be denoted by 
\[
	B^{\perp} \coloneqq \{ x\in \Alg\mid x\perp B\}.
\]
We call
a pairing $\la\cdot,\cdot\ra\colon\Alg\times\Alg\to\R$ \emph{nondegenerate} if the \emph{musical map}
\begin{equation*}
	\begin{aligned}
	\flat\colon\Alg&\longrightarrow\Hom(\Alg,\R)\\
	x&\longmapsto x^\flat \coloneqq \la x,\cdot\ra
	\end{aligned}
\end{equation*}
is injective, and {\em perfect} if it is an isomorphism.
In that case we denote its inverse by 
\[
	\#\colon\Hom(\Alg,\R)\longrightarrow\Alg.
\]
Note that $\flat$ and $\#$ are linear homogenous maps of degrees $\deg\flat = - n$ and $\deg \#=n$.

We point out the following facts which are easy to prove:
\begin{enumerate}
\item A pairing $\la\cdot,\cdot\ra\colon\Alg\times\Alg\to\R$ is perfect if and only if it is nondegenerate and~$\Alg$ is of finite type. 
\item If $\Alg$ is nonegatively graded, then any nontrivial pairing $\la\cdot,\cdot\ra\colon\Alg\times\Alg\to\R$ must have nonnegative degree $n\in \N_0$, and the existence of a nondegenerate pairing of degree $n$ implies $\Alg=\bigoplus_{i=0}^n\Alg^i$.
\item If $\Alg$ is nonegatively graded, then a pairing $\la\cdot,\cdot\ra\colon\Alg\times\Alg\to\R$ is perfect if and only if it is nondegenerate and $\dim\Alg<\infty$.
\end{enumerate}

{\bf Convention. }\textit{From now on all pairings will be of degree $n\in \Z$.}

Consider now a cochain complex $(\Alg,\dd)$, i.e., a graded vector space $\Alg$ together with a differential $\dd\in\Hom^1(\Alg,\Alg)$, $\dd\circ\dd=0$.
A {\em pairing} on $(\Alg,\dd)$ is a pairing $\la\cdot,\cdot\ra\colon\Alg\times\Alg\to\R$ satisfying for all homogenous $x,y\in\Alg$,
\begin{equation*}
   \la \dd x, y \ra = (-1)^{1+\deg x} \la x, \dd y \ra. %
\end{equation*}
Such a pairing descends naturally to a pairing $\la\cdot,\cdot\ra_H\colon H(\Alg)\times H(\Alg)\to\R$ on cohomology $(H(\Alg)\coloneqq H(\Alg,\dd),\dd=0)$.
Following the terminology of~\cite{Cieliebak-Fukaya-Latschev}, we call a cochain complex with a perfect pairing $(\Alg,\dd,\la\cdot,\cdot\ra)$ a {\em cyclic cochain complex}.%
\footnote{In~\cite{Cieliebak-Fukaya-Latschev} the authors use $(-1)^{\deg x}\la x,y\ra$ instead of $\la x,y\ra$.}
    
{\bf Propagators and symmetric projections. }
Let $(\Alg,\dd,\la\cdot,\cdot\ra)$ be a cochain complex with pairing.
A {\em projection} is a chain map $\pi\in\Hom^0(\Alg,\Alg)$ which satisfies $\pi\circ \pi = \pi$. 
We say that $P\in\Hom^{-1}(\Alg,\Alg)$ is a \emph{homotopy operator} if the map $-\dd\circ P - P \circ\dd\colon\Alg\to\Alg$ is a projection.
A homotopy operator $P$ is called \emph{special} if
\begin{align}
	P\circ P &= 0\quad\text{and}\label{eq:P5}\\
	P\circ\dd \circ P &= - P. \label{eq:PP2}
\end{align}
Every homotopy operator $P$ determines a projection 
\[
	\pi_P\coloneqq \Id+\dd\circ P+P\circ\dd\colon\Alg\longrightarrow\Alg
\]
which is a {\em quasi-isomorphism}, i.e., the induced map on cohomology $H(\pi_P)\colon H(\Alg)\to H(\Alg)$ is an isomorphism.
Given a projection $\pi\colon\Alg\to\Alg$ which is a quasi-isomorphism, we say that $P\in\Hom^{-1}(\Alg,\Alg)$ is a \emph{homotopy operator with respect to $\pi$} if it is a homotopy operator and $\pi_P=\pi$, so that
\begin{equation}
	\dd\circ P+P\circ\dd=\pi-\Id. \label{eq:P2} 
\end{equation}
Note: assuming \eqref{eq:P5} and \eqref{eq:P2}, condition \eqref{eq:PP2} on $P\in\Hom^{-1}(\Alg,\Alg)$ holds if and only if 
\begin{equation}
  P\circ \pi = \pi\circ P = 0. \label{eq:P4}
\end{equation}

We say that a homotopy operator $P\colon \Alg\to\Alg$ is a \emph{propagator}\footnote{
Thus a ``propagator'' is a partial inverse of the differential $d$ and not of the Laplace operator.}
if it satisfies the symmetry property
\begin{equation}
	 \la Px,y\ra = (-1)^{\deg x}\la x,Py\ra. \label{eq:P3} \\
\end{equation}
The associated projection $\pi_P\colon\Alg\to\Alg$ is then \emph{symmetric}:
\begin{equation*}
	 \la\pi_P x,y\ra = \la x,\pi_P y\ra.
\end{equation*}

\begin{lemma}[{\cite[Remark 2]{Cattaneo-Mnev}}]\label{lem:propagator}
Any propagator $P$ can be modified to a special propagator $P_3$ with respect to the same projection $\pi$ by setting
\[
   P_2 \coloneqq (\pi-\Id)\circ P\circ (\pi-\Id),\qquad P_3 \coloneqq -P_2\circ \dd\circ P_2.
\]
\end{lemma}
Note that a special propagator with respect to a given projection is not unique in general.

Given a subcomplex $B\subset \Alg$, we say that a projection $\pi\colon\Alg\to\Alg$ is onto $B$ if $\im\pi=B$ and identify $\pi$ with the induced surjection $\pi\colon\Alg\onto B$ in that case.
One can show that a homotopy operator~$P\colon\Alg\to\Alg$ with respect to a projection $\pi\colon\Alg\onto B$ exists if and only if~$\pi$ is a quasi-isomorphism (cf.~the construction in the proof below). 
In the case with pairing, we have the following:

\begin{lemma}\label{lem:retraction}
Let $(\Alg,\dd,\la\cdot,\cdot\ra)$ be a cyclic cochain complex and $B\subset \Alg$ a subcomplex.
A propagator $P$ with respect to a projection $\pi\colon\Alg\onto B$ exists if and only if~$\pi$ is symmetric and a quasi-isomorphism.
\end{lemma}

\begin{proof}
Suppose that $\pi\colon\Alg\onto B$ is a symmetric projection and a quasi-isomorphism.
Then $\ker \pi\subset\Alg$ is an acyclic subcomplex, and thus there exists a subspace $C\subset \ker \pi$ such that 
$\ker \pi = C\oplus \dd C$. Consider the linear map $P\colon \Alg\to \Alg$ defined by
\begin{equation*}
		P(x) \coloneqq \begin{cases}
			-y & \text{if } x= \dd y \text{ for some } y\in C, \\
			0 & \text{for } x\in B\oplus C.
		\end{cases}
\end{equation*}
Then $P$ is a homotopy operator with respect to $\pi$.
Due to the perfection of $\la\cdot,\cdot\ra\colon \Alg\times\Alg\to\R$, the operator $P\colon\Alg\to\Alg$ has a unique adjoint $P^{\dagger}\colon \Alg\to\Alg$, i.e., a linear map such that $\la P x, y\ra = (-1)^{\deg x}\la x, P^{\dagger}y\ra$ for every $x, y\in \Alg$.
Clearly, $P^{\dagger}$ is also a homotopy operator with respect to~$\pi$, hence
$P_1\coloneqq \frac{1}{2}(P + P^\dagger)$ is a propagator with respect to~$\pi$.
\end{proof}

We say that~$B\subset\Alg$ is a \emph{quasi-isomorphic subcomplex} if it is a subcomplex and the inclusion $\iota\colon B\into \Alg$ is a quasi-isomorphism. 
One can show that any quasi-isomorphic subcomplex $B\subset\Alg$ admits a projection $\pi\colon \Alg\onto B$.
Given a cochain complex with pairing $(\Alg,\dd,\la\cdot,\cdot\ra)$, every subcomplex $B\subset \Alg$ such that $\Alg = B\oplus B^\perp$ admits a unique symmetric projection $\pi_B\colon \Alg \onto B$, which is the projection with $\ker \pi_B = B^\perp$.
This occurs in particular when $B\subset\Alg$ is a \emph{cyclic subcomplex}, i.e., the restriction $\la\cdot,\cdot\ra|_B\colon B\times B \to \R$ is perfect, so that $(B,\dd,\la\cdot,\cdot\ra|_B)$ is a cyclic cochain complex.
Lemmas~\ref{lem:propagator} and~\ref{lem:retraction} now imply the following:

\begin{cor}
Let $(\Alg,\dd,\la\cdot,\cdot\ra)$ be a cyclic cochain complex and $B\subset \Alg$ a quasi-isomorphic cyclic subcomplex.	
Then there exists a special propagator $P\colon\Alg\to\Alg$ such that $\im \pi_P=B$.
\qed
\end{cor}

{\bf Harmonic subspaces and projections. }
Let $(\Alg,\dd,\la\cdot,\cdot\ra)$ be a cochain complex with pairing.
We say that $\HH\subset\Alg$ is a \emph{harmonic subspace} if
\[
	\ker\dd = \HH\oplus \im\dd.
\]
We call a projection $\pi\colon \Alg\to\Alg$ \emph{harmonic} if it is symmetric and $\im\pi\subset\Alg$ is a harmonic subspace.
Note that the last condition implies $\dd\circ\pi=\pi\circ \dd=0$. 
Given a harmonic subspace $\HH\subset\Alg$, the inclusion $\iota\colon \HH\into \Alg$ is a quasi-isomorphism, hence any harmonic projection $\pi\colon\Alg\onto\HH$ is a quasi-isomorphism.
If a harmonic projection onto $\HH$ exists and $\la\cdot,\cdot\ra|_\HH\colon\HH\times\HH\to\R$ is nondegenerate, then $\Alg=\HH\oplus\HH^{\perp}$ and every harmonic projection agrees with the unique symmetric projection $\pi_\HH\colon\Alg\onto\HH$.

\begin{lemma}\label{lem:harmonic}
Suppose that $(\Alg,\dd,\la\cdot,\cdot\ra)$ is a cyclic cochain complex.
Then every harmonic subspace $\HH\subset \Alg$ is a quasi-isomorphic cyclic subcomplex and the unique symmetric projection $\pi_\HH\colon \Alg \onto \HH$ is harmonic.
\end{lemma}

\begin{proof}
Let $\alpha\colon\Alg\to\Alg$ be a linear map of degree $-1$ such that $\Id - \alpha\circ\dd$ is a projection onto $\ker\dd$.
Such $\alpha$ can be constructed by choosing a complement $C\subset\Alg$ of $\ker\dd$ and defining $\alpha(\dd x)\coloneqq x$ for $x\in C$ and $\alpha(x)\coloneqq 0$ for $x\in\HH\oplus C$.
Given a homogenous element $x\in \Alg$, let $x^\prime$ denote the image of $(-1)^{\deg x}\la x,\cdot\ra \circ \alpha$ under the musical isomorphism $\#\colon\Hom(\Alg,\R)\to\Alg$.
A straightforward computation shows that $x\perp \ker\dd$ implies $x=\dd x^\prime$.
Consequently, $\la\cdot,\cdot\ra|_\HH\colon\HH\times\HH\to\R$ is nondegenerate, hence perfect since $\la\cdot,\cdot\ra$ is perfect by assumption.
The projection $\pi_\HH$ is clearly harmonic because $\im\dd\subset \HH^{\perp}=\ker \pi_\HH$.
\end{proof}

Lemma~\ref{lem:harmonic} implies that the cohomology $(H(\Alg),\dd=0,\la\cdot,\cdot\ra_H)$ of a cyclic cochain complex $(\Alg,\dd,\la\cdot,\cdot\ra)$ equipped with the induced pairing is a cyclic cochain complex as well.
Indeed, if $\HH\subset\Alg$ is a harmonic subspace, then $(H(\Alg),\dd=0,\la\cdot,\cdot\ra_H)$ can be canonically identified with $(\HH,\dd=0,\la\cdot,\cdot\ra|_\HH)$ via the quotient map
\begin{equation}\label{eq:canon-quotient-to-hom}
	\pi\colon \HH\subset \ker\dd \longrightarrow H(\Alg) = \ker\dd/\im\dd.	
\end{equation}

Let $P\colon\Alg\to\Alg$ be a propagator in a cochain complex with pairing $(\Alg,\dd,\la\cdot,\cdot\ra)$.
Then the associated projection $\pi_P\colon\Alg\to\Alg$ satisfies
\[
	\pi_P \circ\dd =\dd + \dd\circ P\circ\dd =\dd\circ \pi_P, 
\]
and it follows that $\pi_P$ is harmonic if and only if $P$ satisfies condition~\eqref{eq:PP2}.

{\bf Hodge decompositions. }
Let $(\Alg,\dd,\la\cdot,\cdot\ra)$ be a cochain complex with pairing.
A~\emph{Hodge decomposition} of $(\Alg,\dd,\la\cdot,\cdot\ra)$ is the data of a harmonic subspace $\HH\subset\Alg$ and a complement $C$ of $\ker\dd$ in $\Alg$ such that 
\[
   C\perp C\oplus \HH.
\]
Associated to each Hodge decomposition $\Alg=\HH\oplus \im\dd\oplus C$, there is a canonical harmonic projection $\pi_{\HH,C}\colon \Alg \onto\HH$ defined by
\begin{align*}
	\pi_{\HH,C}(x) = \begin{cases}
		x & \text{for } x\in \HH,\\
		0 & \text{for } x\in \im\dd\oplus C,
	\end{cases}
\end{align*}
and a canonical special propagator $P_{\HH,C}\colon \Alg\to \Alg$ with respect to $\pi_{\HH,C}$ defined by
\begin{equation}\label{eq:special}
	P_{\HH,C}(x) \coloneqq \begin{cases} 
	-y & \text{if }x=\dd y\text{ for some }y\in C,\\
		0 & \text{for }x\in \HH\oplus C.
		\end{cases}
\end{equation}
Note that if in addition $\la\cdot,\cdot\ra\colon\Alg\times\Alg\to\R$ is nondegenerate, then also the restriction $\la\cdot,\cdot\ra|_\HH\colon\HH\times\HH\to\R$ is nondegenerate and every harmonic projection $\pi:A\onto\HH$ thus agrees with~$\pi_{\HH,C}$. 
The proof of the following lemma is straightforward:

\begin{lemma}\label{lem:special}
Let $(\Alg,\dd,\la\cdot,\cdot\ra)$ be a cochain complex with pairing.
Then \eqref{eq:special} defines a one-to-one correspondence between Hodge decompositions of $\Alg$ and special propagators $P\colon\Alg\to\Alg$. 
Under this correspondence, the following holds:
\begin{equation*}
	\HH=\im(\Id+\dd \circ P + P \circ\dd), \quad C=\im P,\quad \pi_{\HH,C}=\pi_P. 
\end{equation*}
\end{lemma}

Following \cite{Fiorenza+}, we say that $(\Alg,\dd,\la\cdot,\cdot\ra)$ is of \emph{Hodge type} if it admits a Hodge decomposition.
If $(\Alg,\dd,\la\cdot,\cdot\ra)$ is of Hodge type and the induced pairing on cohomology is perfect, then \cite[Remark~2.6]{Fiorenza+} shows that for every harmonic subspace $\HH\subset\Alg$ there is a subspace $C\subset\Alg$ such that $\Alg=\HH\oplus \im\dd\oplus C$ is a Hodge decomposition.
This together with Lemma~\ref{lem:special} implies the following:

\begin{cor}\label{cor:special-Hodge}
A cochain complex with pairing $(\Alg,\dd,\la\cdot,\cdot\ra)$ is of Hodge type if and only if it admits a special propagator.
Suppose that this is the case and the induced pairing on cohomology is perfect.
Let $\HH\subset\Alg$ be any harmonic subspace. 
Then there exists a unique harmonic projection $\pi_\HH\colon\Alg\onto\HH$ and a (possibly nonunique) special propagator $P\colon\Alg\to\Alg$ with respect to $\pi_\HH$.
\qed
\end{cor}

Lemmas~\ref{lem:propagator}, \ref{lem:retraction}, \ref{lem:harmonic}, \ref{lem:special} imply the following:

\begin{lemma}[{\cite[Lemma 11.1]{Cieliebak-Fukaya-Latschev}}]\label{lem:hodge-decomposition}
Any cyclic cochain complex is of Hodge type.	
\end{lemma}

Let us now discuss functoriality.
We say that a chain map $f\colon\Alg\to\Alg^{\prime}$ \emph{respects Hodge decompositions} $\Alg=\HH\oplus\im\dd\oplus C$ and $\Alg^{\prime}=\HH^{\prime}\oplus\im \dd^{\prime}\oplus C^{\prime}$ if 
\begin{equation}\label{eq:respect-hodge}
	f(\HH)\subset\HH^{\prime}\quad\text{and}\quad f(C)\subset C^{\prime}.
\end{equation}
This implies in particular 
\[
	f\circ\pi_{\HH,C}=\pi_{\HH^\prime,C^\prime}\circ f.
\]
If $P$ and $P^\prime$ are the associated special propagators, then \eqref{eq:respect-hodge} is equivalent to
\begin{equation*}
	f\circ P = P^{\prime}\circ f.
\end{equation*}

{\bf Relative Hodge decompositions. }
Let $(\Alg,\dd,\la\cdot,\cdot\ra)$ be a cochain complex with pairing, and let $B\subset \Alg$ be a subcomplex.
We say that a Hodge decomposition $\Alg=\HH\oplus\im\dd\oplus C$ is \emph{relative} to $B$ if 
\begin{equation*}
B=(B\cap\HH)\oplus (B\cap\im\dd) \oplus (B\cap C)
\end{equation*}
is a Hodge decomposition of~$(B,\dd,\la\cdot,\cdot\ra|_B)$.
The inclusion $\iota\colon B\into\Alg$ is then a chain map which respects the Hodge decompositions.
Let $P^\Alg\colon \Alg\to\Alg$ and $P^B\colon B\to B$ be the special propagators corresponding to the Hodge decompositions of $\Alg$ and $B$, respectively, and suppose that there is a symmetric projection $\pi\colon\Alg\onto B$ which is a quasi-isomorphism.
Then a simple computation shows that
	\[
		P\coloneqq (\Id-\pi)\circ P^{\Alg}\colon\Alg\longrightarrow\Alg
	\]
is a propagator with respect to $\pi$.
If $\pi$ respects the Hodge decompositions, then $P= (\Id-\pi)\circ P^{\Alg} = P^B\circ(\Id-\pi)$ is special.

\begin{lemma}\label{lem:relative-hodge-decomposition}
Let $(\Alg,\dd,\la\cdot,\cdot\ra)$ be a cyclic cochain complex and $B\subset\Alg$ a quasi-isomorphic cyclic subcomplex. Then there exists a Hodge decomposition of~$\Alg$ relative to~$B$ such that the unique symmetric projection $\pi_B\colon\Alg\onto B$ respects the Hodge decompositions.
\end{lemma}

\begin{proof}
The subcomplex $B$ and its orthogonal complement $B^{\prime}\coloneqq B^{\perp}$, both equipped with the restrictions of $\la\cdot,\cdot\ra$, are cyclic cochain complexes and hence admit Hodge decompositions $B=\HH\oplus \dd C \oplus C$ and $B^{\prime} = \HH^\prime\oplus \dd C^{\prime} \oplus C^{\prime}$ by Lemma~\ref{lem:hodge-decomposition}.
Setting $\HH^\dprime \coloneqq \HH\oplus\HH^\prime$ and $C^{\dprime}\coloneqq C\oplus C^{\prime}$, we obtain a Hodge decomposition $\Alg=\HH\oplus \dd C^{\dprime} \oplus C^{\dprime}$.
The rest is a straightforward check.
\end{proof}

{\bf Hodge star. }
Let $(\Alg,\dd,\la\cdot,\cdot\ra)$ be a nonnegatively graded cochain complex with pairing of degree $n\in\N_0$.
A \emph{Hodge star} on $\Alg$ is a linear map $\star\colon\Alg\to\Alg$
which can be written as $\star = \sum_{k=0}^n \star^k$ for linear maps $\star^k\colon\Alg^k\to\Alg^{n-k}$ satisfying
\[
	\star^{n-k}\circ\star^k = (-1)^{k(n-k)}\Id\colon\Alg^k\longrightarrow\Alg^k
\]
such that the following bilinear form is a positive definite inner product:
\[
	(\cdot,\cdot)\coloneqq\la\cdot,\star\cdot\ra\colon \Alg\times\Alg\longrightarrow\R.
\]
We define the codifferential
\[
	\dd^\star \coloneqq \sum_{k=0}^n (-1)^{1+n(k-1)} \star^{n-k+1}\circ \dd\circ \star^k\colon \Alg\longrightarrow\Alg
\] 
which satisfies $(\dd x,y)=(x,\dd^\star y)$.
The nondegeneracy of $(\cdot,\cdot)$ then implies that $\ker\dd$ is the orthogonal subspace to $\im \dd^\star$ and $\ker \dd^\star$ is the orthogonal subspace to $\im\dd$ with respect to $(\cdot,\cdot)$.
We define 
\[
	\HH\coloneqq \ker\dd \cap \ker \dd^\star
\]
and suppose that the following condition is satisfied:
\begin{enumerate}
	\item[(proj)] There exist projections $\Alg\onto\ker\dd$ and $\ker\dd\onto \HH$ which are symmetric with respect to $(\cdot,\cdot)$.
\end{enumerate}
Then we obtain the decomposition
\begin{equation}\label{eq:HodgeStar}
	\Alg = \HH \oplus \im\dd \oplus \im \dd^\star
\end{equation}
which is orthogonal with respect to $(\cdot,\cdot)$, and it follows from the previous relations that~\eqref{eq:HodgeStar} is also a Hodge decomposition.
Let $P_\star\colon\Alg\to\Alg$ be the special propagator and $\pi_\star \colon\Alg\onto\HH$ the harmonic projection canonically associated to~\eqref{eq:HodgeStar}.
Since $\la\cdot,\cdot\ra\colon\Alg\times\Alg\to\R$ is nondegenerate, $\pi_\star$ can be equivalently characterized as the unique symmetric projection $\pi_\star=\pi_\HH\colon\Alg\onto\HH$.
The special propagator~$P_\star$ can also be equivalently characterized as the unique homotopy operator $P\colon\Alg\to\Alg$ with respect to $\pi_\HH$ such that $\im P \subset \im \dd^\star$. 

Two important examples which admit a Hodge star such that the condition (proj) above holds arise when $\dim \Alg < \infty$, or when $\Alg = \Om$ is the de Rham complex of an oriented closed manifold (see Section~\ref{sec:analytic-construction}).  

The following lemma follows by inspection of the construction of the Hodge star in the proof of \cite[Lemma 11.1]{Cieliebak-Fukaya-Latschev}: 

\begin{lemma}\label{lem:rel-hodge-star}
Let $(\Alg,\dd,\la\cdot,\cdot\ra)$ be a nonnegatively graded cyclic cochain complex and $B\subset \Alg$ a cyclic subcomplex.
Then there exists a Hodge star on $\Alg$ which restricts to a Hodge star on $B$.
\qed
\end{lemma}

In the situation of this lemma, the corresponding Hodge decomposition of $\Alg$ is relative to the corresponding Hodge decomposition of $B$,
which gives an alternative proof of Lemma~\ref{lem:relative-hodge-decomposition}.

{\bf Nondegenerate quotient. }
Let $(\Alg,\dd,\la\cdot,\cdot\ra)$ be a cochain complex with pairing.
We consider the degenerate subspace $\Alg_{\mathrm{deg}}\coloneqq\{a\in \Alg \mid a \perp \Alg\}$ and define the \emph{nondegenerate quotient}
\[
   \QQ(\Alg) \coloneqq \Alg / \Alg_{\mathrm{deg}}.
\]
Since~$\Alg_{\mathrm{deg}}\subset\Alg$ is a subcomplex, the differential $\dd$ descends to a differential~$\dd_\QQ$ on $\QQ\coloneqq \QQ(\Alg)$.
Moreover, the pairing $\la\cdot,\cdot\ra\colon\Alg\times\Alg\to\R$ descents to a nondegenerate pairing $\la\cdot,\cdot\ra_\QQ\colon\QQ\times\QQ\to\R$, and the quotient map $\pi_\QQ\colon \Alg\onto \QQ$ is a chain map preserving the pairings.

\begin{lemma}[{%
\cite[Lemma 3.3]{Pavel-Hodge}, \cite[Prop.~6.1.17]{Pavel-thesis}}]\label{lem:nondeg-quotient}
Let $(\Alg,\dd,\la\cdot,\cdot\ra)$ be a cochain complex with pairing, and let $\la\cdot,\cdot\ra_H\colon H(\Alg)\times H(\Alg)\to\R$ be the induced pairing on cohomology.
Consider the nondegenerate quotient $\QQ\coloneqq\QQ(\Alg)$.
Then:
\begin{enumerate}[(a)]
\item If $\Alg$ is of Hodge type and $\la \cdot,\cdot\ra_H$ is nondegenerate, then $\Alg_{\mathrm{deg}}\subset\Alg$ is an acyclic subcomplex and $\pi_\QQ\colon \Alg\onto \QQ$ is a quasi-isomorphism.
\item If $\Alg_{\mathrm{deg}}\subset\Alg$ is acyclic and $\QQ$ is of finite type, then $\Alg$ is of Hodge type.
\item Each Hodge decomposition $\Alg=\HH\oplus\im\dd\oplus C$ induces a Hodge decomposition
\[
	\QQ=\pi_\QQ(\HH)\oplus\im\dd_\QQ\oplus \pi_\QQ(C),
\]
and $\pi_\QQ\colon\Alg\onto\QQ$ respects these Hodge decompositions. 
\end{enumerate}
\end{lemma}

\subsection{Differential graded algebras with pairing}

A {\em differential graded algebra} ($\DGA$) $(\Alg,\dd,\wedge)$ is a cochain complex $(\Alg,\dd)$ equipped with an associative product $\wedge\colon \Alg\times \Alg \to \Alg$ of degree $0$ satisfying the Leibniz identity 
\[
	\dd(x\wedge y) = \dd x \wedge y + (-1)^{\deg x} x \wedge \dd y.
\]
A {\em commutative $\DGA$} ($\CDGA$) satisfies in addition the graded commutativity
\[
	x \wedge y = (-1)^{\deg x \deg y}y \wedge x.
\]

A nonzero element $1\in\Alg^0$ in a $\DGA$ $(\Alg,\dd,\wedge)$ is called a \emph{unit} if $1\wedge x = x \wedge 1 = x$ for all $x\in \Alg$.
A $\DGA$ which admits a unit (which is then uniquely determined) is called a \emph{unital $\DGA$}.
We say that a unital $\DGA$ $\Alg$ is {\em connected} if it is nonnegatively graded and $\Alg^0 = \R\cdot 1$.
Given $k\in\N$, we say that $\Alg$ is {\em $k$-connected} if it is connected and $\Alg^i = 0$ for all $i\in\{1,\dotsc,k\}$.
A 1-connected $\DGA$ is also called {\em simply connected}.

A {\em pairing} on a $\DGA$ $(\Alg,\dd,\wedge)$ is a pairing $\la \cdot,\cdot\ra\colon\Alg\times\Alg\to\R$ on the cochain complex $(\Alg,\dd)$ satisfying in addition
\begin{equation}\label{Eq:ProdCyclic}
	\la x \wedge y,z\ra = \la x, y\wedge z\ra.
\end{equation}
Note: using \eqref{Eq:Symmetry}, condition \eqref{Eq:ProdCyclic} is equivalent to the \emph{cyclicity condition} 
\begin{equation}\label{Eq:ProdCyclicOld}
	\la x\wedge y, z\ra  = (-1)^{\deg z(\deg x+\deg y)} \la z\wedge x, y\ra,
\end{equation}
which is considered in \cite{Cieliebak-Fukaya-Latschev}
and later in Subsection~\ref{subsec:cyclic} in the context of $\Ainfty$ algebras.
Also note that \eqref{Eq:Symmetry} is implied by \eqref{Eq:ProdCyclic} if $\Alg$ is commutative and by \eqref{Eq:ProdCyclicOld} if $\Alg$ is unital. 
Following~\cite{Cieliebak-Fukaya-Latschev}, we call a $\DGA$ with a perfect pairing $(\Alg,\dd,\wedge,\la\cdot,\cdot\ra)$ a {\em cyclic $\DGA$}.%
\footnote{In~\cite{Cieliebak-Fukaya-Latschev} the authors use $(-1)^{\deg x}\la x,y\ra$ instead of $\la x,y\ra$.}

{\bf Orientations. }
Following~\cite{Lambrechts-Stanley}, an \emph{orientation} of degree $n\in \N_0$ on a nonnegatively graded cochain complex $(\Alg,\dd)$ is a linear map $o\colon \Alg\to \R$ with $\deg o = -n$ such that $o\circ \dd = 0 $ and the induced map on cohomology $H(o)\colon H(\Alg)\to \R$ is nontrivial.
The triple $(\Alg,\dd,o)$ is then called an \emph{oriented cochain complex.}

If we view the cohomology as the trivial cochain complex $(H(\Alg),\dd=0)$, then for every orientation $\wt o$ on $H(\Alg)$ there is an orientation $o$ on $\Alg$ such that $\wt o = H(o)$, and such orientation is unique if and only if $\dd \Alg^n = 0$. 

An orientation $o$ on a $\DGA$ $(\Alg,\dd,\wedge)$ is defined as an orientation on the underlying cochain complex $(\Alg,\dd)$.
It induces a pairing $\la\cdot,\cdot\ra\colon \Alg\times\Alg\to\R$ of degree $n$ via
\begin{equation}\label{eq:pairing-orientation}
	\la x,y \ra \coloneqq o(x\wedge y).
\end{equation}
Conversely, a pairing $\la\cdot,\cdot\ra\colon\Alg\times\Alg\to\R$ of degree $n$ induces a chain map $o\colon \Alg\to \R$ with $\deg o = -n$ provided that $\Alg$ is unital (see \cite[Lemma 3.1]{Pavel-thesis}).

{\bf Small subalgebra. }
Consider a $\DGA$ of Hodge type $(\Alg,\dd,\wedge,\la\cdot,\cdot\ra)$, and let $P\colon\Alg\to\Alg$ be a special propagator with associated harmonic subspace $\HH_P\coloneqq \im\pi_P$.
Following~\cite{Fiorenza+}, we call the smallest dg-subalgebra $\SS\subset
\Alg$ such that 
\[
	\HH_P\subset \SS\quad\text{and}\quad P(\SS)\subset \SS
\]
the \emph{small subalgebra} of $\Alg$ associated to~$P$ and denote it by
\[
	\SS_P.
\]
\begin{lemma}[{\cite[Lemma 3.4]{Pavel-Hodge}, \cite[Prop.~6.1.15]{Pavel-thesis}}]\label{lem:small-subalgebra}
	Let $(\Alg,\dd,\wedge,\la\cdot,\cdot\ra)$ be a $\DGA$ of Hodge type, $P\colon\Alg\to\Alg$ a special propagator with associated harmonic subspace $\HH_P\subset\Alg$, and 
	$\SS_P$ the associated small subalgebra.
	Then:
\begin{enumerate}[(a)]
	\item The vector space $\SS_P$ is generated by iterated applications of $P$ and $\wedge$ to tuples of homogenous elements $h_1,\dotsc,h_i\in\HH_P$, $i\in\N$.
\item If $\Alg$ is unital and the cohomology $H(\Alg)$ is simply connected and of finite type, then so is~$\SS_P$. 
\item The restriction $P|_{\SS_P}\colon \SS_P\to\SS_P$ is a special propagator in the $\DGA$ with pairing $(\SS_P,\dd,\wedge,\la\cdot,\cdot\ra|_{\SS_P})$ and induces the Hodge decomposition $\SS_P = \HH_P\oplus \dd\SS_P\oplus P(\SS_P)$ which is relative to the Hodge decomposition of $\Alg$ associated to $P$.
In particular, the inclusion $\iota\colon\SS_P\into \Alg$ respects the Hodge decompositions and is a quasi-isomorphism.
\end{enumerate}
\end{lemma}

\begin{figure}
	\centering
	\begin{tikzpicture}
	\node[root] (R) at (0,0) {};
	\node[point,label={[above]:$\wedge$}] (RU) at ($(R)+(0,\leaflen)$) {};
	\node[point,label={[above]:$\wedge$}] (RUR) at ($(RU)+(90-\brancheangle:\edgelen)$) {};
	\node[leaf,label={[above]:$h_2$}] (RURL) at ($(RUR)+(90+\brancheangle:\leaflen)$) {};
	\node[leaf,label={[above]:$h_1$}] (RURR) at ($(RUR)+(90-\brancheangle:\leaflen)$) {};
	\node[leaf,label={[above]:$h_3$}] (RULL) at ($(RU)+(90+\brancheangle:\leaflen)$) {};
	\draw (R) -- (RU) node[midway,right] {P};
	\draw (RU) -- (RUR) node[midway,below] {$\Id$};
	\draw (RU) -- (RULL) node[midway,below left] {P};
	\draw (RUR) -- (RURL) node[midway,below left] {P};
	\draw (RUR) -- (RURR) node[midway,below right] {$\Id$};
	\end{tikzpicture}
	\[
		\ev_{T,L}(h_1,h_2,h_3)= P(P(h_3)\wedge P(h_2) \wedge h_1)
	\]
	\caption{Kontsevich--Soibelman-like evaluation of a labeled planar rooted binary tree.
	}
	\label{fig:tree}
\end{figure}
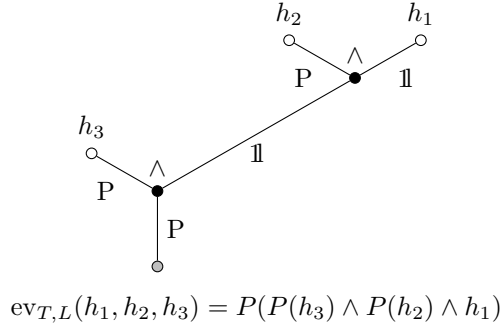

\begin{remark}\label{rem:ks-eval}
The following description of $\SS_P$ in terms of Kontsevich--Soibelman-like evaluations of trees from \cite{Pavel-thesis} is of interest in the context of this paper.
Let $\TT_{i}^{\rmbin}$ denote the set of isotopy classes of planar embeddings of rooted binary trees with $i\in\N$ leaves (for $i=1$ we include the trivial tree with only one edge).
There is a natural orientation of edges of $T\in \TT_i^{\rmbin}$ towards the root and a natural numbering of leaves of $T$ in the counterclockwise direction from the root.
A labeling $L$ of $T$ is an assignment of either~$P$ or~$\Id$ to each edge (interior and exterior) and of $\wedge$ to each interior vertex.
We denote the set of all labelings of $T$ by $L(T)$.
For $L\in L(T)$, we interpret the labeled tree $(T,L)$ as a composition rule for the operations $\Id$, $P$, $\wedge$ and obtain a linear map (see Figure~\ref{fig:tree})
\begin{equation} \label{eq:tree-eval}
\ev_{T,L}\colon \Alg^{\otimes i} \longrightarrow \Alg.
\end{equation}
Then we have
\[
	\SS_P = \sum_{i\in\N} \sum_{T\in \TT_{i}^{\rmbin}}\sum_{L\in L(T)} \ev_{T,L}(\HH^{\otimes i}). 
\]
\end{remark}

\subsection{Poincar\'e DGAs and differential Poincar\'e duality models}\label{ss:PDGA}

In this subsection, we will work in the category of nonnegatively graded unital $\CDGA$s with orientations of degree $n\in \N_0$.
Following~\cite{Lambrechts-Stanley} we define:
\begin{definition}

A \emph{differential Poincar\'e duality algebra ($\dPD$ algebra)} $(\Alg,\dd,\wedge,o)$ of degree $n\in \N_0$ is a finite dimensional nonnegatively graded unital $\CDGA$ equipped with an orientation~$o\colon\Alg\to\R$ of degree~$n$ such that the pairing $\la\cdot,\cdot\ra\colon\Alg\times\Alg\to\R$ corresponding to $o$ via~\eqref{eq:pairing-orientation} is nondegenerate. 
\end{definition}
Note: this is the same notion as that of a nonnegatively graded unital cyclic $\CDGA$.
Following~\cite{Fiorenza+} we define:
\begin{definition}
A \emph{Poincar\'e $\DGA$ ($\PDGA$)} $(\Alg,\dd,\wedge,o_H)$ of degree $n\in\N_0$ is a nonnegatively graded unital $\CDGA$ whose cohomology $H(\Alg)$ is equipped with an orientation $o_H\colon H(\Alg)\to\R$ making it into a \emph{Poincar\'e duality algebra}, i.e., a $\dPD$ algebra with trivial differential. 
An \emph{oriented $\PDGA$} $(\Alg,\dd,\wedge,o)$ is a $\PDGA$ $(\Alg,\dd,\wedge,o_H)$ equipped with an orientation~$o\colon\Alg\to\R$ such that $H(o)=o_H$.
\end{definition}

Let us emphasize that a $\dPD$ algebra involves a perfect pairing on chain level, whereas a $\PDGA$ involves a perfect pairing only on cohomology.

A {\em morphism of $\PDGA$s $(\Alg,\dd,\wedge,o_H)$ and $(\Alg^{\prime},\dd^{\prime},\wedge^{\prime},o_H^{\prime})$ ($\PDGA$ morphism)} is a morphism of unital $\DGA$s $f\colon (\Alg,\dd,\wedge) \to (\Alg^{\prime},
\dd^{\prime},\wedge^{\prime})$ such that the induced map on cohomology $H(f)\colon H(\Alg)\to H(\Alg^{\prime})$ satisfies
\begin{equation*}
	o_H^{\prime}\circ H(f)=o_H.
\end{equation*}
Recall that $f$ is called a {\em quasi-isomorphism} if in addition $H(f)$ is an isomorphism. 
A~{\em morphism of oriented $\PDGA$s} $(\Alg,\dd,\wedge,o)$ and $(\Alg^\prime,\dd^\prime,\wedge^\prime,o^\prime)$ is a $\PDGA$ morphism $f\colon(\Alg,\dd,\wedge,o_H)\to(\Alg^\prime,\dd^\prime,\wedge^\prime,o_H^\prime)$ such that
\begin{equation*}
o^{\prime} \circ f = o.
\end{equation*}
A morphism of $\dPD$ algebras is defined in the same way.
Let us emphasize that a morphism of $\dPD$ algebras is injective on chain level whereas a morphism of $\PDGA$s is injective only on cohomology.

Suppose that $\dd \Alg^{n} = 0$, $\dd^\prime \Alg^{\prime n} = 0$ so that orientations on cohomology $o_H, o_H^\prime$ are in one to one correspondence with orientations $o, o^\prime$ on chain level.
Then any $\PDGA$ morphism $f\colon (\Alg,\dd,\wedge,o_H)\to(\Alg^\prime,\dd^\prime,\wedge^\prime,o^\prime_H)$ is also a morphism of oriented $\PDGA$s $f\colon (\Alg,\dd,\wedge,o)\to(\Alg^\prime,\dd^\prime,\wedge^\prime,o^\prime)$.
In particular, a $\PDGA$ morphism of two $\dPD$ algebras is automatically a morphism of $\dPD$ algebras.

A {\em Hodge decomposition} of an oriented $\PDGA$ $(\Alg,\dd,\wedge,o)$ is defined as a Hodge decomposition of the corresponding cochain complex with pairing $(\Alg,\dd,\la\cdot,\cdot\ra)$.
We then say that $(\Alg,\dd,\wedge,o)$ is of {\em Hodge type} if $(\Alg,\dd,\la\cdot,\cdot\ra)$ is of Hodge type, and that a $\PDGA$ $(\Alg,\dd,\wedge,o_H)$ is of Hodge type if $o_H$ is induced by a chain level orientation $o\colon \Alg\to \R$ such that $(\Alg,\dd,\wedge,o)$ is of Hodge type.

{\bf Differential Poincar\'e duality models. }
We say that two $\PDGA$s $\Alg$ and $\Alg^{\prime}$ are \emph{weakly equivalent} if there exist $\PDGA$s $\Alg_1$, $\dotsc$, $\Alg_k$ for some $k\in \N$ and a zigzag of $\PDGA$ quasi-isomorphisms
\begin{equation*}
\begin{tikzcd}
   	& \Alg_1 \ar[dl] \ar[dr] & & \Alg_3 \cdots \ar[dl] & & \Alg_k \ar[dl] \ar[dr] \\
	\Alg & & \Alg_2 & & \cdots \Alg_{k-1} & & \Alg^{\prime}. 
\end{tikzcd}
\end{equation*}
A weak equivalence of oriented $\PDGA$s resp.~$\dPD$ algebras can be defined similarly by replacing the term ``$\PDGA$'' with the terms ``oriented $\PDGA$'' resp.~``$\dPD$ algebra''.

A \emph{differential Poincar\'e duality model} of a $\PDGA$ $\Alg$ is a $\dPD$ algebra $\MM$ which is weakly equivalent to $\Alg$ as a $\PDGA$. 
We have the following existence and uniqueness theorem. 
Part (a) essentially corresponds to~\cite[Theorem~1.1]{Lambrechts-Stanley}, and part (b) to~\cite[Theorem~7.1]{Lambrechts-Stanley} under the additional hypotheses $n\geq 7$, $H^3(\Alg)=H^3(\Alg^{\prime})=0$ and $\Alg^2={\Alg}^{\prime 2}=0$. 
In its present formulation, the theorem is proved in~\cite{Pavel-Hodge}.

\begin{theorem}[Existence and uniqueness of $\dPD$ models~\cite{Lambrechts-Stanley,Pavel-Hodge}]\label{thm:existence-uniqueness-PDmodel} 
\hfill\break
(a) A $\PDGA$ $\Alg$ whose cohomology $H(\Alg)$ is simply connected admits a simply connected differential Poincar\'e duality model $\MM$ in the form
\begin{equation*}
\begin{tikzcd}
   & \Alg_1\ar[dl] \ar[two heads, dr] & \\
   \Alg & & \QQ(\Alg_1)\eqqcolon\MM,
\end{tikzcd}
\end{equation*}
where the $\PDGA$ $\Alg_1$ is simply connected, of Hodge type, and of finite type, and $\QQ$ denotes the nondegenerate quotient. 

(b) Let $\Alg, \Alg^{\prime}$ be simply connected $\dPD$ algebras which are weakly equivalent as $\PDGA$s, and suppose that $H^2(\Alg) = H^2(\Alg^{\prime}) =0$.
Then there exists a simply connected $\dPD$ algebra $\Alg_1$ and quasi-isomorphisms of $\dPD$ algebras
\begin{equation}\label{eq:weak-equivalence-of-dpd}
\begin{tikzcd}	
	& \Alg_1 & \\
	\Alg \ar[hook,ur] & & \Alg^{\prime}. \ar[hook',ul]
\end{tikzcd}
\end{equation}
In particular, $\Alg$ and $\Alg^\prime$ are weakly equivalent as (simply connected) $\dPD$ algebras.
\end{theorem}

The $\PDGA$ $\Alg_1$ in Theorem~\ref{thm:existence-uniqueness-PDmodel}(a) is constructed as an extension of the Sullivan minimal model of $\Alg$.
In the case that $\Alg$ is of Hodge type, we have the following more explicit construction of a differential Poincar\'e duality model which will be crucial in the sequel.

\begin{prop}[{\cite[Section 5]{Pavel-Hodge}}]\label{prop:existence-PDmodel}
	Let $(\Alg,\dd,\wedge,o)$ be an oriented $\PDGA$ of Hodge type such that its cohomology $H(\Alg)$ is simply connected.
	Let $P\colon\Alg\to\Alg$ be a special propagator and $\SS_P$ the associated small subalgebra.
	Then the nondegenerate quotient $\QQ_P\coloneqq\QQ(\SS_P)$ is a differential Poincar\'e duality model of $\Alg$ via the canonical zigzag of $\PDGA$ quasi-isomorphisms
\begin{equation}\label{eq:canon-zigzag}
	\begin{tikzcd}
		& \SS_P \ar[hook',"\iota",swap]{dl} \ar[two heads,"\pi_{\QQ}"]{dr} & \\
		\Alg & & \QQ_P.
	\end{tikzcd}	
\end{equation}
Here both $\SS_P$ and $\QQ_P$ are simply connected, of finite type, and equipped with orientations and Hodge decompositions which are respected by $\iota, \pi_\QQ$.
\end{prop}

\begin{proof}
The inclusion $\iota\colon\SS_P\into\Alg$ and the quotient map $\pi_\QQ\colon\SS_P\onto\QQ_P$ are $\PDGA$ morphisms by construction; they are quasi-isomorphisms (and hence $\PDGA$ quasi-isomorphisms) by Lemma~\ref{lem:small-subalgebra}(b) and Lemma~\ref{lem:nondeg-quotient}(a), respectively. 
According to Lemma~\ref{lem:small-subalgebra}(b), a Hodge decomposition of $\Alg$ induces a Hodge decomposition of~$\SS_P$, which in turn induces a Hodge decomposition of $\QQ_P$ by Lemma~\ref{lem:nondeg-quotient}(c), such that~$\iota$ and~$\pi_\QQ$ respect the Hodge decompositions.  
By Lemma~\ref{lem:small-subalgebra}(c), $\SS_P$ is simply connected and of finite type, and the same is clearly true for its nondegenerate quotient $\QQ_P$.
The pairing on~$\QQ_P$ is nondegenerate by construction, hence~$\QQ_P$ is a $\dPD$ algebra.
\end{proof}

\begin{remark}
Examples in~\cite{Pavel-Hodge} show:

(a) The zigzag in Theorem~\ref{thm:existence-uniqueness-PDmodel}(a) cannot generally be replaced by a single $\PDGA$ quasi-isomorphism: there exists a $\PDGA$ $\Alg$ with simply connected cohomology which can not be connected to any of its Poincar\'e duality models $\MM$ by a single $\PDGA$ quasi-isomorphism (neither $\MM \to\Alg$ nor $\Alg \to \MM$). 

(b) Theorem~\ref{thm:existence-uniqueness-PDmodel}(b) fails without the hypothesis $H^2(\Alg_1) = H^2(\Alg_2) =0$: there exist two simply connected $\dPD$ algebras $\Alg$ and $\Alg^{\prime}$ such that there is no simply connected $\dPD$ algebra $\Alg_1$ which would fit in~\eqref{eq:weak-equivalence-of-dpd}.
Nevertheless, we expect that~$\Alg$ and~$\Alg^\prime$ are still weakly equivalent as $\dPD$ algebras (cf.~the Conjecture at the end of~\cite{Lambrechts-Stanley}). 

(c) The differential Poincar\'e duality model in Proposition~\ref{prop:existence-PDmodel} depends on the choice of $P$: there exists an oriented $\PDGA$ $\Alg$ of Hodge type with simply connected cohomology and two special propagators $P, P^\prime\colon\Alg\to\Alg$ such that~$\QQ_P$ and~$\QQ_{P^\prime}$ are not isomorphic as graded vector spaces.
\end{remark}

\section{\texorpdfstring{$\IBLinfty$ algebras}{IBL-infinity algebras}}\label{sec:IBLinfty}
\subsection{Basic definitions and properties}

In this subsection, we recall from~\cite{Cieliebak-Fukaya-Latschev} the basic notions of $\IBLinfty$ algebras (over $\R$).

{\bf $\IBLinfty$ algebras. }
Let $C=\bigoplus_{i\in\Z} C^i$ be a $\Z$-graded $\R$-vector space.
We denote the \emph{shifted grading} in $C[1]$ on a homogenous element $c$ by
\begin{equation}\label{eq:shifted-grading}
   |c| \coloneqq \deg c -1,
\end{equation}
and use it to induce a grading $|\cdot|$ on all Cartesian and tensor products of $C[1]$ and on their morphisms.
For every $k\in\N$, let
\begin{equation}\label{eq:extalg}
	E_k C \coloneqq C[1]^{\otimes k}/\sim
\end{equation}
be the quotient of the $k$-fold tensor product (over $\R$) of $C[1]$ under the action of the symmetric group~$S_k$ generated by the transposition $c_1 \otimes c_2 \mapsto (-1)^{|c_1||c_2|} c_2\otimes c_1$.
Following \cite{Cieliebak-Fukaya-Latschev}, an {\em $\IBLinfty$ structure of degree $d\in \Z$} on $C$ is a collection of linear maps 
\[
    \mathfrak p_{k,\ell,g} \colon E_k C \longrightarrow E_{\ell} C,\quad k\ge 1,\ \ell \ge 1,\ g\geq 0
\] 
of degrees
\[
    |\mathfrak p_{k,\ell,g}| = -2d(k+g-1)-1
\]
which for every $(k,\ell,g)$ satisfy the quadratic relation%
\footnote{Here, as well as in formula~\eqref{eq:ibl-mor-rel} below, the sums have only finitely many nonzero terms.}  
\begin{equation}\label{eq:ibl-relation}
	\sum_{h=1}^\infty \sum_{\substack{k_1+k_2-h = k \\ \ell_1 + \ell_2 -h = \ell \\ g_1 + g_2 + h-1 = g}} \fp_{k_2,\ell_2,g_2}\circ_{h}\fp_{k_1,\ell_1,g_1} = 0,
\end{equation}
where $\circ_h$ denotes the composition of $h$ outputs of $\fp_{k_1,\ell_1,g_1}$ with $h$ inputs of $\fp_{k_2,\ell_2,g_2}$ in all possible ways.
We call the tuple
\[
	\bigl(C,\fp\coloneqq\{\mathfrak p_{k,\ell,g}\}_{k,\ell \geq 1,g\geq 0}\bigr)
\]
an {\em $\IBLinfty$ algebra of degree~$d$}.

We think of~$\fp_{k,\ell,g}$ as being encoded by a connected compact oriented surface of signature $(k,\ell,g)$, meaning that it has $k$ incoming boundary components, $\ell$~outgoing boundary components, and genus~$g$, and interpret \eqref{eq:ibl-relation} as a sum over pairs of surfaces of signatures $(k_1,\ell_1,g_1)$, $(k_2,\ell_2,g_2)$ such that if we glue them at $h$ common boundaries, then we obtain a connected surface of signature $(k,\ell,g)$.
In fact, the left hand side of \eqref{eq:ibl-relation} is a special case of the \emph{connected composition}
\[
	(\fp_{k^-_1,\ell^-_1,g^-_1},\dotsc,\fp_{k^-_{r^-},\ell^-_{r^-},g^-_{r^-}})\circ_{\mathrm{conn}}(\fp_{k^+_1,\ell^+_1,g^+_1},\dotsc,\fp_{k^+_{r^+},\ell^+_{r^+},g^+_{r^+}})
\]
defined in \cite[Definition D.4.2]{Pavel-thesis} with fixed total genus~$g$ and restricted to $E_k C \to E_{\ell} C$.
We will denote this modification of $\circ_{\mathrm{conn}}$ by $\circ_{\mathrm{conn}}^{k,\ell,g}$.

An $\IBLinfty$ algebra is called a {$\dIBL$ algebra} if the only possibly nontrivial operations are $\fp_{1,1,0}, \fp_{2,1,0}, \fp_{1,2,0}$,  and an {\em $\IBL$ algebra} if the only possibly nontrivial operations are $\fp_{2,1,0}, \fp_{1,2,0}$.
The data of an $\IBL$ algebra is the same as the data of a graded involutive Lie bialgebra up to signs, which can be explained by choosing a natural convention for degree shifts (see \cite[Proposition 3.2.5]{Pavel-thesis}).

The relation (1,1,0) of \eqref{eq:ibl-relation} implies that $\fp_{1,1,0}\colon C\to C$ squares to zero and the relations (2,1,0) resp.~(1,2,0) that the operations $\fp_{2,1,0}$ resp.~$\fp_{1,2,0}$ descend to homology $H\coloneqq H(C,\fp_{1,1,0})$ where they constitute an $\IBL$ algebra. 
The Jacobi identity on $H$ follows from (3,1,0), the coJacobi identity from (1,3,0), Drinfeld compatibility from (2,2,0), and the involutivity from (1,1,1).

{\bf $\IBLinfty$ morphisms. }
Let $(C,\fp=\{\fp_{k,\ell,g}\})$ and $(D,\fq=\{\fq_{k,\ell,g}\})$ be  two $\IBLinfty$ algebras of the same degree $d$. 
An {\em $\IBLinfty$ morphism} $\ff\coloneqq\{\ff_{k,\ell,g}\}\colon C\to D$ is a collection of linear maps  
\[
    \ff_{k,\ell,g} \colon E_k C \longrightarrow E_{\ell} D,\quad k,\ell \ge 1, g\geq 0
\] 
of degrees 
\[
    |\ff_{k,\ell,g}| = -2d(k+g-1)
\] 
which for every $(k,\ell,g)$ satisfy
\begin{equation}\label{eq:ibl-mor-rel}
\begin{aligned}
	&\sum \frac{1}{r^+!}
	\fq_{k^-,\ell^-,g^-}\circ_{\mathrm{conn}}^{k,\ell,g}(\ff_{k_1^+,\ell^+_1,g^+_1},\dotsc,\ff_{k^+_{r^+},\ell^+_{r^+},g^+_{r^+}}) \\ 
	&= \sum \frac{1}{r^-!}
	(\ff_{k^-_1,\ell^-_1,g^-_1},\dotsc,\ff_{k^-_{r^-},\ell^-_{r^-},g^-_{r^-}})\circ_{\mathrm{conn}}^{k,\ell,g}\fp_{k^+,\ell^+,g^+}.
\end{aligned}
\end{equation}
Here and in all composition formulas below $\sum$ means a summation over all free indices in $\N$ and genus in $\N_0$.
The relation $(1,1,0)$ of~\eqref{eq:ibl-mor-rel} implies that $\ff_{1,1,0}\colon (C^+,\fp^+)\to(C^-,\fp^-)$ is a chain map, and the relations $(2,1,0)$ resp.~$(1,2,0)$ imply that the induced map on homology $H(\ff_{1,1,0})\colon H(C^+)\to H(C^-)$ preserves the bracket resp.~cobracket.

An \emph{$\IBLinfty$ quasi-isomorphism} is an $\IBLinfty$ morphism $\ff=\{\ff_{k,\ell,g}\}\colon (C^+,\fp^+)\to (C^-,\fp^-)$ such that $\ff_{1,1,0}\colon (C^+,\fp_{1,1,0}^+)\to (C^-,\fp_{1,1,0}^-)$ is a quasi-isomorphism.
In~\cite{Cieliebak-Fukaya-Latschev}, the notion of a {\em homotopy of $\IBLinfty$ morphisms} is defined and the following facts are proved using obstruction theory:
\renewcommand{\descriptionlabel}[1]{\hspace{\labelsep}(\textit{#1})}
\begin{description}
\item[Homotopy inverse] Every $\IBLinfty$ quasi-isomorphism is an $\IBLinfty$ homotopy equivalence, i.e., it has an $\IBLinfty$ homotopy inverse.
\item[Homotopy transfer] Every $\IBLinfty$ algebra $(C,\fp)$ induces an $\IBLinfty$ structure $\fq=\{\fq_{k,\ell,g}\}$ on its homology $H\coloneqq H(C,\fp_{1,1,0})$ together with an $\IBLinfty$ homotopy equivalence $\ff\colon(C,\fp)\to (H,\fq)$.
\end{description}

{\bf Twisting with a Maurer--Cartan element. }
The deformation theory of $\IBLinfty$ algebras is formulated in terms of \emph{filtered $\IBLinfty$ algebras}.\footnote{
All our filtered $\IBLinfty$ algebras will be {\em strict} in the sense of~\cite{Cieliebak-Fukaya-Latschev}.} 
Similarly to $\IBLinfty$ algebras they are governed by the relations \eqref{eq:ibl-relation}, \eqref{eq:ibl-mor-rel}, with the difference that $C$ is now a filtered graded vector space and the maps $\fp_{k,\ell,g}\colon\wh{E}_kC\to\wh{E}_{\ell}C$, $\ff_{k,\ell,g}\colon\wh{E}_kC\to\wh{E}_{\ell}D$ are defined on suitable completions.
Besides the degree $d\in\Z$, the notion of a filtered $\IBLinfty$ algebra also involves a filtration degree $\gamma\in\N_0$.
All the $\IBLinfty$ algebras considered in this paper are based on the dual cyclic bar complex, which has a natural filtration by word-length with $\gamma=2$.
We refer to~\cite[Section~8]{Cieliebak-Fukaya-Latschev} and \cite[Appendix~D]{Pavel-thesis} for more details.

A {\em Maurer--Cartan element} in a filtered $\IBLinfty$ algebra $(C,\fp=\{\fp_{k,\ell,g}\})$ of degree $d$ and filtration degree $\gamma$ is a collection $\fm=\{\fm_{\ell,g}\}$ of elements 
\[
    \fm_{\ell,g}\in \wh{E}_\ell C,\quad \ell\geq 1,g\geq 0
\]
of degrees 
\[
   |\fm_{\ell,g}| = -2d(g-1)
\] 
and filtration degrees
\[
	\|\fm_{\ell,g}\| \ge \gamma (2-2g-\ell), 
\]
where a strict inequality is assumed for $(1,0)$, $(2,0)$, which satisfy the relations
\begin{equation*}
	\sum \frac{1}{r^+!}\fp_{k^-,\ell^-,g^-}\circ_{\mathrm{conn}}^{k,\ell,g}(\fm_{\ell^+_1,g^+_1},\dotsc,\fm_{\ell^+_{r^+},g^+_{r^+}})= 0
\end{equation*}
for all $(\ell,g)\in\N\times\N_0$.
Here we view $\fm_{\ell,g}$ as a map $\R\to\wh E_{\ell} C$, $1\mapsto \fm_{\ell,g}$ while applying the composition.
Given a Maurer--Cartan element $\fm=\{\fm_{\ell,g}\}$ in a filtered $\IBLinfty$ algebra $(C,\fp=\{\fp_{k,\ell,g}\})$, the \emph{twist of $\fp$ with $\fm$} is the collection $\fp^\fm=\{\fp^\fm_{k,\ell,g}\}$ of operations $\fp_{k,\ell,g}^\fm\colon\wh{E}_k C\to\wh{E}_{\ell} C$ for $k,\ell\ge 1$, $g\ge 0$ defined by
\begin{equation}\label{eq:twisted-operations}
	\fp_{k,\ell,g}^\fm \coloneqq \sum \frac{1}{r^+!}
	\fp_{k^-,\ell^-,g^-}\circ_{\mathrm{conn}}^{k,\ell,g}(\fm_{\ell^+_1,g^+_1},\dotsc,\fm_{\ell^+_{r^+},g^+_{r^+}}).
\end{equation}
Given another filtered $\IBLinfty$ algebra $(D,\fq=\{\fq_{k,\ell,g}\})$ and a morphism $\ff\colon (C,\fp)\to (D,\fq)$, the \emph{pushforward of $\fm$ along $\ff$} is the collection $\ff_*\fm=\{(\ff_*\fm)_{\ell,g}\}$ of elements $(\ff_*\fm)_{\ell,g}\in\wh{E}_\ell D$ for $\ell\ge 1$, $g\ge 0$ defined by
\begin{equation}\label{eq:pushforward-mc}
	(\ff_*\fm)_{\ell,g}= 
	 \sum \frac{1}{r^+!}\frac{1}{r^-!}
	 (\ff_{k_1^-,\ell_1^-,g_1^-},\dotsc,
	 \ff_{k_{r^-}^-,\ell_{r^-}^-,g_{r^-}^-})\circ_{\mathrm{conn}}^{0,\ell,g}
     (\fm_{\ell_1^+,g_1^+},\dotsc, 
     \fm_{\ell_{r^+}^+,g_{r^+}^+}).	
\end{equation}
Finally, the \emph{twist of $\ff$ with $\fm$} is the collection $\ff^\fm = \{\ff^\fm_{k,\ell,g}\}$ of maps $\ff^\fm_{k,\ell,g}\colon\wh{E}_k C \to \wh{E}_\ell D$ for $k,\ell\geq 1$, $g\ge 0$ defined by
\begin{equation*}
	\ff^\fm_{k,\ell,g} =
	\sum
	 \frac{1}{r^+!}\frac{1}{r^-!}
	 (\ff_{k_1^-,\ell_1^-,g_1^-},\dotsc,	 
	 \ff_{k_{r^-}^-,\ell_{r^-}^-,g_{r^-}^-})\circ_{\mathrm{conn}}^{k,\ell,g}
     (\fm_{\ell_1^+,g_1^+},\dotsc, 
     \fm_{\ell_{r^+}^+,g_{r^+}^+}).	
\end{equation*}

\begin{prop}[{\cite[Propositions~9.3, 9.6]{Cieliebak-Fukaya-Latschev}}]\label{prop:MC}
In the situation above, we have:
\begin{enumerate}
\item The twist $\fp^\fm$ defines an filtered $\IBLinfty$ structure on $C$.
\item The pushforward $\ff_*\fm$ defines a Maurer--Cartan element in $(D,\fq)$.
\item The twist $\ff^\fm$ defines a filtered $\IBLinfty$ morphism $\ff^\fm\colon (C,\fp^\fm) \to (D,\fq^{\ff_*\fm})$.
\item If $\ff$ is a homotopy equivalence of filtered $\IBLinfty$ algebras, then so is $\ff^\fm$.  
\end{enumerate}
\end{prop}
In the situation above, we will refer to $\fp^\fm$ as the \emph{twisted $\IBLinfty$ structure}, to $\ff^\fm$ as the \emph{twisted $\IBLinfty$ morphism}, and to $\ff_*\fm$ as the \emph{pushforward Maurer--Cartan element}.

\subsection{The \texorpdfstring{${\rm dIBL}$}{dIBL} structure associated to a cyclic cochain complex}\label{ss:dIBL-cyc-cochain}

In this and the following subsections, we recall the algebraic constructions from~\cite{Cieliebak-Fukaya-Latschev} which, taken all together, associate to a cyclic $\DGA$~$\Alg$ of degree $n$ a filtered $\IBLinfty$ structure of degree $n-3$ on the degree shifted dual cyclic bar complex of the cohomology $(\dcbc H(\Alg))[2-n]$, whose homology is isomorphic to the cyclic cohomology of~$\Alg$ and which is defined naturally up to $\IBLinfty$ homotopy equivalence.
The trickiest part of these constructions are the signs, which we will not spell out but refer to~\cite{Cieliebak-Fukaya-Latschev}.

{\bf The dual cyclic bar complex and the $\dIBL$ functor. }
Let $\Alg$ be a $\Z$-graded $\R$-vector space,
and let $|\cdot|$ be the shifted grading on $\Alg[1]$ as in~\eqref{eq:shifted-grading}.
For every $k\in \N$, we define
\begin{equation*}
	B_k^{\text{\rm cyc}}\Alg \coloneqq A[1]^{\otimes k}/\sim 
\end{equation*}
as the quotient of the $k$-fold tensor product of $\Alg$ over $\R$ under the restriction of the action of $S_k$ from \eqref{eq:extalg} to the cyclic permutations $\Z_k\subset S_k$.
We will write the equivalence class of $a_1\otimes\dots\otimes a_k$ in $B_k^{\text{\rm cyc}}\Alg$ as $a_1\cdots a_k$. 
Following~\cite{Cieliebak-Fukaya-Latschev}, we define the \emph{cyclic bar complex}%
\footnote{Note that our cyclic bar complex does not include a $k=0$ term.}
\[
	\cbc\Alg\coloneqq\bigoplus_{k\in\N} \cbc_k\Alg.
\] 
We define the {\em dual cyclic bar complex}
\begin{equation*}
	\begin{aligned}
	\dcbc_k \Alg& \coloneqq (B_k^{\text{\rm cyc}} \Alg)^*, \\
	\dcbc \Alg & \coloneqq (\cbc\Alg)^* = \prod_{k\in\N} \dcbc_k\Alg,
	\end{aligned}
\end{equation*}
where the upper $*$ denotes the {\em graded dual} with respect to the degrees in $A[1]$.
This definition differs from the one in~\cite{Cieliebak-Fukaya-Latschev} where $\dcbc \Alg$ was defined as a direct sum. Our $\dcbc \Alg$ is already {\em complete} with respect to the filtration of $\prod_{k\in\N} \dcbc_k\Alg$ by the tensor degree $k$, which avoids unnecessary completions later. 

We consider the \emph{reversed shifted grading} on $\dcbc\Alg$, i.e., for homogenous $\varphi\in \dcbc\Alg$ and $x\in \cbc\Alg$, we have
\[
	\varphi(x)\neq 0\quad\Longrightarrow\quad |x| = |\varphi|.
\]

\begin{prop}[{\cite[Proposition~10.4]{Cieliebak-Fukaya-Latschev}}]\label{prop:structureexists}
Let $(\Alg,\dd,\la\cdot,\cdot\ra)$ be a cyclic cochain complex of degree $n\in\Z$.
The degree shifted dual cyclic bar complex $C\coloneqq(\dcbc\Alg)[2-n]$ carries a canonical filtered $\dIBL$ structure $\{\fp_{1,1,0}=\dd^*,\fp_{2,1,0},\fp_{1,2,0}\}$ of degree $n-3$.
\end{prop}

\begin{proof}[Sketch of proof]
Let us describe the operations.
The coboundary operator~$\dd$ extends as a graded derivation to $B^{\rm cyc}\Alg$, still denoted by $\dd$, and $\fp_{1,1,0}\coloneqq \dd^*$ is just its dual. 
To define the other two operations $\fp_{2,1,0}$, $\fp_{1,2,0}$, we pick a homogeneous basis $(e_i)$ of~$\Alg$.
Let $(e^i)$ be the dual basis of $\Alg$ with respect to the pairing $\la\cdot,\cdot\ra\colon\Alg\times\Alg\to\R$, i.e., we require
\[
    \la e_i,e^j\ra = \delta_i^j. 
\]
We set 
\[
    g_{ij} \coloneqq \la e_i,e_j\ra,\quad g^{ij} \coloneqq \la e^i,e^j\ra. 
\]
Let $\varphi \in B_{k_1+1}^{\text{\rm cyc} *}\Alg$, $\psi \in B_{k_2+1}^{\text{\rm cyc} *}\Alg$,  $k_1,k_2 \ge 0$, $k_1+k_2\ge 1$.
We define $\fp_{2,1,0}(\phi,\psi)\in B_{k_1+k_2}^{\text{\rm cyc} *}\Alg$ on elements $x_i\in \Alg$ by
\begin{align}\label{eq:mu}
    & \fp_{2,1,0}(\varphi,\psi)(x_1,\dots,x_{k_1+k_2}) \cr
    &= \sum_{a,b}\sum_{c=1}^{k_1+k_2} \pm g^{ab} \varphi(e_a,x_c,\dots,x_{c+k_1-1})\psi(e_b,x_{c+k_1}, \dots,x_{c-1}).
\end{align}
See~\cite{Cieliebak-Fukaya-Latschev} for the appropriate signs. 
Next, let $\varphi \in \dcbc_k \Alg$, $k\ge 4$.
We define
\[
    \fp_{1,2,0}(\varphi) \in \bigoplus_{k_1+k_2= k-2} \dcbc_{k_1} \Alg \otimes \dcbc_{k_2} \Alg
\]
on elements $x_i,y_j\in \Alg$ by
\begin{align}\label{eq:delta}
    & \fp_{1,2,0}(\varphi)(x_1\dots x_{k_1}\otimes y_1,\dots,y_{k_2}) \cr
    &= \sum_{a,b}\sum_{c=1}^{k_1}\sum_{c^{\prime}=1}^{k_2} \pm g^{ab}\varphi(e_a,x_c,\dots,x_{c-1},e_b,y_{c^{\prime}},\dots,y_{c^{\prime}-1}),
\end{align}
again with suitable signs.
It is easy to see that operations $\fp_{2,1,0}$, $\fp_{1,2,0}$ do not depend on the choice of the basis $(e_i)$.
It is straightforward, though tedious, to verify that together with $\fp_{1,1,0}$ they satisfy the relations of a $\dIBL$ algebra.
\end{proof}

Following~\cite{Pavel-thesis}, we will denote the $\dIBL$ algebra from Proposition~\ref{prop:structureexists} by
\[
	\dIBL(\Alg)\coloneqq \bigl((\dcbc \Alg)[2-n],\fp\coloneqq\{\fp_{1,1,0},\fp_{2,1,0},\fp_{1,2,0}\}\bigr)
\]
This reflects the fact that it is part of a natural \emph{covariant} functor from the category of cyclic cochain complexes to the category of $\dIBL$ algebras%
\footnote{A morphism of $\dIBL$ algebras is an $\IBLinfty$ morphism with $\ff_{k,\ell,g}=0$ for $(k,\ell,g)\neq (1,1,0)$.} defined on morphisms as follows.
A morphism of cyclic cochain complexes $f\colon B\to \Alg$ is automatically injective (because it preserves the nondegenerate pairings), hence it can be identified with an inclusion $\iota\colon B\into \Alg$.
One can show that the pullback $\pi^{*}_B\colon\dcbc B\to\dcbc\Alg$ along the unique symmetric projection $\pi_B\colon \Alg \onto B$ induces a $\dIBL$ morphism $\dIBL(f)\coloneqq\pi^{*}_B\colon\dIBL(B)\to\dIBL(\Alg)$.

\subsection{Weak functoriality of the dIBL construction}
\label{subsec:weak}

We continue the discussion of functoriality of the $\dIBL$ construction from the previous subsection and ask now about its \emph{contravariant} properties.
Given an inclusion of cyclic cochain complexes $\iota\colon B\into\Alg$, the pullback $\iota^*$ may not preserve the bracket and cobracket and hence may not define a morphism of $\dIBL$ algebras $\dIBL(\Alg)\to\dIBL(B)$.
The following theorem extends $\iota^*$ to an $\IBLinfty$ homotopy equivalence provided that $\iota$ is a quasi-isomorphism.

\begin{thm}[{\cite[Theorem~11.3]{Cieliebak-Fukaya-Latschev}}]\label{thm:homotopyequiv}
	Let $(\Alg,\dd,\la\cdot,\cdot\ra)$ be a cyclic cochain complex and $B\subset \Alg$ a quasi-isomorphic cyclic subcomplex. 
	Let $P\colon\Alg\to\Alg$ be a propagator with respect to the unique symmetric projection $\pi_B\colon\Alg\onto B$. 
	Then there is an $\IBL_{\infty}$ homotopy equivalence 
	\[
		{\ff^P} = \{\ff^P_{k,\ell,g}\} \colon \dIBL(\Alg) \longrightarrow \dIBL(B)
	\]
	which extends the pullback $\ff^P_{1,1,0}=\iota^*\colon \dcbc \Alg \to \dcbc B$ such that the value $\ff^P_{k,\ell,g}(\varphi)(\alpha)\in\R$ of $\ff^P_{k,\ell,g}$ on the tensor products
	\begin{equation*}
	\begin{aligned}
		\varphi&\coloneqq\varphi^1\otimes \dotsb\otimes\varphi^k,\quad\text{where }\varphi^i\in \dcbc\Alg\text{, and}\\
		\alpha&\coloneqq\alpha^1_1\cdots \alpha^1_{s_1}\otimes\dots\otimes\alpha^\ell_1\cdots \alpha^\ell_{s_{\ell}},\quad\text{where }s_i\in \N\text{ and }\alpha^b_j\in B,
	\end{aligned}
	\end{equation*}
	can be written as a sum over isomorphism classes of ribbon graphs $\Gamma$ with $k$ interior vertices, $\ell$ boundary components, and genus $g$, of contributions naturally associated to~$\Gamma$ by labeling the interior vertices with $\varphi^1,\dotsc,\varphi^k$, the interior edges with~$P$, and the exterior vertices on the $i$-th boundary component with $\alpha_{1}^i,\dotsc,\alpha_{s_i}^i$ (see Figure~\ref{fig:htpy}).	
\end{thm}

We refer to ${\ff^P}$ from the theorem above as to the \emph{natural $\IBLinfty$ homotopy equivalence $\dIBL(\Alg)\to\dIBL(B)$} associated to $P$.

The construction of ${\ff^P}$ extends similar constructions for $\Alg_{\infty}$ or $L_{\infty}$ structures based on planar rooted trees (see \cite{Kontsevich-Soibelman}, \cite[Subsection 5.4.2]{FOOO-I}) by considering general ribbon graphs.
Before sketching the construction, we summarize our conventions on ribbon graphs.

A~{\em ribbon graph} $\Gamma$ is a finite connected graph
with a cyclic ordering of the adjacent edges at every vertex.
We denote by $d(v)$ the {\em degree} of a vertex $v$, i.e., the number of edges adjacent to $v$.
We suppose that $\Gamma$ is equipped with a decomposition
\[
    C^0(\Gamma) = C^0_{\text{\rm int}}(\Gamma) \sqcup C^0_{\text{\rm ext}}(\Gamma) 
\]
of the set of vertices into {\em interior} and {\em exterior} vertices, where an exterior vertex has degree 1 and an interior vertex has degree at least $1$.
This induces a decomposition
\[
    C^1(\Gamma) = C^1_{\text{\rm int}}(\Gamma) \sqcup C^1_{\text{\rm ext}}(\Gamma) 
\]
of the set of edges into {\em interior} and {\em exterior} edges, where an edge is called exterior if and only if it contains an exterior vertex.
We denote by~$\Sigma_\Gamma$ the compact oriented surface with boundary obtained by thickening $\Gamma$ such that $\Gamma\cap\p\Sigma_\Gamma=C^0_\ext(\Gamma)$.
We require that each boundary component has precisely one \emph{marked} exterior vertex.

The {\em signature} of $\Gamma$ is the triple $(k,\ell,g)$, where $k=\#C^0_\inn(\Gamma)$, $\ell$ is the number of boundary components of $\Sigma_\Gamma$, and $g$ is the genus of $\Sigma_\Gamma$.
Given $k,\ell\geq 1$ and $g\geq 0$, we denote by $RG_{k,\ell,g}$ the set of isomorphism classes of ribbon graphs of signature $(k,\ell,g)$ with the additional data specified above.

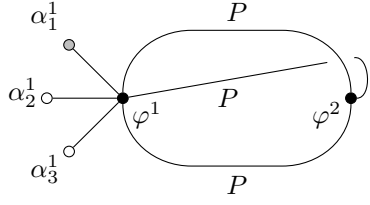
\begin{figure}
	\centering
	\begin{tikzpicture}
		\node[point,label={[below right]:$\varphi^1$}] (A) at (0,0) {};
		\node[point,label={[below left]:$\varphi^2$}] (B) at ($(A)+(\edgelen,0)$) {};
		\node[root,label={[above left]:$\alpha_1^1$}] (CI) at ($(A)+(135:\leaflen)$) {};
		\node[leaf,label={[left]:$\alpha_2^1$}] (CII) at ($(A)+(180:\leaflen)$) {};
		\node[leaf,label={[below left]:$\alpha_3^1$}] (CIII) at ($(A)+(-135:\leaflen)$) {};
		\coordinate (AI) at ($(A)+(.3*\edgelen,.3*\edgelen)$);
		\coordinate (AII) at ($(A)+(.7*\edgelen,.3*\edgelen)$);
		\coordinate (BI) at ($(A)+(.3*\edgelen,-.3*\edgelen)$);
		\coordinate (BII) at ($(A)+(.7*\edgelen,-.3*\edgelen)$);
		\coordinate (TI) at ($(A)+(10:.91*\edgelen)$);
		\coordinate (TII) at ($(A)+(10:1.03*\edgelen)$);	
		\draw (A) to[out=90,in=180] (AI) -- (AII) node[midway,above] {$P$} to[out=0,in=90] (B); 
		\draw (A) to[out=-90,in=180] (BI) -- (BII) node[midway,below] {$P$} to[out=0,in=-90] (B); 
		\draw (A) -- (TI) node[midway,below] {$P$};
		\draw (TII) to[out=10,in=0] (B);
		\draw (A) -- (CI);
		\draw (A) -- (CII);
		\draw (A) -- (CIII);
	\end{tikzpicture}
	\caption{A labeled ribbon graph contributing to 
	$\ff^P_{2,1,1}(\varphi^1\otimes\varphi^2)(\alpha^1_1\alpha_2^1\alpha_3^1)$ immersed in the plane so that the cyclic ordering at interior vertices agrees with the counterclockwise orientation.}
	\label{fig:htpy}
\end{figure}

\begin{proof}[Sketch of the construction of ${\ff^P}$ from Theorem~\ref{thm:homotopyequiv}]
Given $\Gamma\in RG_{k,\ell,g}$, we will define the contribution $\ff^P_\Gamma(\varphi)(\alpha)\in\R$ of $\Gamma$ to ${\ff^P}$ such that we can write
\[
	\ff^P_{k,\ell,g}(\varphi)(\alpha) = \sum_{\Gamma\in RG_{k,\ell,g}} (-1)^{r_\Gamma} C_\Gamma \ff^P_\Gamma(\varphi)(\alpha),
\]
where $C_\Gamma\in\Q^+$ is a combinatorial coefficient and $r_\Gamma\in\Z$ a sign exponent.
The values of~$C_\Gamma$ and~$r_\Gamma$ can be deduced from \cite{Cieliebak-Fukaya-Latschev} and we will not discuss them further.
By an {\em ordering $O$ on $\Gamma$} we mean an ordering of the interior vertices and an ordering of the boundary components of~$\Sigma_\Gamma$.
Having $O$ allows us to label the interior vertices by $\varphi^1,\dotsc,\varphi^k$ and the exterior vertices on the $b$-th boundary component by $\alpha_1^b,\dotsc,\alpha_{s_b}^b$,
where $j=1$ corresponds to the marked exterior vertex and $j\mapsto j+1$ goes in the positive direction of the boundary.
Here we require implicitly that $s_b$ equals the number of exterior vertices at the $b$-th boundary component for every $b\in\{1,\dotsc,\ell\}$, and define the contribution of the pair $(\Gamma,O)$ to be $0$ otherwise.

Using the notation of~\cite{Cieliebak-Fukaya-Latschev}, we define a {\em flag} of $\Gamma$ as a pair $(v,t)$ consisting of a vertex~$v$ and an edge~$t$ such that $v \in t$
(where an edge starting and ending at the same vertex gives rise to two flags).
We have the decomposition
\[
	\Flag(\Gamma) = \Flag_\rmint(\Gamma) \sqcup \Flag_\rmext(\Gamma)
\]
of the set of flags into \emph{interior} and \emph{exterior} flags, where a flag $(v,t)$ is called exterior if and only if $v$ is an exterior vertex.
As in the proof of Proposition~\ref{prop:structureexists}, we pick a basis $(e_i)_{i\in I}$ of $\Alg$, denote by $(e^i)_{i\in I}$ the dual basis of $\Alg$ with respect to $\la\cdot,\cdot\ra$, and set
\[
    P_{ij} \coloneqq \la Pe_i,e_j\ra,\quad P^{ij} \coloneqq \la Pe^i,e^j\ra. 
\]
Consider the set of maps $\mathfrak I \coloneqq {\rm Map}({\rm Flag}_{\rm int}(\Gamma),I)$, where $I$ is the index set of the basis.
Given $\mathfrak i \in \mathfrak I$ and $(v,t)\in\Flag(\Gamma)$, we define
\[
	\fe_{{\mathfrak i}(v,t)} \coloneqq \begin{cases}
					e_{\mathfrak i(v,t)} & \text{if }(v,t)\in{\rm Flag}_{\rm int}(\Gamma), \\
					\alpha_i^b	     & \text{if }(v,t)\in{\rm Flag}_{\rm ext}(\Gamma)\text{ and }v\text{ is labeled by }\alpha_i^b.
				 \end{cases}
\]
Notice that ${\mathfrak i}(v,t)$ in the subscript of $\fe$ for $(v,t)\in\Flag_\rmext(\Gamma)$ is not defined by itself.
In order to write down a linear expression, we make the following additional choices:
for every $v\in C^0_\rmint(\Gamma)$, we choose a bijection of the set of adjacent flags of $v$ and the set $\{(v,1),\dotsc,(v,d(v))\}$ compatible with the cyclic ordering,
and for every $t\in C^1_\rmint(\Gamma)$, we choose a bijection of the set of flags containing $t$ and the set $\{(1,t),(2,t)\}$ (equivalent to the choice of an orientation of the edge $t$).
We can then write down the natural expression
\begin{equation}\label{eq:contribution-of-graph}
	\sum_{\mathfrak i \in \mathfrak I}\pm \hspace{-10pt} \prod_{t\in C^1_\inn(\Gamma)} \hspace{-10pt}P^{\mathfrak i(1,t),\mathfrak i(2,t)}\hspace{-10pt} \prod_{v \in C^0_{\text{int}}(\Gamma)}\hspace{-10pt} \varphi^v(\fe_{{\mathfrak i}(v,1)}, \cdots, \fe_{\mathfrak i(v,d(v))}),
\end{equation}
where the sign is obtained from a natural convention introduced in \cite{Cieliebak-Fukaya-Latschev} such that the whole expression does not depend on the additional choices of the bijections provided that $P$ and $\varphi$ have the required symmetries.
The contribution $\ff^P_\Gamma(\varphi)(\alpha)$ is then defined as a sum of the contributions \eqref{eq:contribution-of-graph} of the pair $(\Gamma,O)$ over all orderings~$O$.
\end{proof}

\subsection{The twisted \texorpdfstring{${\rm dIBL}$}{dIBL} structure associated to a cyclic DGA}

Consider a nonnegatively graded cyclic $\DGA$ $(\Alg,\dd,\wedge,\la\cdot,\cdot\ra)$ of degree $n\in\N_0$ and the filtered $\dIBL$ algebra $\dIBL(\Alg)=((\dcbc \Alg)[2-n],\fp=\{\fp_{1,1,0},\fp_{2,1,0},\fp_{1,2,0}\})$ of degree $n-3$ associated to the underlying cyclic cochain complex $(\Alg,\dd,\la\cdot,\cdot\ra)$.
Following~\cite{Cieliebak-Fukaya-Latschev}, we will use the multiplication~$\m_2(x,y)\coloneqq(-1)^{\deg x} x\wedge y$ to obtain a canonical Maurer--Cartan element in $\dIBL(\Alg)$. 
We define the {\em triple intersection product} $\m_2^+\in\dcbc_3\Alg$ by%
\begin{equation}\label{eq:triple-intersection}
   \m_2^+(x_0x_1x_2) \coloneqq (-1)^{n + \deg x_1}\la x_0\wedge x_1,x_2\ra.
\end{equation}

\begin{prop}[{\cite[Proposition~12.5]{Cieliebak-Fukaya-Latschev}}]\label{propIBLI2}
Let $(\Alg,\dd,\wedge,\la\cdot,\cdot\ra)$ be a nonnegatively graded cyclic $\DGA$ of degree $n\in\N_0$.
Then
\begin{equation*}
	\fm_{\ell,g}\coloneqq 
	\begin{cases}
		\m_2^+ & \text{if }(\ell,g)=(1,0),\\
		0 & \text{otherwise,}
	\end{cases}
\end{equation*}
defines a Maurer--Cartan element $\fm=\{\fm_{\ell,g}\}$ in $\dIBL(\Alg)$.
The twisted operations
\[
   \fp^\fm=\bigl\{\fp_{1,1,0}^\fm=\dd^*+\fp_{2,1,0}(\fm_{1,0},\cdot),\,\fp_{2,1,0},\,\fp_{1,2,0}\bigr\}
\]
define a $\dIBL$ structure on $(\dcbc \Alg)[2-n]$.
\end{prop}

We call $\fm$ in Proposition~\ref{propIBLI2} the \emph{canonical Maurer Cartan element} in $\dIBL(\Alg)$ and denote it by
\[
	\fm^\can_\Alg.
\]
Generally, for a Maurer--Cartan element $\fm$ in $\dIBL(\Alg)$, we denote the corresponding twisted $\IBLinfty$ algebra by
\[
	\dIBL^\fm(\Alg) \coloneqq \bigl((\dcbc\Alg)[2-n],\fp^\fm \bigr).
\]
In the case of $\fm=\fm^\can_\Alg$, we call it the \emph{canonical twisted $\dIBL$ algebra}.
It is straightforward to show that $\fp_{2,1,0}((\fm_\Alg^\can)_{1,0},\cdot)=\bb^*$ is the dual of the \emph{Hochschild differential} $\bb\colon\cbc\Alg\to\cbc\Alg$ for the algebra $(\Alg,\wedge)$ which is given by
\[
\bb(x_1\dotsb x_k) =\begin{aligned}[t]
		&\sum_{i=1}^{k-1} (-1)^{|x_1|+\dotsb+|x_i|+1} x_1\dotsb x_{i-1}(x_{i}\wedge x_{i+1})x_{i+2}\dotsb x_k \\
		&{}+ (-1)^{|x_k|(|x_1|+\dotsb+|x_{k-1}|)+\deg x_k} (x_k\wedge x_1) x_2 \dotsb x_k.
		\end{aligned}
\]
Therefore, the homology $H(\dcbc\Alg,\fp^{\fm^\can_\Alg}=\dd^*+\bb^*)$ is a version of cyclic cohomology of the $\DGA$ $(\Alg,\dd,\wedge)$.
We refer to \cite{Cieliebak-Volkov-cyclic} for a comparison of eight versions of cyclic cohomology
(see also \cite[Section 3.3]{Pavel-thesis}). 

Note that the twist of a general filtered $\dIBL$ structure $\fp=\{\fp_{1,1,0},\fp_{2,1,0},\fp_{1,2,0}\}$ with a general Maurer--Cartan element $\fm=\{\fm_{\ell,g}\}$ has by \eqref{eq:twisted-operations} the form%
\footnote{The twisted cobrackets alone form a version of a quantum co-$L_\infty$ algebra.
We could thus study the structures in this paper up to homotopy of quantum co-$L_\infty$ algebras instead.}
\[
	\fp^\fm = \left\{
		\begin{aligned}
			\fp^\fm_{1,1,0} &= \fp_{1,1,0} + \fp_{2,1,0}\circ_1 \fm_{1,0},\\
			\fp^\fm_{2,1,0} &= \fp_{2,1,0},\\
			\fp_{1,2,0}^\fm &= \fp_{1,2,0} + \fp_{2,1,0}\circ_1\fm_{2,0},\\
			\fp_{1,\ell,g} &= \fp_{2,1,0}\circ_1 \fm_{\ell,g}\quad\text{for }(\ell,g)\in\N\times\N_0\backslash\{(1,0),(2,0)\}
		\end{aligned}\right\}.
\]

Consider now the cohomology $(H\coloneqq H(\Alg),\dd=0,\la\cdot,\cdot\ra_H)$ as a cyclic cochain complex and write the associated canonical $\dIBL$ algebra as
\[
	\dIBL(H)=\bigl((\dcbc H)[2-n],\fq=\{\fq_{1,1,0}=0,\fq_{2,1,0},\fq_{1,2,0}\}\bigr).
\]
Let $P\colon\Alg\to\Alg$ be a special propagator and $\HH\subset\Alg$ the associated harmonic subspace.
We identify the cyclic cochain complexes $(H,\dd=0,\la\cdot,\cdot\ra_H)$ and $(\HH,\dd=0,\la\cdot,\cdot\ra|_\HH)$ via the quotient map \eqref{eq:canon-quotient-to-hom}.
Proposition~\ref{prop:structureexists} then associates to $P$ an $\IBL_{\infty}$ homotopy equivalence ${\ff^P} \colon \dIBL(\Alg)\to \dIBL(H)$,
and Proposition~\ref{prop:MC} asserts that the pushforward 
\begin{equation*}
	\ff^P_*\fm^\can_\Alg = \{ (\ff^P_*\fm_\Alg^\can)_{\ell,g}\}
\end{equation*}
is a Maurer--Cartan element in $\dIBL(H)$.
We call $\ff^P_*\fm^\can_\Alg$ the \emph{pushforward Maurer--Cartan element} in $\dIBL(H)$ associated to $P$.

The $\IBLinfty$ algebras $\dIBL^{\ff^P_*\fm^\can_\Alg}(H)$ and $\dIBL^{\ff^{P^\prime}_*\fm^\can_\Alg}(H)$ for different special propagators $P$ and $P^\prime$ are $\IBLinfty$ homotopy equivalent because each of them is $\IBLinfty$ homotopy equivalent to $\dIBL^{\fm^\can_\Alg}(\Alg)$ via the twisted morphisms 
\[
	(\ff^P)^{\fm^\can_\Alg}\colon\dIBL^{\fm^\can_\Alg}(\Alg)\longrightarrow\dIBL^{\ff^P_*\fm^\can_\Alg}(H).
\]
In fact, a stronger assertion holds: the Maurer--Cartan elements $\ff^P_*\fm^\can_\Alg$ and $\ff^{P^\prime}_*\fm^\can_\Alg$ are \emph{gauge equivalent} in the sense of~\cite[Section~9]{Cieliebak-Fukaya-Latschev}.%
\footnote{A Maurer--Cartan element in $\dIBL(H)$ is equivalent to a quantum $\Ainfty$ algebra on $H$ (via a BV formalism due to the second author), and gauge equivalence corresponds to homotopy equivalence.}
See~\cite{Cieliebak-Volkov-Chern-Simons} for a proof in the analytic case (to be introduced in Section~\ref{sec:analytic-construction}), which can be adapted to work here as well.

The preceding discussion is summarized in the following theorem.

\begin{thm}[{\cite[Theorem~12.10]{Cieliebak-Fukaya-Latschev}}]\label{homologyBLI}
Let $(\Alg,\dd,\wedge,\la\cdot,\cdot\ra)$ be a nonnegatively graded cyclic $\DGA$ of degree~$n\in\N_0$, and let $H=H(\Alg,\dd)$ be its cohomology.
Then there is a natural $\IBLinfty$ structure on $(\dcbc H)[2-n]$
which is $\IBLinfty$ homotopy equivalent to the canonical twisted $\dIBL$ structure
on $(\dcbc \Alg)[2-n]$.
In particular, the homology of $\dcbc H$ is isomorphic to the cyclic cohomology of $(\Alg,\dd,\wedge)$. 
\end{thm}

Inserting \eqref{eq:contribution-of-graph} and \eqref{eq:triple-intersection} in~\eqref{eq:pushforward-mc} leads to an explicit description of the pushforward Maurer--Cartan element $({\ff_P}_*\fm^\can_\Alg)_{k,\ell,g}$ in terms of \emph{trivalent ribbon graphs}
\[
	RG^3_{k,\ell,g}\coloneqq\bigl\{\Gamma\in RG_{k,\ell,g}\mid d(v) = 3\text{ for all } v\in C^0_{\rmint}(\Gamma)\bigr\}.
\]
First of all, given $\Gamma\in RG_{k,\ell,g}^3$, we denote by 
\[
	e\coloneqq \# C^1_\rmint(\Gamma)\quad\text{and}\quad s\coloneqq \# C^0_\rmext(\Gamma)
\]
the numbers of interior edges and exterior vertices, respectively.
Evaluating the Euler characteristic $\chi(\Sigma_\Gamma)$ of the associated surface $\Sigma_\Gamma$ in two ways gives then 
\begin{equation}\label{eq:euler-characteristic}
	2-2g-\ell = \chi(\Sigma_\Gamma) = k-e.
\end{equation}
Counting flags of $\Gamma$ in two ways, using trivalency, gives
\begin{equation}\label{eq:trivalency}
	2e + s = 3k.
\end{equation}
Eliminating $e$ from these equations yields
\begin{equation}\label{eq:k}
   k = 2(2g+\ell-2)+s. 
\end{equation}

\begin{cor}\label{cor:alg-MC}
The value of $(\ff^P_*\fm^\can_\Alg)_{\ell,g}\in \wh{E}_{\ell}\dcbc H$ on the tensor product
\[
	\alpha\coloneqq\alpha^1_1\cdots \alpha^1_{s_1}\otimes\dots\otimes  \alpha^\ell_1\cdots \alpha^\ell_{s_{\ell}},\quad\text{where }\alpha^b_j\in H, s_b\in \N, 
\]
can be written as 
\[
   (\ff^P_*\fm^\can_\Alg)_{\ell,g}(\alpha) = 
   \sum_{\Gamma\in RG^3_{k,\ell,g}}(-1)^{r_\Gamma}C_\Gamma (\ff^P_*\fm^\can_\Alg)_\Gamma(\alpha),
\]
where $k$ is determined by \eqref{eq:k} with $s=\sum_{i=1}^\ell s_i$ and the numbers $(\ff^P_*\fm^\can_\Alg)_\Gamma(\alpha)$ are defined similarly as in the construction from Theorem~\ref{thm:homotopyequiv} using the following assignments (see Figure~\ref{fig:pushforward-mc}):
\begin{itemize}
\item to the $j$-th exterior vertex on the $b$-th boundary component we assign $\alpha^b_j$;
\item to each interior vertex we assign the triple intersection product $\m_2^+$;  
\item to each interior edge we assign the element 
	\[
		\PP=\sum_{i,j}\pm P^{ij}e_i\otimes e_j
	\]
	dual to the map $x\otimes y\mapsto \la Px,y\ra$ (the Schwarz kernel of $P$). 
\end{itemize}
\end{cor}
\begin{figure}
	\centering
	\begin{tikzpicture}
	\node[point,label={[below,yshift=-.1cm]:$\m_2^+$}] (L) at (0,0) {};
	\node[point,label={[below,yshift=-.1cm]:$\m_2^+$}] (R) at ($(L)+(\edgelen,0)$) {};
	\node[point,label={[above]:$\m_2^+$}] (U) at ($(L)+(60:\edgelen)$) {};
	\node[root,label={[above left]:$\alpha_1^1$}] (O) at ($(L)+(150:\leaflen)$) {};
	\node[leaf,label={[above right]:$\alpha_2^1$}] (OI) at ($(R)+(30:\leaflen)$) {};
	\node[root,label={[below, yshift=-.1cm]:$\alpha_1^2$}] (OII) at ($(U)+(-90:\leaflen)$) {};
	\draw (L) -- (R) node[midway,below] {$\PP$} -- (U) node[pos=0.5,right] {$\PP$} -- (L) node[left,pos=0.5] {$\PP$};
	\draw (L) -- (O);
	\draw (R) -- (OI);
	\draw (U) -- (OII);
	\end{tikzpicture}
	\[ = \sum_{i_1,\dotsc,i_6} \pm P^{i_1 i_2} P^{i_3 i_4} P^{i_5 i_6} \m_2^+(\alpha_1^1 e_{i_6} e_{i_1}) \m_2^+(\alpha_1^2 e_{i_3} e_{i_2}) \m_2^+(\alpha_2^1 e_{i_4}e_{i_5})\]
	\caption{The contribution to $(\ff^P_*\fm^\can_\Alg)_{2,0}(\alpha_{1}^1\alpha_2^1\otimes\alpha_2^1)$ of a labeled ribbon graph immersed in the plane so that the cyclic ordering at interior vertices agrees with the counterclockwise orientation.}
	\label{fig:pushforward-mc}
\end{figure}

\section{The algebraic construction}
\subsection{Algebraic vanishing results}

We are interested in conditions which imply the vanishing of the number $(\ff^P_*\fm^\can_\Alg)_\Gamma(\alpha)\in\R$ associated in Corollary~\ref{cor:alg-MC} to a trivalent ribbon graph $\Gamma\in RG_{k,\ell,g}^3$ and an element $\alpha\in(\cbc B)^{\otimes\ell}$.

We abbreviate the Euler characteristic of the graph $\Gamma$ by
\[
	\chi\coloneqq \chi(\Gamma)=\chi(\Sigma_\Gamma)\in\{1,0,-1,-2,\dotsc\}
\]
and define the \emph{number of loops} 
\[
	\gamma\coloneqq 1-\chi\in\N_0.
\]
The following proposition is an algebraic analog of Proposition~\ref{prop:ana-vanishing} in the analytic case which will be introduced in Section~\ref{sec:analytic-construction}.
Its proof is a straightforward adaption of the proof of \cite[Propositions 4.4.1 and 4.4.2]{Pavel-thesis}.

\begin{prop}\label{prop:alg-vanishing}
	Let $(\Alg,\dd,\wedge,\la\cdot,\cdot\ra)$ be a nonnegatively graded unital cyclic $\DGA$ of degree~$n\in\N_0$.
	Let $B\subset \Alg$ be a quasi-isomorphic subcomplex, and let $P\colon\Alg\to\Alg$ be a propagator with respect to a projection onto~$B$.
	Suppose that $\Gamma\in RG_{k,\ell,g}^3$ is a trivalent ribbon graph which is not the Y-tree (see Figure~\ref{fig:Y-tree}) and $\alpha\in(\cbc B)^{\otimes\ell}$ a tensor product
	\[
		\alpha=\alpha_1^1\dotsb\alpha_{s_1}^1\otimes\dotsb\otimes\alpha_1^\ell\dotsb\alpha_{s_{\ell}}^\ell,\quad\text{where }\alpha_i^b\in B, s_b\in\N, 
	\]
	such that
	\begin{equation*}
		(\ff^P_*\fm^\can_\Alg)_\Gamma(\alpha)\neq 0.
	\end{equation*}
	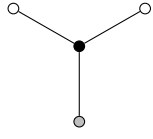
\begin{figure}
	\centering
	\begin{tikzpicture}
	\node[point] (A) at (0,0) {};
	\node[leaf] (AA) at ($(A)+(30:\leaflen)$) {};
	\node[leaf] (AB) at ($(A)+(150:\leaflen)$) {};
	\node[root] (AC) at ($(A)+(270:\leaflen)$) {};
	\draw (A) -- (AA);
	\draw (A) -- (AB);
	\draw (A) -- (AC);
	\end{tikzpicture}
	\caption{The Y-tree.}
	\label{fig:Y-tree}
	\end{figure}%
	If $P$ is special, then the following holds:
	\begin{enumerate}
		\item If $B^0 = \R\cdot 1$, then \emph{positivity of degrees} holds: 
		\[ 
			\deg(\alpha_i^b)>0\quad\text{for all }i\in\{1,\dotsc,s_b\}, b\in\{1,\dotsc,\ell\}.
		\]
		\item If $B$ satisfies
			\begin{equation*}
				B\wedge B\subset B,
			\end{equation*}
		      then $\gamma\ge 1$, i.e., all trees vanish.
	\end{enumerate}
	Positivity of degrees implies:
	\begin{enumerate}
	\setcounter{enumi}{2}
		\item If $\gamma=1$ or $n=3$, then $\deg(\alpha_i^b)=1$ for all $i\in\{1,\dotsc,s_b\}$, $b\in\{1,\dotsc,\ell\}$.
		\item If $n\ge 4$, then $\gamma\le 1$, i.e., all graphs with more than one loop vanish.
	\end{enumerate}
\end{prop}

\begin{proof}
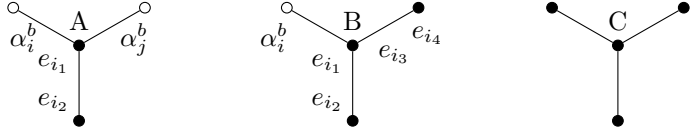
\begin{figure}
	\centering
	\def\edgelen{3cm}
	\def\leaflen{1cm}
	\begin{tikzpicture}
	\node[point,label={A}] (A) at (0,0) {};
	\node[leaf] (AA) at ($(A)+(30:\leaflen)$) {};
	\node[leaf] (AB) at ($(A)+(150:\leaflen)$) {};
	\node[point] (AC) at ($(A)+(270:\leaflen)$) {};
	\draw (A) -- (AA) node[pos=.9,below] {$\alpha_j^b$};
	\draw (A) -- (AB) node[pos=.9,below] {$\alpha_i^b$};
	\draw (A) -- (AC) node[pos=.8,left] {$e_{i_2}$} node[pos=.2,left] {$e_{i_1}$};
	\end{tikzpicture}
	\hspace{1cm}
	\begin{tikzpicture}
	\node[point,label={B}] (A) at (0,0) {};
	\node[point] (AA) at ($(A)+(30:\leaflen)$) {};
	\node[leaf] (AB) at ($(A)+(150:\leaflen)$) {};
	\node[point] (AC) at ($(A)+(270:\leaflen)$) {};
	\draw (A) -- (AA) node[pos=.2,below right] {$e_{i_3}$} node[pos=.8,below right] {$e_{i_4}$};
	\draw (A) -- (AB) node[pos=.9,below left] {$\alpha_i^b$};
	\draw (A) -- (AC) node[pos=.2,left] {$e_{i_1}$} node[pos=.8,left] {$e_{i_2}$};
	\end{tikzpicture}
	\hspace{1cm}
	\begin{tikzpicture}
	\node[point,label={C}] (A) at (0,0) {};
	\node[point] (AA) at ($(A)+(30:\leaflen)$) {};
	\node[point] (AB) at ($(A)+(150:\leaflen)$) {};
	\node[point] (AC) at ($(A)+(270:\leaflen)$) {};
	\draw (A) -- (AA);
	\draw (A) -- (AB);
	\draw (A) -- (AC);
	\end{tikzpicture}	
	\caption{The three types of interior vertices in $\Gamma\neq Y$.}
	\label{fig:abc-vertex}
\end{figure}
Recall that in order to write down $(\ff^P_*\fm^\can_\Alg)_\Gamma(\alpha)$, the construction in Theorem~\ref{thm:homotopyequiv} associates basis elements $(e_i)$ to interior flags, elements $\alpha_i^b$ to exterior flags, and the triple intersection product $\m_2^+$ to interior vertices.
The crucial observation is that the degrees must add up to~$\deg\m_2^+=n$ around each interior vertex and~$\deg\PP=n-1$ along each interior edge.
This implies that the \emph{total degree}
\begin{equation*}
	\deg\alpha\coloneqq\sum_{b=1}^\ell\sum_{i=1}^{s_b} \deg(\alpha_i^b)
\end{equation*}
satisfies
\begin{equation}\label{eq:degree-eqn}
   nk = (n-1)e + \deg \alpha.
\end{equation}
Expressing $(k,e)$ in terms of $(\chi,s)$ using~\eqref{eq:trivalency} and~\eqref{eq:euler-characteristic} gives 
\begin{equation*}
   \deg\alpha = (n-3)\chi + s.
\end{equation*}
With this (i) immediately implies~(iii) and~(iv).
Hence, only (i) and (ii) are left.

Following \cite{Pavel-thesis}, we call an interior vertex an
\begin{itemize}
\item {\em A-vertex} if it has $1$ adjacent interior edge; 
\item {\em B-vertex} if it has $2$ adjacent interior edges; 
\item {\em C-vertex} if it has $3$ adjacent interior edges (see Figure~\ref{fig:abc-vertex}).
\end{itemize}
Since $\Gamma$ is not the $Y$-tree, each interior vertex is of type A, B or C.

(i) 
Since the expression for $(\ff^P_*\fm_\Alg^\can)_\Gamma(\alpha)$ is homogenous in $\alpha_i^b$ and $B^0=\R\cdot 1$, it suffices to show that none of the $\alpha_i^b$ equals $1$. 

For an A-vertex, let $e_{i_1}$ be the basis element associated to the adjacent interior flag, $e_{i_2}$ the basis element associated to the other flag on the corresponding interior edge, and $\alpha_i^b,\alpha_j^b\in B$ the elements associated to the adjacent exterior edges.
Then $(\ff^P_*\fm_\Alg^\can)_\Gamma(\alpha)$ involves the sum
\begin{equation*}
	\sum_{i_1}\la\alpha_i^b\wedge\alpha_j^b,e_{i_1}\ra P^{i_2 i_1}
	= \la\alpha_i^b\wedge\alpha_j^b,Pe^{i_2}\ra
	= \pm\la P(\alpha_i^b\wedge\alpha_j^b),e^{i_2}\ra,
\end{equation*}
where we have used $\sum_{i_1} P^{i_2 i_1}e_{i_1}=Pe^{i_2}$ and the symmetry of $P$. 
If $\alpha_i^b$ or $\alpha_j^b$ is equal to $1$, then $\alpha_i^b\wedge\alpha_j^b\in B$, hence $P(\alpha_i^b\wedge\alpha_j^b)=0$ because $P$ is special and thus $P(B)=0$.

For a B-vertex, let $e_{i_1},e_{i_3}$ be the basis elements associated to the adjacent interior flags, $e_{i_2},e_{i_4}$ the basis elements associated to the other flags on the corresponding interior edges, and $\alpha_i^b\in B$ the element associated to the adjacent exterior edge.
If $\alpha_i^b=1$, then $(\ff^P_*\fm^\can_\Alg)_\Gamma(\alpha)$ involves the sum
\begin{equation*}
	\begin{aligned}
		\smash{\sum_{i_1,i_3}}\la 1,e_{i_1}\wedge e_{i_3}\ra P^{i_2 i_1}P^{i_4 i_3}
		&= \la 1,Pe^{i_2}\wedge Pe^{i_4}\ra\\
		&= \pm\la Pe^{i_2},Pe^{i_4}\ra\\
		&= \pm\la P(Pe^{i_2}),e^{i_4}\ra,
	\end{aligned}
\end{equation*}
which vanishes again because $P$ is special and thus $P\circ P=0$.

(ii)
If $B$ is a sub-$\DGA$, then $\alpha_i^b,\alpha_j^b\in B$ implies $\alpha_i^b\wedge\alpha_j^b\in B$, so the proof of (i) shows that $\Gamma$ cannot have an A-vertex.
This excludes each tree except the $Y$-tree, which is excluded by hypothesis.
\end{proof}

Notice that if $n\neq 3$, then $\alpha$ determines the signature $(k,\chi)$ of any~$\Gamma\in RG_{k,\ell,g}^3$ whose contribution to $(\ff^P_*\fm^\can_\Alg)_{\ell,g}(\alpha)$ might be nonzero as follows:
\[
	\begin{pmatrix}
		\deg\alpha \\
		s
	\end{pmatrix} =
	\begin{pmatrix}
		1 & n-1 \\
		1 & 2
	\end{pmatrix}
	\begin{pmatrix}
		k \\
		\chi
	\end{pmatrix}
	\quad\overset{n\neq 3}{\Longrightarrow}\quad
	\begin{pmatrix}
		k \\
		\chi
	\end{pmatrix} =
	\frac{1}{3-n}
	\begin{pmatrix}
		2 & 1-n \\
		-1 & 1
	\end{pmatrix}
	\begin{pmatrix}
		\deg \alpha \\
		s
	\end{pmatrix}.
\]
The conditions on $(k,\gamma)$ obtained by restricting to pairs $(\alpha,s)$ satisfying the bounds $s\le \deg\alpha\le ns$ are then equivalent to (iii) and (iv).

We remark that the formula for $(\ff^P_*\fm^\can_\Alg)_\Gamma(\alpha)\in\R$ makes sense also for graphs with $0\le s<\ell$, i.e., when there is a boundary component with no exterior vertex, and the vanishing results still hold.
However, such graphs do not naturally appear in our theory.

\begin{remark}\label{rem:vanishing-in-low-dimensions}
	We summarize some facts from \cite{Pavel-thesis} (originally in the analytic case) about the low-degree cases in Proposition~\ref{prop:alg-vanishing}:
	\begin{description}
	\item[$n=0$] We must have $P=0$, so all graphs vanish trivially.
	\item[$n=1$] If~$B^{0}=\R\cdot 1$ and $P$ is special, then every $\alpha_i^b$ has degree $1$ by (i).
	We then have $k=s$ by \eqref{eq:degree-eqn} and $k=e$ by \eqref{eq:trivalency}.
	If $v\in B^1$ is the element dual to~$1$, then $B=\mathrm{span}_\R\{1,v\}$, and hence there is no A-vertex by the proof of (ii).
	We deduce that $(\ff^P_*\fm^\can_\Alg)_{\Gamma}(\alpha)$ for $\Gamma\in RG^3_{k,\ell,g}$ does not necessarily vanish only if the underlying graph is the \emph{circular graph}~$C_s$, i.e., the graph with $s$ interior vertices, $1$ loop, and no A-vertex, and $\alpha$ is a product of $s$ copies of~$v$ (or its nonzero multiples).
	This value does not depend on the ribbon structure, can be computed explicitly, and in the de Rham case is related to Bernoulli numbers. 
\item[$n=2$] Suppose that $B^0=\R\cdot 1$, $B^1=0$ and $P$ is special, so that $B=\mathrm{span}_\R\{1,v\}$, where $v\in B^2$ is the element dual to $1$.
	One can show that the only case when $(\ff^P_*\fm^\can_\Alg)_{\Gamma}(\alpha)$ for $\Gamma\in RG_{k,\ell,g}^3$ does not necessarily vanish is when the characteristic of the underlying graph satisfies $(k,\gamma)=(3s,s+1)$ and $\alpha$ is the product of $s$ copies of $v$ (or its nonzero multiples).
	We were not able to prove vanishing in general except for small $s$. 
	\end{description}
\end{remark}

Proposition~\ref{prop:alg-vanishing} and Remark~\ref{rem:vanishing-in-low-dimensions} immediately imply

\begin{cor}\label{cor:alg-vanishing}
	Let $(\Alg,\dd,\wedge,\la\cdot,\cdot\ra)$ be a nonnegatively graded unital cyclic $\DGA$ of degree $n\in\N_0\backslash\{2\}$ whose cohomology $H\coloneqq H(\Alg,\dd)$ is simply connected.
Let $P\colon\Alg\to \Alg$ be a special propagator with associated harmonic subspace $\HH\subset\Alg$, and let $\ff^P\colon\dIBL(\Alg)\to\dIBL(H)$ be the associated $\IBLinfty$ homotopy.
Then the only contributions to the pushforward Maurer--Cartan element $\ff^P_*\fm^\can_\Alg$ come from {\em trees}, so that 
\[
	(\ff^P_*\fm^\can_\Alg)_{\ell,g}=0\quad\text{for all }(\ell,g)\neq(1,0).
\]
Consider $H$ as a cyclic $\DGA$ $(H,\dd=0,\wedge,\la\cdot,\cdot\ra_H)$, and let $\fm_H^\can$ be the canonical Maurer--Cartan element in $\dIBL(H)$.
If in addition $\HH\wedge \HH \subset \HH$, then the only contribution comes from the $Y$-tree, so that
\[
	\ff^P_*\fm^\can_\Alg=\fm_H^\can.
\]
\end{cor}

\subsection{Weak functoriality of the twisted dIBL construction}

We have the following analog of Theorem~\ref{thm:homotopyequiv} for the twisted $\dIBL$ algebra in the simply connected case:

\begin{prop}\label{prop:functoriality}
Let $(\Alg,\dd,\wedge,\la\cdot,\cdot\ra)$ be a simply connected cyclic $\DGA$ of degree~$n\in \N_0$, and let $B\subset \Alg$ be a quasi-isomorphic cyclic sub-$\DGA$.
Let $P\colon\Alg\to\Alg$ be a special propagator with respect to the unique symmetric projection $\pi_B\colon\Alg\onto B$, and let $\ff^P\colon \dIBL(\Alg)\to\dIBL(B)$ be the associated $\IBLinfty$ homotopy equivalence. 
Then the canonical Maurer--Cartan elements $\fm_\Alg^\can$ in $\dIBL(\Alg)$ and $\fm_B^\can$ in $\dIBL(B)$ satisfy
\begin{equation}\label{eq:pushforward-equal}
	\ff^P_*\fm^\can_\Alg = \fm^\can_B.
\end{equation}
In particular, the twisted morphism $(\ff^P)^{\fm_\Alg^\can}$ is an $\IBLinfty$ homotopy equivalence
\[
	(\ff^P)^{\fm_\Alg^\can}\colon \dIBL^{\fm_\Alg^\can}(\Alg)\stackrel{\simeq}\longrightarrow\dIBL^{\fm_B^\can}(B). 
\]
\end{prop}

\begin{proof}
Equation \eqref{eq:pushforward-equal} for $n\ge 3$ follows from Proposition~\ref{prop:alg-vanishing} using that $P$ is special and $B$ simply connected. %
Since $\Alg$ is simply connected, the case $n=1$ is impossible, and for $n=2$ we must have $P=0$ 
for degree reasons.
\end{proof}

Let $f\colon B\to \Alg$ be a morphism of cyclic $\DGA$s.
Consider the associated morphism of $\dIBL$ algebras $\dIBL(f)=\pi^{*}_B\colon\dIBL(B)\to\dIBL(\Alg)$ at the end of \S\ref{ss:dIBL-cyc-cochain}.
It will in general not define a morphism $\dIBL^{\fm^\can_B}(B)\to\dIBL^{\fm^\can_\Alg}(\Alg)$ as it need not preserve the Hochschild codifferential $\fp_{1,1,0}^{\fm^\can} = \bb^*$.
On the other hand, there is a natural $\IBLinfty$ morphism in the opposite direction $\dIBL^{\fm^\can_\Alg}(\Alg)\to\dIBL^{\fm^\can_B}(B)$ from Proposition~\ref{prop:functoriality} defined up to $\IBLinfty$ homotopy equivalence.  
It would be interesting to know whether the $\dIBL^{\fm^\can}$ construction induces a natural \emph{contravariant} functor from the homotopy category of cyclic $\DGA$s to the homotopy category of $\dIBL$ algebras.

\subsection{Twisted dIBL algebras associated to Poincar\'e DGAs}

We now explain how to extend the $\dIBL^{\fm^\can}$ construction to Poincar\'e $\DGA$s using differential Poincar\'e duality models.
Let $\Alg$ be a $\PDGA$ whose cohomology $H\coloneqq H(\Alg)$ is 2-connected, i.e., $H^0=\R$ and $H^1=H^2=0$. 
By Theorem~\ref{thm:existence-uniqueness-PDmodel}(a),  $\Alg$ admits a simply connected $\dPD$ model $\MM$.
Proposition~\ref{propIBLI2} associates to~$\MM$ the twisted $\dIBL$ algebra $\dIBL^{\fm^\can_\MM}(\MM)$. 
Suppose that $\MM^{\prime}$ is another simply connected $\dPD$ model of~$\Alg$.
By Theorem~\ref{thm:existence-uniqueness-PDmodel}(b), there exists a simply connected $\dPD$ algebra $\Alg_1$ and quasi-isomorphisms of $\dPD$ algebras
\begin{equation*}
\begin{tikzcd}	
	& \Alg_1 & \\
	\MM \ar[hook,ur,
	] & & \MM^{\prime}.\ar[hook',ul]
\end{tikzcd}
\end{equation*}
Proposition~\ref{prop:functoriality} then extends the pullbacks to $\IBLinfty$ homotopy equivalences
\begin{equation*}
\begin{tikzcd}	
	& \dIBL^{\fm^\can_{\Alg_1}}(\Alg_1) \ar[dl,"\simeq",swap]\ar[dr,"\simeq"] & \\
	\dIBL^{\fm^\can_{\MM}}(\MM)  & & \dIBL^{\fm^\can_{\MM^\prime}}(\MM^{\prime}).
\end{tikzcd}
\end{equation*}
Since $\IBLinfty$ homotopy equivalences are invertible and composable, this shows that $\dIBL^{\fm_\MM^\can}(\MM)$ is independent of the simply connected $\dPD$ model $\MM$ of $\Alg$ up to $\IBLinfty$ homotopy equivalence (provided that $H$ is $2$-connected). 
If $\Alg^{\prime}$ is a $\PDGA$ weakly equivalent to $\Alg$, then a $\dPD$ model~$\MM^\prime$ for~$\Alg^\prime$ is also one for $\Alg$, and hence the $\dIBL$ algebras associated to $\Alg$ and~$\Alg^{\prime}$ are $\IBLinfty$ homotopy equivalent by the previous discussion. 
We have thus proved the following result which corresponds to Theorem~\ref{thm:alg-intro} in the Introduction:

\begin{thm}\label{thm:alg-construction}
The map $\Alg\mapsto\dIBL^{\fm_\MM^\can}(\MM)$, where $\MM$ is a simply connected differential Poincar\'e duality model of $\Alg$, assigns to each Poincar\'e $\DGA$ $\Alg$ of degree $n\in\N_0$ whose cohomology is 2-connected 
a $\dIBL$ algebra of degree $n-3$ whose homology is the cyclic cohomology of the $\DGA$ $\Alg$, canonically up to $\IBLinfty$ homotopy equivalence. 
If two such Poincar\'e $\DGA$s~$\Alg$ and~$\Alg^{\prime}$ are weakly equivalent, then their associated $\dIBL$ algebras are $\IBLinfty$ homotopy equivalent. 
\hfill$\square$
\end{thm}

{\bf Application to the de Rham complex. }
Let $M$ be a connected closed oriented manifold of dimension $n$.
Consider the de Rham complex
\[
	\bigl(\Om\coloneqq\Om^*(M),\dd,\wedge\bigr)
\]
with the \emph{intersection pairing} $\la\cdot,\cdot\ra\colon \Om\times\Om\to \R$ of degree $n$ defined by
\[
	\la\alpha,\beta\ra\coloneqq\int_M\alpha\wedge\beta.
\]
Stokes' theorem implies that $(\Om,\dd,\wedge,\la\cdot,\cdot\ra)$ is a $\DGA$ with pairing in the sense of Section~\ref{sec:PDGA}.
The Poincar\'e duality theorem implies that the induced pairing $\la\cdot,\cdot\ra_{\HdR}\colon \HdR\times \HdR\to\R$ on the de Rham cohomology $\HdR\coloneqq \HdR(M)\coloneqq H(\Om,\dd)$ is perfect, and so $\Om$ is an oriented $\PDGA$ of degree $n$.
It is of course unital with unit the constant $1$.

A smooth homotopy equivalence $f\colon M\to M^{\prime}$ induces a $\DGA$ quasi-isomorphism $f^*\colon \Om^*(M^{\prime})\to\Om^*(M)$.
So the map $f^*$ is a $\PDGA$ quasi-isomorphism if and only if the induced isomorphism $H(f^*)\colon H_n(M;\Z)\to H_n(M^{\prime};\Z)$ intertwines the orientations, i.e maps the fundamental cycle~$[M]$ to the fundamental cycle~$[M^{\prime}]$.
We will call such $f$ an {\em orientation preserving homotopy equivalence}. 
Theorem~\ref{thm:alg-construction} then implies the following result which corresponds to Corollary~\ref{cor-alg-intro} in the Introduction:

\begin{cor}\label{cor:alg-construction}
The map $M\mapsto\dIBL^{\fm_\MM^\can}(\MM)$, where $\MM$ is a simply connected differential Poincar\'e duality model of the de Rham complex $\Om^*(M)$, assigns to each closed $n$-manifold $M$ whose de Rham cohomology $\HdR(M)$ is $2$-connected a $\dIBL$ algebra of degree $n-3$ whose homology is the cyclic cohomology of $(\Om,\dd,\wedge)$, canonically up to $\IBLinfty$ homotopy equivalence.
An orientation preserving homotopy equivalence between two such manifolds $M$ and $M^{\prime}$ gives rise to an $\IBLinfty$ homotopy equivalence between their associated $\dIBL$ algebras. 
\hfill$\square$
\end{cor}

\begin{rem}
On the class of manifolds $M$ that admit an orientation reversing homotopy equivalence $M\to M$, the specification ``orientation preserving'' can be dropped in Corollary~\ref{cor:alg-construction}.
However, there are manifolds (such as the complex or quaternionic projective spaces $\C P^m$, $\mathbb{H}P^m$ for even $m$) that do not admit an orientation reversing homotopy equivalence $M\to M$.
For further examples see~\cite{chiral}.
\end{rem}

\section{The analytic construction}
\label{sec:analytic-construction}

Throughout this section, $M$ will be a connected closed oriented manifold of dimension~$n$ and $\bigl(\Om\coloneqq\Om^*(M),\dd,\wedge)$ its de Rham complex equipped with the intersection pairing $\la\alpha,\beta\ra=\int_M\alpha\wedge\beta$.
We denote by $\HdR\coloneqq \HdR(M)$ the de Rham cohomology. 

Given a Riemannian metric $g$ on $M$, let~$\star_g\colon \Om^{\bullet}\to\Om^{n-\bullet}$ be the induced Hodge star operator, $\dd^\star_g\colon \Om^\bullet\to\Om^{\bullet-1}$ the associated codifferential, and $\HH_g\colon\ker\dd\cap \ker \dd^\star_g$ the corresponding harmonic subspace.
The Hodge theorem asserts that
\begin{equation}\label{eq:riemannian-hodge-decomposition}
	\Om = \HH_g \oplus\im\dd\oplus \im \dd^\star_g
\end{equation}
is a Hodge decomposition, and we conclude that $\Om$ is an oriented $\PDGA$ of Hodge type.
Moreover, the pairing $\la\cdot,\cdot\ra\colon\Om\times\Om\to\R$ is nondegenerate, so every harmonic subspace $\HH\subset \Om$ admits a \emph{unique} harmonic projection $\pi_\HH\colon\Om\onto\HH$.

\subsection{Analytic propagators}

We call a propagator $P\colon\Om\to\Om$ {\em analytic}%
\footnote{ Here ``analytic'' does not stand for ``real analytic'' but just for good analytic properties.}
if it can be written as
\begin{equation}\label{eq:propagator-and-kernel}
	(P\alpha)(x) = \int_{y\in M}\PP(x,y)\wedge\alpha(y)
\end{equation}
for a smooth $(n-1)$-form $\PP$ on the {\em oriented real blow-up} of the diagonal $\Delta\subset M\times M$.
This blow-up, which we denote by $\Bl_\Delta(M\times M)$, is obtained by replacing~$\Delta$ with the real oriented projectivization $P^+N\Delta \coloneqq N\Delta/\sim$ of the normal bundle $N\Delta\to \Delta$, where $v_1\sim v_2$ holds if and only if $v_1 = \alpha v_2$ for some $\alpha\in(0,\infty)$, and promoting polar coordinates around~$\Delta$ to boundary charts around $P^+N\Delta$.
This leads to a smooth compact manifold with boundary~$\partial \Bl_\Delta(M\times M)=P^+N\Delta$ whose interior is canonically identified with $(M\times M)\backslash\Delta$ (see~\cite{Cieliebak-Volkov-Chern-Simons,Pavel-thesis} for details).

A~version of the following lemma has been proved in dimension $3$ in~\cite{Bott-Cattaneo,Cattaneo-Mnev}, and in arbitrary dimension (with essentially the same proof) in~\cite{Cieliebak-Volkov-Chern-Simons,Pavel-thesis}:

\begin{lem}\label{lem:existence-of-analytic-propagator}
Given a harmonic subspace $\HH\subset\Om$, there exists an analytic propagator $P\colon\Om\to\Om$ with respect to the harmonic projection $\pi_\HH\colon\Om\onto\HH$.
The corresponding special propagator~$P_3$ from Lemma~\ref{lem:propagator} is again analytic.
\end{lem}

\begin{proof}[Sketch of proof]
	The integral kernel of $\pi_\HH\colon\Om\onto\HH$ is a closed smooth $n$-form on $M\times M$ Poincar\'e dual to $\Delta$. 
	Therefore, it admits a primitive $\PP\in\Om^{n-1}(\Bl_\Delta(M\times M))$ by Poincar\'e duality; we choose one and define $P$ up to a sign by~\eqref{eq:propagator-and-kernel}.
	The homotopy equation~\eqref{eq:P2} for $P$ then follows from Stokes' theorem.
	The symmetry condition~\eqref{eq:P3} for $P$ is equivalent to $\PP(x,y)=\pm\PP(y,x)$, which can be always achieved by taking $\PP_1(x,y)\coloneqq \frac{1}{2}(\PP(x,y)\pm \PP(y,x))$.
The last assertion follows by translating the formulas in Lemma~\ref{lem:propagator} to the integral kernels.
\end{proof}

\begin{remark}
It has been frequently claimed that the special propagator $P_g$ corresponding to the Hodge decomposition \eqref{eq:riemannian-hodge-decomposition} is analytic (see, e.g. \cite{Cattaneo-Mnev} and \cite{Axelrod-Singer-II}).
However, we have been unable to find a proof of this assertion.
\end{remark}

\subsection{The analytic Maurer--Cartan element}

The intersection pairing $\la\cdot,\cdot\ra\colon\Om\times\Om\to\R$ is not perfect for $n>0$, so that we cannot apply Propositions~\ref{prop:structureexists}, \ref{propIBLI2} to get ``$\dIBL^{\fm_\Om^\can}(\Om)$''.
Instead, one can construct a Maurer--Cartan element in
\[
	\dIBL(\HdR)=\bigl((\dcbc \HdR)[2-n],\fq_{1,1,0}=0,\fq_{2,1,0},\fq_{1,2,0}\bigr)
\]
by formally applying the homotopy transfer from Theorem~\ref{thm:homotopyequiv} to the triple intersection product $\m^+_2\in\dcbc_3\Om$ given by
\begin{equation}\label{eq:analytic-triple-intersection}
	\m_2^+(\alpha_0\alpha_1\alpha_2) = (-1)^{\deg \alpha_1 + n} \int_M \alpha_0 \wedge \alpha_1 \wedge \alpha_2.
\end{equation}
In analogy with Proposition~\ref{propIBLI2}, we can view \eqref{eq:analytic-triple-intersection} as the ``canonical Maurer--Cartan element in $\dIBL(\Om)$'', so that the twist of $\dIBL(\HdR)$ with its formal pushforward can be viewed as a model of ``$\dIBL^{\fm^{\can}}(\Om)$'', provided that the formal pushforward is a Maurer--Cartan element.
This strategy has been proposed in~\cite{Cieliebak-Fukaya-Latschev} and is carried out in~\cite{Cieliebak-Volkov-Chern-Simons}, leading to the following refinement of Theorem~\ref{thm:CV-intro} in the Introduction: 

\begin{thm}[\cite{Cieliebak-Volkov-Chern-Simons}]\label{thm:CV-MC}
There is a Maurer--Cartan element $\fm=\{\fm_{\ell,g}\}$ in $\dIBL(\HdR)$ which is defined naturally up to a gauge equivalence such that the homology of $(\dcbc \HdR,\fq^{\fm}_{1,1,0})$ equals the cyclic cohomology of $(\Om,\dd,\wedge)$.
\end{thm}

\begin{proof}[Sketch of proof]
	Let $P\colon\Om\to\Om$ be a special analytic propagator with respect to the harmonic projection $\pi_\HH\colon\Om\onto\HH$, and let~$\PP\in\Om^{n-1}(\Bl_\Delta(M\times M))$ be its integral kernel.
The value of~$\fm_{\ell,g}$
on the tensor product $\alpha=\alpha^1_1\cdots \alpha^1_{s_1}\otimes\dots\otimes  \alpha^\ell_1\cdots \alpha^\ell_{s_{\ell}}$, $\alpha^b_j\in \HdR$, $s_b\in\N$ is defined as a sum over trivalent ribbon graphs $\Gamma\in RG_{k,\ell,g}^3$   
with~$k$ determined by~\eqref{eq:k} similarly as in Corollary~\ref{cor:alg-MC}.
The contribution of $\Gamma$ is defined as the integral
\begin{equation}\label{eq:n}
	\fm_\Gamma(\alpha) \coloneqq \int_{X_\Gamma}\PP_\Gamma\times \alpha,
\end{equation}
where $X_\Gamma$ is a suitable configuration space of $k$ points of $M$ assigned to interior vertices, $\PP_\Gamma$ is a wedge product of the $(n-1)$-forms $\PP$ assigned to interior edges, and~$\alpha$ is a product of harmonic representatives of $\alpha_j^b$ assigned to exterior flags (see Figure~\ref{fig:triangle} for a clarifying example).
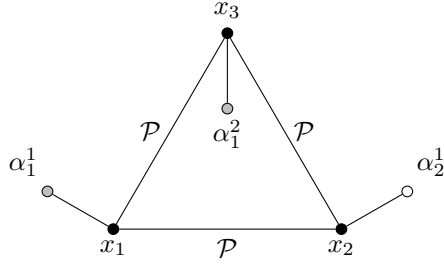
\begin{figure}
	\centering
	\begin{tikzpicture}
	\node[point,label={[below,yshift=-.1cm]$x_1$}] (A) at (0,0) {};
	\node[point,label={[below,yshift=-.1cm]$x_2$}] (B) at ($(A)+(0:\edgelen)$) {};
	\node[point,label={[above]$x_3$}] (C) at ($(B)+(120:\edgelen)$) {};
	\node[root,label={[above left]$\alpha_1^1$}] (AA) at ($(A)+(150:\leaflen)$) {};
	\node[leaf,label={[above right]$\alpha_2^1$}] (BB) at ($(B)+(30:\leaflen)$) {};
	\node[root,label={[below,yshift=-.1cm]$\alpha_1^2$}] (CC) at ($(C)+(270:\leaflen)$) {};
	\draw (A) -- (B) node[midway, below] {$\PP$}-- (C)  node[midway, right] {$\PP$} -- (A)  node[midway, left] {$\PP$};
	\draw (A) -- (AA);
	\draw (B) -- (BB);
	\draw (C) -- (CC);
	\end{tikzpicture}
\begin{equation*}
= \pm \int_{x_1, x_2, x_3}\PP(x_1,x_2)\wedge \PP(x_2,x_3)\wedge \PP(x_3,x_1)\wedge\alpha_1^1(x_1)\wedge \alpha_2^1(x_2)\wedge \alpha_2^1(x_3)
\end{equation*}
\caption{The contribution $\fm_\Gamma(\alpha_{1}^1\alpha_{2}^1\otimes\alpha^2_1)$ of a labeled ribbon graph immersed in the plane so that the cyclic ordering at interior vertices agrees with the counterclockwise orientation.
}
	\label{fig:triangle}
\end{figure}

The proof consists in overcoming the following four difficulties:

(1) Since $\PP$ is singular along the diagonal, it is a priori not clear that the integrals converge.
This is resolved by taking for $X_\Gamma$ a Fulton-MacPherson type compactification~\cite{Fulton-MacPherson} similar to the ones used in~\cite{Axelrod-Singer-II,Bott-Taubes}.

(2) The compactification $X_\Gamma$ has additional codimension one boundary components, so-called ``hidden faces'', which may obstruct the Maurer--Cartan equation for $\fm$.
This is resolved by showing, via a symmetry argument similar to the one in~\cite{Cattaneo-Mnev} going back to~\cite{Kontsevich-lowdim} and~\cite{Bott-Taubes}, 
that the integrals over hidden faces vanish.

(3) One needs to sum~\eqref{eq:n} over $\Gamma\in RG_{k,\ell,g}^3$ with suitable signs and combinatorial coefficients in order to obtain a Maurer--Cartan element. 

(4) One needs to produce a gauge equivalence between the Maurer--Cartan elements corresponding to different choices of a special analytic propagator $P$.
This uses an alternative description of gauge equivalence in terms of the Weyl formalism from~\cite{Eliashberg-Givental-Hofer}, and can also be reformulated in terms of an equivalence of BV actions corresponding to $\fm$ introduced in \cite{Pavel-thesis}.
\end{proof}

We denote the Maurer--Cartan element $\fm$ associated to a special analytic propagator~$P$ in Theorem~\ref{thm:CV-MC} by 
\[
	\fm^\ana_P
\]
and call it the \emph{analytic Maurer--Cartan element} in $\dIBL(\HdR)$.

\subsection{Analytic vanishing results}

The following result is an analog of Proposition~\ref{prop:alg-vanishing} in the analytic case and corresponds to~\cite[Propositions 4.4.1 and 4.4.2]{Pavel-thesis}.
Recall that $\gamma$ denotes the number of loops in $\Gamma$.%

\begin{prop}\label{prop:ana-vanishing}
Let $P\colon\Om\to\Om$ be a special analytic propagator,
$\Gamma\in RG_{k,\ell,g}^3$ a trivalent ribbon graph which is not the Y-tree, and $\alpha\in(\cbc\Om)^{\otimes \ell}$ a tensor product
	\[
		\alpha=\alpha_1^1\dotsb\alpha_{s_1}^1\otimes\dotsb\otimes\alpha_1^\ell\dotsb\alpha_{s_{\ell}}^\ell\text{ with }\alpha_i^b\in \HdR, s_b\in\N 
	\]
such that
	\begin{equation*}
		(\fm^{\ana}_P)_\Gamma(\alpha)\neq 0.
	\end{equation*}%
	Then the following holds:
	\begin{enumerate}
		\item The \emph{positivity of degrees} ($\HdR$ is connected by assumption): 
		\[ 
			\deg(\alpha_i^b)>0\quad\text{for all }i\in\{1,\dotsc,s_b\}, b\in\{1,\dotsc,\ell\}.
		\]
	\item If the harmonic subspace $\HH\subset\Om$ associated to $P$ satisfies
		\begin{equation*}
			\HH\wedge\HH\subset\HH,
		\end{equation*}
		then $\gamma\ge 1$, i.e., all trees vanish.
	\end{enumerate}
	The positivity of degrees implies:
	\begin{enumerate}
	\setcounter{enumi}{2}
\item If $\gamma=1$ or $n=3$, then $\deg(\alpha_i^b)=1$ for all $i\in\{1,\dotsc,s_b\}$, $b\in\{1,\dotsc,\ell\}$.
	\item If $n\ge 4$, then $\gamma\le 1$, i.e., all graphs with more than one loop vanish.
	\end{enumerate}
\end{prop}

Recall that a manifold $M$ is called \emph{geometrically formal} if it admits a Riemannian metric $g$ such that $\HH_g\wedge\HH_g\subset\HH_g$.
Remark~\ref{rem:vanishing-in-low-dimensions} on the cases $n\in\{0,1,2\}$ applies here, too, and we conclude:

\begin{cor}\label{cor:ana-vanishing}
Suppose that $M$ is a connected closed oriented manifold such that $\HdR^1=0$ which is not diffeomorphic to 
$S^2$.
Let~$P\colon\Om\to\Om$ be a special analytic propagator and $\fm^\ana_P$ the associated analytic Maurer--Cartan element in $\dIBL(\HdR)$.
Then we have
\[
	(\fm^{\ana}_P)_{\ell,g}=0\quad\text{for all }(\ell,g)\neq(1,0).
\]
Consider $\HdR$ as a cyclic $\DGA$ $(\HdR,\dd=0,\wedge,\la\cdot,\cdot\ra_{\HdR})$, and let $\fm^\can_{\HdR}$ be the canonical Maurer--Cartan element in $\dIBL(\HdR)$. 
If $M$ is in addition geometrically formal, then $P$ can be chosen such that
\[
	\fm^{\ana}_P=\fm^{\can}_{\HdR}.
\]
\end{cor}

\section{Comparison of the algebraic and analytic constructions}

As in the previous section, $M$ will be a connected closed oriented manifold of dimension $n$ and $\bigl(\Om\coloneqq\Om^*(M),\dd,\wedge)$ its de Rham complex equipped with the intersection pairing $\la\alpha,\beta\rangle = \int_M\alpha\wedge\beta$.
We denote by $\HdR\coloneqq \HdR(M)$ the de Rham cohomology of $M$.

Suppose that the Poincar\'e $\DGA$ $\Om$ has a differential Poincar\'e duality model $\MM$.
We choose a zig-zag of $\PDGA$ quasi-isomorphisms connecting $\Om$ to $\MM$ and identify the Poincar\'e duality algebras $(\HdR,\dd=0,\wedge,\la\cdot,\cdot\ra_{\HdR})$ and $(H\coloneqq H(\MM),\dd=0,\wedge,\la\cdot,\cdot\ra_H)$ via the induced isomorphism on cohomology.
Consider the canonical Maurer--Cartan element $\fm^\can_\MM$ in $\dIBL(\MM)$, and let ${\ff^{P^\MM}}\colon\dIBL(\MM)\to\dIBL(H)$ be the $\IBLinfty$ homotopy associated to a special propagator $P^\MM\colon\MM\to\MM$.
The pushforward $\ff^{P^\MM}_*\fm^\can_\MM$ then naturally induces a Maurer--Cartan element in $\dIBL(\HdR)$.

Suppose now that $\HdR^1=0$, and let $P\colon\Om \to \Om$ be a special propagator.
Proposition~\ref{prop:existence-PDmodel} provides a \emph{canonical differential Poincar\'e duality model} of $\Om$,  
\[
	\QQ_P\coloneqq\QQ(\SS_P(\Om)),
\]
where $\SS_P$ is the small subalgebra associated to~$P$ and~$\QQ$ denotes the nondegenerate quotient.
We also have the canonical zigzag of $\PDGA$ quasi-isomorphisms~\eqref{eq:canon-zigzag}.
In particular, the Poincar\'e duality algebras~$\HdR$ and $H(\QQ_P)$ are canonically identified.
Moreover, Lemmas~\ref{lem:nondeg-quotient} and~\ref{lem:small-subalgebra} equip~$\SS_P$ and~$\QQ_P$ canonically with special propagators induced from $P$.
Therefore, the construction in the previous paragraph for $\MM=\QQ_P$ provides a canonical Maurer--Cartan element in $\dIBL(\HdR)$.
We denote it by
\[
	\fm^\alg_P
\]
and call it the \emph{algebraic Maurer--Cartan element} in $\dIBL(\HdR)$.

In the next subsections, we will compare $\fm^{\ana}_P$ to $\fm^{\alg}_P$ for a special analytic propagator $P$. Our approach is to use the algebraic and analytic vanishing results to reduce the problem to a comparison of the corresponding cyclic $\Ainfty$ algebras.

\subsection{\texorpdfstring{$\Ainfty$}{A-infinity} algebras and homotopy transfer}\label{ss:hty-transfer}

In this subsection, we recall $\Ainfty$ algebras and their homotopy transfer and prove its functoriality in a special case.

{\bf $\Ainfty$ algebras. }
An {\em $\Ainfty$ algebra} $(\Alg,\{\m_i\})$ consists of a $\Z$-graded vector space~$\Alg$ and a sequence of operations $\m_i\colon\Alg^{\otimes i}\to \Alg$, $i\in\N$, of degrees $\deg \m_i = 2-i$ satisfying  for each $r\in\N$ and $x_1,\dots,x_r\in \Alg$ the relation
\begin{equation}\label{eq:Ainfty}
	\sum_{\substack{i+j=r+1 \\ i,j\geq 1}}\sum_{c=1}^{r+1-j} (-1)^{|x_1|+\dotsb+|x_{c-1}|}\mathfrak m_{i}(x_1,\cdots,\mathfrak m_{j}(x_c,\cdots,x_{c+j-1}),\cdots,x_{r}) = 0, 
\end{equation}
where $|x| = \deg(x) - 1$ is the shifted degree.
In particular, $\dd=\m_1$ is a differential and $x\wedge y=(-1)^{\deg x}\m_2(x,y)$ a product on $\Alg$ which becomes associative on the cohomology $H(\Alg,\dd)$.
In fact, a $\DGA$ $(\Alg,\dd,\wedge)$ corresponds naturally to an $\Ainfty$ algebra $(\Alg,\{\m_1,\m_2,\m_i = 0\text{ for }i\ge 3\})$ and vice versa via  
\begin{equation*}
	\m_1=\dd, \quad \m_2(x,y)=(-1)^{\deg x}x\wedge y.
\end{equation*}
We refer to~\cite{FOOO-I} for the notions of a morphism, quasi-isomorphism, and homotopy equivalence of $\Ainfty$ algebras and their basic properties, such as the following result:

\begin{prop}[Homotopy transfer for $\Ainfty$ algebras,~\cite{Kontsevich-Soibelman}]\label{prop:hty-transfer}
Let $(\Alg,\{\m_i\})$ be an $\Ainfty$ algebra and $B\subset\Alg$ a quasi-isomorphic subcomplex with respect to $\dd=\m_1$.
Given a homotopy operator $P\colon\Alg\to\Alg$ with respect to a projection $\pi\colon\Alg\onto B$,
there is a natural $\Ainfty$ structure~$\{\m^B_i\}$ on $B$ extending $\m_1^B=\dd|_B$ and an $\Ainfty$ homotopy equivalence 
\[
	\bg=\{\bg_i\colon B^{\otimes i}\to \Alg\}\colon(B,\{\m^B_i\}) \longrightarrow (\Alg,\{\m_i\})
\]
extending the inclusion $\bg_1 = \iota\colon B\into\Alg$.
\end{prop}

\begin{proof}[Sketch of proof]
Following \cite{Kontsevich-Soibelman}, the maps $\m^B_i\colon B^{\otimes i}\to B$ and $\bg_i\colon\Alg^{\otimes i}\to B$ for $i\ge 2$ can be explicitly written as
\begin{equation}\label{eq:kontsevich-soibelman}
	\m^B_i \coloneqq \sum_{T\in\TT_{i}}\m^B_T\colon B^{\otimes i}\longrightarrow B\quad\text{and}\quad\bg_i \coloneqq \sum_{T\in\TT_{i}}\bg_T\colon B^{\otimes i}\longrightarrow \Alg,
\end{equation}
where $\TT_{i}$ denotes the set of isotopy classes of planar embeddings of rooted trees with~$i$ \emph{leaves} similarly as in Remark~\ref{rem:ks-eval}.
We orient the edges of $T\in \TT_i$ towards the root and require in addition that each interior vertex has at least two incoming edges.
We also order the leaves in the counterclockwise direction of the plane starting from the root.
We then define $\m_T^B\colon B^{\otimes i}\to B$ and $\bg_T\colon B^{\otimes i}\to \Alg$ by labeling $T$ and interpreting it as a composition rule similarly as in~\eqref{eq:tree-eval}.
Namely, we label each interior vertex with $j$ incoming edges by $\m_j$, each interior edge by~$P$, each leaf by $\iota$, and the root by $\pi$ in the case of $\m_T^B$ and by~$P$ in the case of $\bg_T$.
It is now straightforward to check the desired properties.
\end{proof}

Note that if $\m_i = 0$ for all $i\ge 3$, i.e., if $\Alg$ is a $\DGA$, then the construction of $\m^B_i$ and $\bg_i$ in the proof above involves only binary trees $\TT_i^{\rmbin}\subset \TT_i$. 

Let $P\colon\Alg\to\Alg$ be a special propagator and $\HH\subset\Alg$ the associated harmonic subspace.
Taking $\HH$ for $B$ in Proposition~\ref{prop:hty-transfer} and identifying it with the cohomology $H\coloneqq H(\Alg)$ via the quotient map \eqref{eq:canon-quotient-to-hom} gives an $\Ainfty$ structure~$\{\m^H_i\}$ on~$H$ and an $\Ainfty$ homotopy transfer $\bg\colon (H,\{\m_i^H\})\to (\Alg,\{\m_i\})$.

{\bf Functoriality of $\Ainfty$ homotopy transfer for $\DGA$s. }
Consider a commutative diagram
\[
 \begin{tikzcd}	
	 (\Alg,\dd,\wedge) \ar[rr,"f"] && (\wt\Alg,\wt\dd, \wt \wedge)  \\
	 & (B,\dd_B) \ar[hook', ul, shift left=1ex, "\iota"] \ar[hook,ur,shift right=1ex,"\wt\iota",swap],
 \end{tikzcd}
\]
where $\Alg$ and $\wt \Alg$ are $\DGA$s, $f$ is a $\DGA$ morphism, $B$ is a cochain complex, and the chain maps $\iota$ and $\wt \iota$ are injective.
Suppose that $P\colon\Alg\to\Alg$ and $\wt P\colon\wt \Alg\to\wt \Alg$ are homotopy operators with respect to some projections $\pi\colon\Alg\onto B$ and $\wt \pi\colon\wt\Alg\onto B$, respectively, such that the following compatibility with $f$ is satisfied:
\begin{equation*}
	f\circ P = \wt P \circ f.
\end{equation*}%
Viewing $\Alg$ and $\wt\Alg$ as $\Ainfty$ algebras and $B$ as their subcomplex via the inclusions $\iota$ and $\wt\iota$, respectively, Proposition~\ref{prop:hty-transfer} yields the $\Ainfty$ homotopy transfers
\[
   \bg\colon(B,\m^B)\stackrel{\simeq}\longrightarrow(\Alg,\m),\qquad   \wt\bg\colon( B,\wt\m^B)\stackrel{\simeq}\longrightarrow(\wt \Alg,\wt\m).
\] 
\begin{lem}
	In the setup above, we have 
	\[
		\m^B_i = \wt\m^B_i\quad\text{and}\quad f\circ \bg_i = \wt \bg_i\quad\text{for every }i\in\N.
	\]
	In other words, the following diagram of $\Ainfty$ morphisms commutes:
\[
 \begin{tikzcd}	
	 (\Alg,\m) \ar[r,"f"] & (\wt \Alg,\wt\m) \\
 (B,\m^B) \ar[u,"\bg"] \ar[r,equal] & ( B,\wt\m^B)\,. \ar[u,"\wt\bg"]
 \end{tikzcd}
\]
\end{lem}

\begin{proof}
	For a planar rooted tree $T\in\TT_i$, consider the contributions $\m_T^B$ and $\bg_T$ resp.~$\wt\m_T^B$ and $\wt\bg_T$ from the proof of Proposition~\ref{prop:hty-transfer} which were obtained by labeling~$T$ with the data of $\m, P, \iota, \pi$ resp.~$\wt\m, \wt P, \wt \iota, \wt \pi$ and interpreting it as a composition rule.
	The claim then clearly follows from the commutation relations of this data with~$f$.
\end{proof}

We apply the lemma above in the following two situations:
\begin{enumerate}
	\item $\wt \Alg$ is a $\DGA$ with pairing of Hodge type, $\wt P\colon\wt\Alg\to\wt\Alg$ is a special propagator, $B\coloneqq\HH_{\wt P}\subset\wt \Alg$ is the associated harmonic subspace, 
	$\Alg\coloneqq S_{\wt P}$ is the corresponding small subalgebra, $f\colon\Alg\into\wt \Alg$ is the inclusion, and $P\colon\Alg\to\Alg$ is the induced special propagator from Lemma~\ref{lem:small-subalgebra};
\item $\Alg$ is a $\DGA$ with pairing of Hodge type type such that the induced pairing on cohomology is nondegenerate, $P\colon\Alg\to\Alg$ is a special propagator, $B\coloneqq \HH_P\subset \Alg$ is the associated harmonic subspace, $\wt \Alg\coloneqq\QQ(\Alg)$ is the nondegenerate quotient, $f\colon\Alg\onto\wt \Alg$ is the quotient map, and $\wt P \colon \wt \Alg\to\wt \Alg$ is the corresponding special propagator from Lemma~\ref{lem:nondeg-quotient}.
\end{enumerate}
Combining (i) and (ii), we thus have:

\begin{prop}\label{cor:KS-func}
Let $\Alg$ be a $\DGA$ with pairing of Hodge type such that the induced pairing on cohomology $H\coloneqq H(\Alg)$ is nondegenerate.
Let $P\colon\Alg\to\Alg$ be a special propagator, $\HH\coloneqq\HH_P\subset\Alg$ the associated harmonic subspace, 
$\SS\coloneqq\SS_P$ the associated small subalgebra, and $\QQ\coloneqq \QQ(\SS)$ the nondegenerate quotient of $\SS$.
Consider the commutative diagram 
\[
\begin{tikzcd}
	\Alg & \ar[hook',l,swap,"\iota"] \SS \ar[two heads,r,
	"\pi_Q"] & \QQ \\
	     & \HH, \ar[hook',ul] \ar[hook,u] \ar[hook,ur] &
\end{tikzcd}
\]
where $\pi_Q$ is the quotient map and the other maps are the obvious inclusions.
Viewing the $\DGA$s $\Alg$, $\SS$, $\QQ$ as $\Ainfty$ algebras and $\HH$ as their subcomplex via the vertical maps, Proposition~\ref{prop:hty-transfer} yields the $\Ainfty$ homotopy transfers 
\[
	\bg\colon(\HH,\m^\HH)\stackrel{\simeq}\longrightarrow(\Alg,\m), \quad\! \bg_\SS\colon(\HH,\m^\HH_\SS)\stackrel{\simeq}\longrightarrow(\SS,\m_\SS), \quad\!\bg_\QQ\colon(\HH,\m^\HH_\QQ)\stackrel{\simeq}\longrightarrow(\QQ,\m_\QQ).
\]
The following diagram of $\Ainfty$ morphisms then commutes:
\[
\begin{tikzcd}	
	(\Alg,\m) & (\SS,\m_\SS) \ar[hook',l,"\iota"'] \ar[two heads,r,"\pi"] & (\QQ,\m_\QQ) \\
	(\HH,\m) \ar[u,"\bg"] \ar[r,equal]
	  & (\HH,\m_\SS) \ar[u,"\bg_\SS"] \ar[r,equal]
	  & (\HH,\m_\QQ)\,. \ar[u,"\bg_\QQ"]
\end{tikzcd}
\]
\end{prop}

\begin{cor}\label{cor:comparison-of-homotopy-transfers}
Let $\Alg$ be a $\DGA$ with pairing of Hodge type such that the induced pairing on cohomology $H\coloneqq H(\Alg)$ is nondegenerate.
Let $P\colon\Alg\to\Alg$ be a special propagator and $\m^H$, $\m^H_\SS$ and $\m^H_\QQ$ the three $\Ainfty$ structures on $H$ obtained by the induced $\Ainfty$ homotopy transfers from $\Alg$, $\SS\coloneqq \SS_P(\Alg)$ and $\QQ\coloneqq \QQ(\SS)$, respectively.
Then all these $\Ainfty$ structures on $H$ are equal:
\[
	\m^H=\m^H_\SS=\m^H_\QQ.
\]
\end{cor}

\subsection{Cyclic \texorpdfstring{$\Ainfty$}{A-infinity} algebras and \texorpdfstring{$\IBLinfty$}{IBL-infinity} Maurer Cartan elements}\label{subsec:cyclic}
\strut

{\bf Cyclic $\Ainfty$ algebras. }
A \emph{pairing} on an $\Ainfty$ algebra $(\Alg,\{\m_i\})$ is a bilinear form $\la\cdot,\cdot\ra\colon\Alg\times\Alg\to\R$ such that if we define the \emph{cyclic pairing}
\begin{equation*}
	\la x,y \ra_{\cyc} = (-1)^{\deg x} \la x,y\ra
\end{equation*}
and consider the shifted grading $|x|\coloneqq \deg x -1$, then the \emph{cyclicity condition}
\begin{equation}\label{eq:cyc2}
	\la\mathfrak m_i(x_1,\cdots,x_i),x_0\ra_{\cyc} = (-1)^{|x_0|(|x_1|+\dotsb+|x_i|)} \la\mathfrak m_i(x_0,x_1,\cdots,x_{i-1}),x_i\ra_{\cyc}
\end{equation}
holds for all $i\in \N$.
We call the triple $(\Alg,\{\m_i\},\la\cdot,\cdot\ra)$ an \emph{$\Ainfty$ algebra with pairing}.
Following \cite{Cieliebak-Fukaya-Latschev}, we call an $\Ainfty$ algebra with a perfect pairing $(\Alg,\{\m_i\},\la\cdot,\cdot\ra)$ a \emph{cyclic $\Ainfty$ algebra}.%
\footnote{In \cite{Cieliebak-Fukaya-Latschev} the authors use $\la\cdot,\cdot\ra_\cyc$ instead of $\la\cdot,\cdot\ra$.}

The following proposition can be checked by a straightforward yet nontrivial computation using the expression \eqref{eq:kontsevich-soibelman} for the $\Ainfty$ homotopy transfer in terms of trees.

\begin{prop}[{\cite[Theorem~5.15]{Kajiura}}]
Let $(\Alg,\{\m_i\},\la\cdot,\cdot\ra)$ be an $\Ainfty$ algebra with pairing, and $H\coloneqq H(\Alg,\dd)$ its cohomology.
Let $P\colon\Alg\to\Alg$ be a special propagator and $(H,\{\m_i^H\})$ the associated homotopy transferred $\Ainfty$ algebra.
Then the induced map $\la\cdot,\cdot\ra_H\colon H\times H\to\R$ is a pairing on $(H,\{\m_i^H\})$.
In particular, if $(\Alg,\{\m_i\},\la\cdot,\cdot\ra)$ is a cyclic $\Ainfty$ algebra, then $(H,\{\m_i^H\},\la\cdot,\cdot\ra^H)$ is a cyclic $\Ainfty$ algebra as well.
\end{prop}

{\bf Cyclic $\Ainfty$ algebras and $\IBLinfty$ Maurer--Cartan elements.  }
Following \cite{Cieliebak-Fukaya-Latschev}, we can recast cyclic $\Ainfty$ algebras on a cyclic cochain complex $(\Alg,\dd,\la\cdot,\cdot\ra)$ of degree $n$ in terms of Maurer--Cartan elements in $\dIBL(\Alg)$.
For this, we associate to each linear map $f \colon \Alg^{\otimes i} \to \Alg$ its \emph{cyclization} $f^+ \colon \Alg^{\otimes(i+1)} \to \R$ given by%
\begin{equation}\label{eq:m+}
	f^+(x_0,x_1,\cdots,x_i) = (-1)^{n-2}\la f(x_0,\cdots,x_{i-1}), x_i\ra_{\cyc}.
\end{equation}
Since the pairing is perfect, we have a one-to-one correspondence $f\leftrightarrow f^+$.

\begin{prop}[{\cite[Proposition 12.3]{Cieliebak-Fukaya-Latschev}}]\label{MCAinfty}
Let $(\Alg,\dd,\la\cdot,\cdot\ra)$ be a cyclic cochain complex of degree $n\in\Z$ and $\dIBL(\Alg) = (\dcbc \Alg[2-n], \fp_{1,1,0},\fp_{1,2,0},\fp_{2,1,0})$ the associated $\dIBL$ algebra.
Then a linear map $\m_i\colon \Alg^{\otimes k}\to\Alg$ of degree~$\deg \m_i = 2-i$ satisfies the cyclicity condition~\eqref{eq:cyc2} if and only if $\m^+_i \in \dcbc _{i+1}\Alg$ with $|\m_i^+|=1$, and a collection of such maps $\{\m_i\}_{i\geq 2}$ together with $\m_1\coloneqq\dd$ satisfies the $\Ainfty$ relations~\eqref{eq:Ainfty} if and only if the element
\begin{equation}\label{eq:plus-element}
	\m^+ \coloneqq \sum_{i=2}^\infty \m^+_i\in \dcbc\Alg
\end{equation}
satisfies the first part of the Maurer--Cartan equation for $\{\fm_{1,0}\coloneqq\m^+\}$:
\begin{equation}\label{eq:MC1}
	\fp_{1,1,0}(\m^+) + \frac{1}{2}{\fp}_{2,1,0}(\m^+,\m^+) = 0. 
\end{equation}
\end{prop}

The second part of the Maurer--Cartan equation for $\{\fm_{1,0}\coloneqq\m^+\}$,
\begin{equation}\label{eq:second-MC-equation}
	\fp_{1,2,0}(\m^+)=0,
\end{equation}
is not satisfied for a general cyclic $\Ainfty$ algebra (see \cite[Remark~12.4]{Cieliebak-Fukaya-Latschev} for a discussion).
However, it holds for every cyclic $\DGA$ because $\fp_{1,2,0}(\dcbc_3\Alg)=0$ by definition in Proposition~\ref{prop:structureexists} and because $\m^+=\m_2^+\in \dcbc_3\Alg$ is the triple intersection product.
We now discuss another natural class of $\Ainfty$ algebras for which \eqref{eq:second-MC-equation} holds.

A nonzero element $1\in \Alg^0$ in an $\Ainfty$ algebra $(\Alg,\{\m_i\})$ is called a \emph{strict unit} if
\begin{enumerate}
	\item $\m_2(1,x)=(-1)^{\deg x}\m_2(x,1)= x$, and
	\item $\m_{i}(x_1,\dotsc,x_i)=0$ whenever $x_j = 1$ for some $j\in\{1,\dotsc,i\}$, for all $i\ge 3$.
\end{enumerate}
An $\Ainfty$ algebra which admits a strict unit is called a \emph{strictly unital $\Ainfty$ algebra}.
We define connectedness and $k$-connectedness for a strictly unital $\Alg_\infty$ algebra in the same way as for a unital $\DGA$.
We call a cyclic $\Ainfty$ algebra strictly unital, connected, or $k$-connected if the underlying $\Ainfty$ algebra has the respective property.

\begin{lem}\label{lem:a-infinity-mc}
Let $(\Alg,\{\m_i\},\la\cdot,\cdot\ra)$ be a connected strictly unital cyclic $\Ainfty$ algebra. 
Consider the canonical filtered $\dIBL$ algebra $\dIBL(\Alg)$ associated to the underlying cyclic cochain complex $(\Alg,\dd,\la\cdot,\cdot\ra)$, and let $\m^+\in\dcbc\Alg$ be the element defined in~\eqref{eq:plus-element}.
Then $\fm=\{\fm_{1,0}\coloneqq \m^+\}$ is a Maurer--Cartan element in $\dIBL(\Alg)$.
\end{lem}

\begin{proof}
For $x=x_1^1\dotsb x_{s_1}^1\otimes x_1^2\dotsb x_{s_2}^2\in (\cbc\Alg)^{\otimes 2}$, $s_1, s_2\in\N$, we have
\begin{equation*}	
	\begin{aligned}
		 &\fp_{1,2,0}(\m^+)(x_1^1\dotsb x_{s_1}^1\otimes x_1^2\dotsb x_{s_2}^2)\\
		 &\qquad = \sum_{c=1}^{s_1} \sum_{c^\prime=1}^{s_2} \sum_a \pm \la e_a, \m_{s_1+s_2+1}(x_c^1,\dotsc,x_{c-1}^1,e^a,x^2_{c^\prime}\dotsc x^2_{c^\prime-1})\rangle.
	\end{aligned}
\end{equation*}
The pairing in the sum vanishes for degree reasons unless
\begin{equation*}
	\deg(x) = - \deg(\m_{s_1+s_2+1}) =  s_1 + s_2 - 1.
\end{equation*}
But then some $x_i^j$ must have degree $0$, hence be a multiple of $1$ by connectedness on~$\Alg$, and the right hand side vanished by strict unitality since $s_1+s_2+1\geq 3$. 
Therefore, the second Maurer--Cartan equation \eqref{eq:second-MC-equation} is satisfied.%
\footnote{We used here only condition (ii) of strict unitality as $\fp_{1,2,0}(\dcbc_3\Alg)=0$ holds trivially.
For a \emph{unital version} of our construction including a term $\dcbc_0 \Alg \coloneqq \R\cdot 1$, condition (i) would also be needed.}
The first equation~\eqref{eq:MC1} is satisfied by Proposition~\ref{MCAinfty}.
\end{proof}

We assumed so far that we were given an $\Ainfty$ algebra first.
Suppose now that we are given a cyclic cochain complex $(A,\dd,\la\cdot,\cdot\ra)$ and a Maurer--Cartan element $\fm=\{\fm_{\ell,g}\}$ in $\dIBL(\Alg)$ instead.
Proposition~\ref{MCAinfty} then implies that the collection of linear maps $\{\m_i\colon\Alg^{\otimes i}\to\Alg\}_{i\ge 2}$ corresponding via~\eqref{eq:m+} to $\fm_{1,0}\in \dcbc\Alg$ together with $\m_1=\dd$ and the pairing $\la\cdot,\cdot\ra$ constitutes a cyclic $\Ainfty$ structure on $\Alg$.
Notice that the filtration degree condition from the definition of a Maurer--Cartan element implies $\fm_{1,0}(x)=\fm_{1,0}(x_1,x_2)=0$, so~$\fm_{1,0}$ does not contain any operation with less than three inputs.

{\bf Kontsevich--Soibelman Maurer--Cartan element. }

\begin{prop}\label{prop:ks-maurer-cartan}
Let $(\Alg,\dd,\wedge,\la\cdot,\cdot\ra)$ be a nonnegatively graded unital $\DGA$ with pairing, and let $B\subset \Alg$ be a quasi-isomorphic cyclic subcomplex such that~$B^0=\R\cdot 1$.
Let $P\colon\Alg\to\Alg$ be a special propagator with respect to a projection $\pi\colon\Alg\onto B$, and let $\m^{B+}\in \dcbc B$ be the element associated in Proposition~\ref{MCAinfty} to the homotopy transferred $\Ainfty$ algebra $(B,\{\m_i^B\})$ associated to~$P$.
Then 
\[
	\fm\coloneqq \{ \fm_{1,0}\coloneqq \m^{B+}\}
\]
is a Maurer--Cartan element in $\dIBL(B)$.
\end{prop}

\begin{proof}
We will show that $1$ is a strict unit for the $\Ainfty$ algebra $(B,\{\m_i^B\})$.
The proposition then follows from Lemma~\ref{lem:a-infinity-mc}.  
Let $i\ge 2$ and $b_1,\dotsc,b_i\in B$, and consider the formula~\eqref{eq:kontsevich-soibelman} for $\m_i^B(b_1,\dotsc,b_i)$ as a sum of contributions $\m_T^B(b_1,\dotsc,b_i)$ of planar binary trees $T\in\TT_i^{\rmbin}$ with interior edges labeled by~$P$, interior vertices labeled by $\wedge$, and leaves labeled by $b_1,\dotsc, b_i\in B$.
For $i=2$, we have $\m^B_2(b_1,b_2) = (-1)^{\deg b_1} b_1 \wedge b_2$, hence condition (i) of a strict unit in $(B,\{\m_i^B\})$ for $1$ follows from it being a unit in $(\Alg,\dd,\wedge)$. 
For $i\geq 3$, suppose that some $b_j$ equals $1$ and consider the interior vertex $v$ adjacent to the exterior edge labeled with $b_j=1$. Denote by $e^{\mathrm{in}}$ the other incoming edge and by $e^{\mathrm{out}}$ the outgoing edge at $v$. Now there are $3$ cases.
If both $e^{\mathrm{in}}$ and $e^{\mathrm{out}}$ are interior edges, then they are labeled with $P$ and $\m_T^B(b_1,\dotsc,b_i)=0$ because $P\circ P=0$. 
If $e^{\mathrm{in}}$ is exterior and $e^{\mathrm{out}}$ is interior, then $\m_T^B(b_1,\dotsc,b_i)=0$ because $P\circ\pi=0$. 
If $e^{\mathrm{in}}$ is interior and $e^{\mathrm{out}}$ is exterior (hence the root edge), then $e^{\mathrm{out}}$ is labeled with $\pi$ and $\m_T^B(b_1,\dotsc,b_i)=0$ because $\pi\circ P=0$. This proves condition (ii) of the strict unitality of $(B,\{\m_i^B\})$.
\end{proof}

We call $\fm$ from Proposition~\ref{prop:ks-maurer-cartan} the \emph{Kontsevich--Soibelman Maurer--Cartan element} in $\dIBL(B)$ and denote it by 
\[
	\fm^{\ks}_P.
\]
Now we will compare $\fm^\ks_P$ to $\ff^P_*\fm^\can_\Alg$ in the algebraic case for a cyclic $\DGA$ $\Alg$, and~$\fm^\ks_P$ to~$\fm^\ana_P$ in the analytic case for the de Rham complex $\Alg=\Om$.
The following lemmas are clear on the pictorial level and proved in detail in~\cite{Volkov-habil}:

\begin{lem}[\cite{Volkov-habil}]\label{lem:alg=KS}
Let $(\Alg,\dd,\wedge,\la\cdot,\cdot\ra)$ be a nonnegatively graded unital cyclic $\DGA$ with connected cohomology $H\coloneqq H(\Alg)$.
Consider the canonical filtered $\dIBL$ algebras $\dIBL(\Alg)$ and $\dIBL(H)$ and the canonical Maurer--Cartan element $\fm_\Alg^\can$ in $\dIBL(\Alg)$.
Let $P\colon\Alg\to\Alg$ be a special propagator, and let $\ff^P\colon \dIBL(\Alg)\to\dIBL(H)$ be the associated $\IBLinfty$ homotopy equivalence and~$\fm_P^{\ks}$ the associated Kontsevich--Soibelmann Maurer--Cartan element in $\dIBL(H)$.
Then we have
\begin{equation}\label{eq:KontsCFL}
	(\ff^P_*\fm^\can_\Alg)_{1,0} = (\fm^\ks_P)_{1,0}.
\end{equation}
\end{lem}

\begin{proof}[Sketch of proof]
Let $\alpha=\alpha_1^1\dotsb\alpha_{s}^1\in \cbc_{s} H$ with $s\ge 3$.
Evaluated on $\alpha$, the left-hand side of \eqref{eq:KontsCFL} can be written as a sum over $\Gamma\in RG_{s-2,1,0}^3$ of contributions $(\ff^P_*\fm^\can_\Alg)_\Gamma(\alpha)\in\R$ by Corollary~\ref{cor:alg-MC}.
The right-hand side can be written as a sum over $T\in \TT_{s-1}^{\rmbin}$ of contributions~$\m_{T}^{H+}(\alpha)\in\R$ as in the proof of Proposition~\ref{prop:hty-transfer}. 
Identifying the marked exterior vertex of $\Gamma$ with the root of $T$ we see that
\[
	RG_{s-2,1,0}^3\simeq \TT_{s-1}^{\rmbin}.
\]
It remains to compare the contributions of $\Gamma \in RG_{s-2,1,0}^3$ and the corresponding $T\in\TT_{s-1}^{\rmbin}$.
We orient the edges of $\Gamma$ towards the root and order the leaves in the positive direction of the boundary starting from the root.
We choose an ordering of interior vertices and a basis $(e_i)$ of $\Alg$ and consider the coordinate expression \eqref{eq:contribution-of-graph}.
Performing the sum in \eqref{eq:contribution-of-graph} iteratively starting at the leaves and using
\[
	 P e^i = \sum_j P^{ij} e_j\quad\text{and}\quad x = \sum_i \la x, e^i \ra e_i = \sum_i \la e_i, x\ra e^i
\]
gives easily
\[
	(\ff^P_*\fm^\can_\Alg)_T(\alpha) = \pm \la \alpha_1^1,\m_T^H(\alpha_2^1,\dotsc,\alpha_s^1)\ra=\pm\m_T^{H+}(\alpha).
\]
A careful comparison of signs and combinatorial coefficients finishes the proof.
\end{proof} 

\begin{lem}[\cite{Volkov-habil}]\label{lem:ana=KS}
Let $M$ be a connected closed oriented manifold.
Let $P\colon\Om\to\Om$ be a special analytic propagator, and let $\fm^\ana_P$ and $\fm^\ks_P$ be the associated analytic and Kontsevich--Soibelmann Maurer--Cartan elements in $\dIBL(\HdR)$, respectively.
Then we have
\begin{equation*}%
	(\fm^\ana_P)_{1,0} = (\fm^{\ks}_P)_{1,0}.
\end{equation*}
\end{lem}

\begin{proof}[Sketch of proof]
The proof is the same as the proof of Lemma~\ref{lem:alg=KS} on the structural level.
Technically, instead of performing the sum in \eqref{eq:contribution-of-graph} iteratively one has to apply Fubini's theorem for the compactified configuration spaces $X_\Gamma$ in~\eqref{eq:n}.
\end{proof}

\subsection{Comparison of the algebraic and analytic Maurer Cartan elements}

The main result of this section is the following theorem which immediately implies Theorem~\ref{thm:comparison-intro} in the Introduction:

\begin{thm}\label{thm:comparison}
Let $M$ be a connected closed oriented manifold with $\HdR^1=0$ which is not diffeomorphic to $S^2$.
For a special analytic propagator $P\colon\Om\to\Om$, consider the following Maurer--Cartan elements in $\dIBL(\HdR)$ associated to $P$: the analytic Maurer--Cartan element~$\fm^{\ana}_P$, the algebraic Maurer--Cartan element~$\fm^{\alg}_P$, and the Kontsevich--Soibelman Maurer--Cartan element~$\fm^{\ks}_P$. 
Then all three Maurer--Cartan elements are equal:
\[
	\fm^{\ana}_P=\fm^{\alg}_P = \fm^{\ks}_P. 
\]
\end{thm}

\begin{proof}
Suppose that $n= \dim(M)\ge 3$ (the cases $n\in\{1,2\}$ are impossible due to the assumptions, and the case $n=0$ is trivial).
The algebraic vanishing result from Corollary~\ref{cor:alg-vanishing} implies that $(\fm^{\alg}_P)_{\ell,g} = 0$ for all $(\ell,g)\neq (1,0)$.
The analytic vanishing result from Corollary~\ref{cor:ana-vanishing} implies that $(\fm^{\ana}_P)_{\ell,g}=0$ for all $(\ell,g)\neq (1,0)$.
Therefore, it remains to prove that
\[
	(\fm_P^\ana)_{1,0} = (\fm^{\ks}_P)_{1,0} = (\fm_P^\alg)_{1,0}.
\]
The first equality is Lemma~\ref{lem:ana=KS}.
As for the second one, consider the nondegenerate quotient $\QQ_P\coloneqq\QQ(\SS_P(\Om))$ of the small subalgebra 
$\SS_P\subset \Om$, and identify $\HdR\simeq H\coloneqq H(\QQ_P)$ via the canonical zig-zag~\eqref{eq:canon-zigzag}.
Recall the definition $(\fm^{\alg}_P)_{1,0}=\ff^P_*\fm^\can_{\QQ_P}$, where $\ff^P\colon\dIBL(\QQ_P)\to \dIBL(\HdR)$ is the $\IBL_\infty$ homotopy equivalence associated to~$P$ and $\fm^\can_{\QQ_P}$ is the canonical Maurer--Cartan element in $\dIBL(\QQ_P)$.
Lemma~\ref{lem:alg=KS} asserts that $(\ff^P_*\fm^\can_{\QQ_P})_{1,0}=\m^{\HdR+}_{\QQ_P}$, where $\m^{\HdR +}_{\QQ_{P}}\in\dcbc \HdR$ is associated via Proposition~\ref{MCAinfty} to the $\Ainfty$ homotopy transfer $\QQ_P\leadsto \HdR$ from Proposition~\ref{prop:hty-transfer}. 
On the other hand, $(\fm^{\ks}_P)_{1,0}=\m^{\HdR+}$ is associated via Proposition~\ref{MCAinfty} to the $\Ainfty$ homotopy transfer $\Om\leadsto \HdR$ from Proposition~\ref{prop:hty-transfer}.
The comparison of $\Ainfty$ homotopy transfers $\Om\leadsto \HdR$ and $\QQ_P\leadsto \HdR$ in Corollary~\ref{cor:comparison-of-homotopy-transfers} gives $\m^{\HdR+} = \m^{\HdR+}_{\QQ_P}$, and the theorem follows.
\end{proof}

We summarize the situation for a connected closed oriented manifold $M$.
We associate to $M$ up to $\IBLinfty$ homotopy equivalence the following $\IBLinfty$ algebras whose homology is the cyclic cohomology of $(\Om,\dd,\wedge)$:
\begin{itemize}
	\item the $\IBLinfty$ algebra $\dIBL^{\fm^\ana}(\HdR)$ based on ribbon graphs of all genera;
	\item the $\dIBL$ algebra $\dIBL^{\fm^\ks}(\HdR)$ based on ribbon trees only.
\end{itemize}
If $H^1_{\rm dR}=0$ and $M$ is not diffeomorphic to $S^2$, then Theorem~\ref{thm:comparison} implies that these structures for a special analytic propagator $P$ are equal.
A computation for $M=S^1$ in~\cite{Pavel-thesis} shows on the one hand that
\[
	\fq_{1,2,0}^{\fm^\ana_P} \neq \fq^{\fm^\ks_P}_{1,2,0}=\fq_{1,2,0},
\]
and on the other hand that the $\IBL$ structures induced on homology are equal.
In fact, for any~$M$ we have $H(\dcbc \HdR,\fq_{1,1,0}^{\fm^\ks}) = H(\dcbc \HdR,\fq_{1,1,0}^{\fm^\ana})$ by Lemma~\ref{lem:ana=KS}, so it is plausible that there should be an $\IBLinfty$ homotopy equivalence 
$\dIBL^{\fm^\ana}(\HdR)\simeq \dIBL^{\fm^\ks}(\HdR)$ in the general case.%
\footnote{
This conjecture is supported by ongoing work of the second author relating $\fm^\ana_{\ell,g}(\alpha)$ to integrals of certain forms over the moduli space of metric ribbon graphs considered in \cite{Costello2007}.}
One can also ask the stronger question whether $\fm^\ana$ and $\fm^\ks$ are gauge equivalent.%

\subsection{Formality}\label{ss:formality}

In rational homotopy theory (see, e.g., \cite{Sullivan-infinitesimal}),  a manifold $M$ is called {\em formal} if its de Rham complex $\Om$ is weakly equivalent as a $\CDGA$ to its de Rham cohomology~$\HdR$, i.e., if there exists a zigzag of $\CDGA$ quasi-isomorphisms connecting $\Om$ to $\HdR$. 

Similarly, we say that $M$ is \emph{$\IBLinfty$ formal} if the $\IBLinfty$ algebra $\dIBL^{\fm^\ana}(\HdR)$ from Theorem~\ref{thm:CV-MC} is weakly equivalent as an $\IBLinfty$ algebra to the $\dIBL$ algebra $\dIBL^{\fm_{\HdR}^\can}(\HdR)$ from Proposition~\ref{propIBLI2}.%
\footnote{One could also define a notion of formality as a weak equivalence of $\dIBL^{\fm^\ana}(\HdR)$ and its homology which is an $\IBL$ algebra. However, we do not see a geometric motivation for such a definition.}
Recall that two $\IBLinfty$ algebras are weakly equivalent if and only if they are $\IBLinfty$ homotopy equivalent, i.e, if there exists an $\IBLinfty$ homotopy equivalence in either direction. The following result corresponds to Corollary~\ref{cor:formal-intro} in the Introduction:

\begin{cor}[Formality implies $\IBLinfty$ formality]\label{cor:ibl-formality}
Let $M$ be a closed connected oriented manifold such that $\HdR^1=\HdR^2=0$.
Then $M$ being formal implies $M$ being $\IBLinfty$ formal.
\end{cor}

\begin{proof}
Let $P\colon\Om\to\Om$ be a special analytic propagator.
We consider the following three Maurer--Cartan elements in $\dIBL(\HdR)$: the analytic Maurer--Cartan element $\fm^\ana_P$ associated to $P$, the algebraic Maurer--Cartan element $\fm^\alg_P$ associated to $P$, and the canonical Maurer--Cartan element $\fm^\can_{\HdR}$.
Theorem~\ref{thm:comparison} asserts that $\fm^\ana_P=\fm^\alg_P$.
Let $\QQ_P\coloneqq\QQ(\SS_P(\Om))$ be the nondegenerate quotient of the small subalgebra $\SS_P\subset \Om$ associated to~$P$.
By construction (cf.~the proof of Theorem~\ref{thm:comparison}), there is an $\IBLinfty$ homotopy equivalence $\dIBL^{\fm^\alg_P}(\HdR)\simeq \dIBL^{\fm^\can_{\QQ_P}}(\QQ_P)$.  
Recall from Proposition~\ref{prop:existence-PDmodel} that $\QQ_P$ is a differential Poincar\'e duality model of~$\Om$.
On the other hand, if $M$ is formal, then~$\HdR$ is also a differential Poincar\'e duality model of~$\Om$. 
Indeed, a simple argument from~\cite[Proposition 6.2.5]{Pavel-thesis} shows that a zig zag of $\CDGA$ quasi-isomorphisms $f\colon\Om\rightsquigarrow \HdR$ induces a zig-zag of $\PDGA$ quasi-isomorphisms $\wt f\colon\Om\rightsquigarrow \HdR$ (one orients the cohomologies of the intermediate $\CDGA$s so that they become $\PDGA$s and the maps between them $\PDGA$ quasi-isomorphisms, and composes $f$ with the inverse of the induced isomorphism on cohomology $H(f)^{-1}\colon \HdR\to \HdR$ to obtain $\wt f$).
Theorem~\ref{thm:alg-construction} applied to~$\QQ_P$ and~$\HdR$ under the assumption that $\HdR^2=0$ then implies the existence of an $\IBLinfty$ homotopy equivalence $\dIBL^{\fm^\can_{\QQ_P}}(\QQ_P)\simeq\dIBL^{\fm^{\can}}(\HdR)$, and the corollary follows.
\end{proof}

\begin{remark}
Consider a compact connected Lie group $G$ with $\HdR^1(G)=0$.
Then~$G$ is formal and satisfies $\HdR^2(G)=0$, hence it is $\IBLinfty$ formal by Corollary~\ref{cor:ibl-formality}.
In fact, $G$ is even geometrically formal, so Corollary~\ref{cor:ana-vanishing} implies the stronger assertion that there exists a special analytic propagator~$P\colon\Om\to\Om$ such that $\fm^\ana_P = \fm^\can_{\HdR}$ in $\dIBL(\HdR)$.
Finally, $G$ is also simply connected, hence $\dIBL^{\fm^\can_{\HdR}}(\HdR)$ induces a chain model for the equivariant string topology of $G$ by Theorem~\ref{thm:Naef-Willwacher}.
\end{remark}

\subsection{Relation to Massey products}\label{ss:Massey}

Massey products are secondary cohomology operations that are trivial for formal manifolds, and thus give obstructions to formality. See~\cite{Deligne-Griffiths-Morgan-Sullivan} for some examples of non-formal manifolds. In this subsection we discuss the following question which would constitute a partial converse to Corollary~\ref{cor:ibl-formality}:

\begin{question}\label{q:Massey}
If a closed connected oriented manifold $M$ is $\IBLinfty$ formal, are then all its Massey products trivial? 
\end{question}

{\bf Massey products. }
We begin by recalling some facts about Massey products, following the presentation in~\cite{Kraines}. Let $(A,d)$ be a DGA and $H_A$ its cohomology. Let us first describe the triple Massey product. 
Consider three cycles $a_1,a_2,a_3\in A$ of homogeneous degree such that $a_1a_2$ and $a_2a_3$ are exact. For each choice of primitives $b,c\in A$ with
\begin{equation*}
  db=a_1a_2, \qquad dc=a_2a_3,
\end{equation*}
the element
\begin{equation*}
   ba_3+(-1)^{\deg a_1+1}a_1c\in A^{\deg a_1+\deg a_2+\deg a_3-1}
\end{equation*}
is closed. We define
\begin{equation*}
  \la a_1,a_2,a_3\ra \subset H_A^{\deg a_1+\deg a_2+\deg a_3-1}
\end{equation*}
as the set of all cohomology classes $[ba_3+(-1)^{\deg a_1+1}a_1c]$ for primitives $b,c$.
One easily checks that this is an affine space over $[a_1]H_A+H_A[a_3]$ which depends only on the cohomology classes $u_i=[a_i]$ of the $a_i$. The set
\begin{equation*}
  \la u_1,u_2,u_3\ra := \la a_1,a_2,a_3\ra \subset H_A^{\deg u_1+\deg u_2+\deg u_3-1}
\end{equation*}
is called the {\em triple Massey product} of the cohomology classes $u_1,u_2,u_3$ with $u_1u_2=u_2u_3=0$.\footnote{
The triple Massey product is often defined as an element of the quotient space $H_A/(u_1H_A+H_Au_3)$, but the definition as a set is more suitable for the generalization to higher Massey products.}

More generally, for each $k\geq 3$ and homogeneous cohomology classes $u_1,\dots,u_k\in H_A$ satisfying suitable conditions one obtains {\em Massey products} 
\begin{equation*}
  \la u_1,\dots,u_k\ra \subset H_A^{\deg u_1+\dots+\deg u_3+2-k}
\end{equation*}
with the following properties:
\begin{enumerate}
\item If $(A,d)$ has vanishing differential $d=0$, then all Massey products are trivial in the sense that $0\in\la u_1,\dots,u_k\ra$. 
\item If $f:A\to B$ is a DGA quasi-isomorphism, then
$$
   f_*\la u_1,\dots,u_k\ra = \la f_*u_1,\dots,f_*u_k\ra.
$$
\end{enumerate}
Let us emphasize that, although the Massey products are subsets of $H_A$, they depend on the DGA $A$ and not just on its homology. 

The following proposition allows us to extend property (ii) to $A_\infty$ morphisms.

\begin{proposition}\label{prop:rectify}
Two DGAs $A$ and $B$ are $A_\infty$ homotopy equivalent if and only if they can be connected by a zigzag of DGA quasi-isomorphisms.
\end{proposition}

\begin{proof}
We can describe DGAs as dg algebras over the nonsymmetric operad As, and $A_\infty$ algebras as infinity dg algebras
over the operad As, see~\cite[Chapter 10]{Loday-Vallette}. The result now follows immediately from~\cite[Theorem 11.4.9]{Loday-Vallette}.
\end{proof}

Proposition~\ref{prop:rectify} together with the discussion above yields

\begin{lemma}\label{lem:masseyainf}
Let $A$ and $B$ be two DGAs which are $A_\infty$ homotopy equivalent. If all Massey products are trivial on $A$, then so they are on $B$. \qed
\end{lemma}

Let now $M$ be a closed connected oriented manifold. By the discussion above, its de Rham complex $\Om^*(M)$ induces Massey products on its de Rham cohomology $\HdR(M)$ which we call the {\em Massey products on $M$}. 
The homotopy transfer in Proposition~\ref{prop:hty-transfer} yields an $A_\infty$ structure $\{\m^H_k\}_{k\ge 1}$ on $\HdR(M)$ which is $A_\infty$ homotopy equivalent to $\Om^*(M)$. 
On the other hand, we can consider $(\HdR(M),d=0,\wedge)$ as a DGA with trivial differential, and thus trivial Massey products by property (i) above. Therefore, Lemma~\ref{lem:masseyainf} implies

\begin{cor}\label{cor:massenontriv}
Let $M$ be a closed connected oriented manifold. If the $A_\infty$ algebra $(\HdR(M),\{\m_k^H\}_{k\ge 1})$ is $A_\infty$ homotopy equivalent to the DGA $(\HdR(M),d=0,\wedge)$, then all Massey products on $M$ are trivial. \qed
\end{cor}

Let us now abbreviate $H:=\HdR(M)$ and choose a special analytic propagator $P$ which we will suppress from the notation. By Proposition~\ref{prop:ks-maurer-cartan}, the $A_\infty$ structure $\m^H=\{\m^H_k\}_{k\ge 1}$ gives rise to an element $\m^{H+}\in B^{\rm cyc*}H$ and to the Kontsevich-Soibelman Maurer-Cartan element $\fm^\ks$ with $\fm^\ks_{1,0}=\m^{H+}$ and all other components zero. In view of Lemma~\ref{lem:ana=KS} we have
\begin{equation*}%
	\fm^\ana_{1,0} = \fm^{\ks}_{1,0}=\m^{H+}
\end{equation*}
for the analytic Maurer-Cartan element $\fm^\ana$ from Theorem~\ref{thm:CV-MC}. 
On the other hand, by Proposition~\ref{propIBLI2} the DGA $(\HdR(M),d=0,\wedge)$ gives rise to the canonical Maurer-Cartan element $\fm^\can$ on $B^{\rm cyc*}H$ with $\fm^\can_{1,0}=\m_2^+$ and all other components zero. 
IBL formality of $M$ means $\IBLinfty$ homotopy equivalence of the twisted $\IBLinfty$ structures on $B^{\rm cyc*}H$ defined by $\fm^\ana$ and $\fm^\can$. In view of Corollary~\ref{cor:massenontriv}, Question~\ref{q:Massey} therefore comes down to understanding whether this condition implies $A_\infty$ homotopy equivalence between the $A_\infty$ structures encoded by the $(1,0)$ components of $\fm^\ana$ and $\fm^\can$.

\subsection{Reduced version}

Let $(\Alg,\dd,\wedge,\la\cdot,\cdot\ra)$ be a unital cyclic $\DGA$ of degree $n\in\Z$ and $\dIBL(\Alg)=(\dcbc\Alg[2-n],\fp=\{\fp_{1,1,0},\fp_{2,1,0},\fp_{1,2,0}\})$ the canonical filtered $\dIBL$ algebra associated in Proposition~\ref{prop:structureexists} to the underlying cyclic cochain complex $(\Alg,\dd,\la\cdot,\cdot\ra)$.  
We define the \emph{reduced dual cyclic bar complex}~of $\Alg$ by
\[
	\ol{\dcbc}\Alg \coloneqq \bigl\{\varphi\in\dcbc\Alg \mid \varphi(1\dotsb) = 0\bigr\}\subset\dcbc\Alg.
\]
The operations $\fp_{2,1,0}, \fp_{1,2,0}$ defined by \eqref{eq:mu}, \eqref{eq:delta} clearly restrict to $\ol{\dcbc}\Alg$, and so does $\fp_{1,1,0}=\dd^*$ because $\dd(1)=0$.
The restrictions of $\fp_{1,1,0},\fp_{2,1,0},\fp_{1,2,0}$ then define a filtered $\dIBL$ structure of degree $(n-3)$ on $\ol{\dcbc}\Alg[2-n]$.
We will denote the corresponding filtered $\dIBL$ algebra by 
\[
	\ol{\dIBL}(\Alg)\coloneqq\bigl((\rdcbc\Alg)[2-n],\fp=\{\fp_{1,1,0},\fp_{2,1,0},\fp_{1,2,0}\}\bigr).
\]
We call a Maurer--Cartan element $\fm=\{\fm_{\ell,g}\}$ in $\dIBL(\Alg)$ \emph{strictly unital} if:
\begin{enumerate}
	\item the cyclic $\Ainfty$ algebra $(\Alg,\{\m_i\},\la\cdot,\cdot\ra)$ corresponding to $\fm_{1,0}$ via Proposition~\ref{MCAinfty} is \emph{strictly unital} with unit $1$ (see subsection \ref{subsec:cyclic} for the definition);%
	\item we have $\fm_{\ell,g}(1\dotsb) = 0$ for all $(\ell,g)\neq (1,0)$.
\end{enumerate}%
The following was observed in \cite{Pavel-thesis}:

\begin{lemma}\label{lem:strictly-reduced}
	Let $(\Alg,\dd,\wedge,\la\cdot,\cdot\ra)$ be a unital cyclic $\DGA$ of degree $n\in\Z$, and let $\fm=\{\fm_{\ell,g}\}$ be a strictly unital Maurer--Cartan element in $\dIBL(\Alg)=\bigl(\dcbc\Alg[2-n],\fp=\{\fp_{1,1,0},\fp_{2,1,0},\fp_{1,2,0}\}\bigr)$.
	Then the twisted operations 
	\[
		\fp^\fm_{1,\ell,g}\colon \wh{E}_1(\dcbc\Alg)[2-n]\longrightarrow \wh{E}_{\ell}(\dcbc\Alg)[2-n]\quad\text{for }\ell\ge 1, g\ge 0
	\]
	restrict to operations $\wh{E}_1(\ol{\dcbc}\Alg)[2-n]\to \wh{E}_{\ell}(\ol{\dcbc}\Alg)[2-n]$ which together with the restriction of~$\fp_{2,1,0}$ define an $\IBLinfty$ structure of degree $(n-3)$ on $\ol{\dcbc}\Alg[2-n]$.
\end{lemma}

\begin{proof}
	The first two twisted operations can be written as
	\begin{equation*}
		\begin{aligned}
			\fp_{1,1,0}^\fm &= \fp_{1,1,0} + \fp_{2,1,0}\circ_1\fm_{1,0},\\
			\fp_{1,2,0}^\fm &= \fp_{1,2,0} + \fp_{2,1,0}\circ_1\fm_{2,0},
		\end{aligned}
	\end{equation*}
	and the others as 
	\begin{equation*}
		\fp_{1,\ell,g}^\fm = \fp_{2,1,0}\circ_1\fm_{\ell,g}.
	\end{equation*}
	Given $\varphi\in\wh{E}_1(\dcbc \Alg)[2-n]$ and $\alpha=\alpha^1\otimes\dotsb\otimes\alpha^\ell$, $\alpha^i\in \cbc\Alg$, we have
	\begin{equation}\label{eq:twist-0}
		(\fp_{2,1,0}\circ_1\fm_{\ell,g})(\varphi)(\alpha) = \sum_{j=1}^\ell\sum\pm\fm_{\ell,g}^{(1)}(\alpha^{1})\dotsb \fp_{2,1,0}(\fm_{\ell,g}^{(j)},\varphi)(\alpha^{j})\dotsb \fm_{\ell,g}^{(\ell)}(\alpha^{\ell}),
	\end{equation}%
	where we used the following Sweedler's notation for $\fm_{\ell,g}\in\wh{E}_{\ell}(\dcbc\Alg)[2-n]$:
	\[
		\fm_{\ell,g} = \sum \fm_{\ell,g}^{(1)}\wh{\otimes}\dotsb\wh{\otimes}\fm_{\ell,g}^{(\ell)}.
	\]
	Assumption (ii) on $\fm$ gives
	\[
		\fm_{\ell,g}^{(j)}\in\ol{\dcbc}\Alg\quad\text{for all }j\in\{1,\dots,\ell\}, (\ell,g)\neq(1,0),
	\]
	which implies that $\fp_{1,\ell,g}^\fm$ restricts to $\wh{E}_1(\ol{\dcbc}\Alg)[2-n]\to \wh{E}_{\ell}(\ol{\dcbc}\Alg)[2-n]$ for all $(\ell,g)\neq (1,0)$. 
	Given $\varphi\in\wh{E}_1\ol{\dcbc}\Alg$ and $x_2,\dotsc,x_i\in\Alg$, assumption~(i) implies
	\begin{align*}
		&(\fp_{2,1,0}\circ_1\fm_{1,0})(\varphi)(1x_2\dotsb x_i) \\
		&\qquad= \pm \fp_{2,1,0}(\fm_{1,0},\varphi)(1x_2\dotsb x_i)\\
		&\qquad= \pm\bigl(\varphi(\m_2(1,x_2)x_3\dotsb x_i) + (-1)^{|x_i|}\varphi(x_2\dotsb\m_2(x_i,1))\bigr)=0,
	\end{align*}
	hence $\fp^\fm_{1,1,0}$ restricts to $\wh{E}_1(\ol{\dcbc}\Alg)[2-n]\to\wh{E}_1(\ol{\dcbc}\Alg)[2-n]$ as well.
\end{proof}

We denote the $\IBLinfty$ algebra from the previous lemma by
\[
	\ol{\dIBL^\fm}(\Alg)\coloneqq\bigl((\ol{\dcbc}\Alg)[2-n],\fp^\fm=\{\fp^\fm_{2,1,0},\fp_{1,\ell,g}^\fm\text{ for }\ell\ge 1, g\ge 0\}\bigr).
\]
A useful observation from \cite{Pavel-thesis} is that if $\Alg$ is simply connected, then we have 
\[
	\wh{E}_{\ell}\ol{\dcbc}\Alg[2-n] \simeq E_{\ell}\ol{\dcbc}\Alg[2-n]\quad\text{for all }\ell\in\N_0.
\]
The following Maurer--Cartan elements are strictly unital:
\begin{enumerate}
\item the canonical Maurer--Cartan element $\fm^\can_\Alg$ in the canonical filtered $\dIBL$ algebra $\dIBL(\Alg)$ associated to a unital cyclic $\DGA$ $\Alg$;
\item the pushforward $\ff^P_*\fm^\can_\Alg$ in $\dIBL(B)$ for a quasi-isomorphic cyclic cochain subcomplex $B\subset \Alg$ of a nonegatively graded unital cyclic $\DGA$ 
$\Alg$ satisfying $B^0=\R\cdot 1$ along an $\IBLinfty$ homotopy $\ff^P\colon\dIBL(\Alg)\to\dIBL(B)$ associated to a special propagator $P\colon\Alg\to\Alg$ with respect to a projection $\Alg\onto B$ (by the positivity of degrees in Propositions~\ref{prop:alg-vanishing}(i)). 
\item the Kontsevich--Soibelman Maurer--Cartan element $\fm^\ks$ in $\dIBL(B)$ for a quasi-isomorphic cyclic cochain subcomplex $B\subset\Alg$ of a unital $\DGA$ with pairing $\Alg$ satisfying $B^0=\R\cdot 1$.
\end{enumerate}

\begin{prop}
	Let $(\Alg,\dd,\wedge,\la\cdot,\cdot\ra)$ be a nonegatively graded unital cyclic $\DGA$, and let $B\subset\Alg$ be a quasi-isomorphic cyclic subcomplex such that $B^0=\R\cdot 1$ and $B^1 = 0$.
	Let $P\colon\Alg\to\Alg$ be a special propagator with respect to a projection $\Alg\onto B$, and let $\ff^P\colon\dIBL(\Alg)\to\dIBL(B)$ be the associated $\IBLinfty$ homotopy equivalence. 
	Consider the canonical Maurer--Cartan element $\fm^\can_\Alg$ in $\dIBL(\Alg)$ and its pushforward $\ff^P_*\fm^\can_\Alg$ in $\dIBL(B)$.
	Let $\fm^\ks_P$ be the Kontsevich--Soibelman Maurer--Cartan element in $\dIBL(B)$ associated to $P$.
	Then we have:
	\[
		\ol{\dIBL^{\ff^P_*\fm^\can_\Alg}}(B) = \ol{\dIBL^{\fm_P^\ks}}(B).
	\]
\end{prop}

\begin{proof}
	For $n\ge 3$, the vanishing results in Proposition~\ref{prop:alg-vanishing} imply that $\ff^P_*\fm^\can_\Alg = \fm_P^\ks$.
	The case $n=1$ is impossible and the case $n=0$ is trivial ($P=0$).	
	Suppose therefore that $n=2$ (the argument actually works for $n=3$, too). 
	Since $B$ is a cyclic cochain complex with $B^0=\R\cdot 1$ and $B^1=0$, there is a unique $v\in B$ with $\deg v = 2$ and $\la 1,v\ra = 1$ such that $B=\mathrm{span}\{1,v\}$. 
	Consider the basis $e_1=1$, $e_2=v$ of~$B$.
	Then $g^{ij}=\la e^i,e^j\ra = 0$ unless $\{i,j\}=\{1,2\}$ in \eqref{eq:mu}, hence the restriction $\fp_{2,1,0}\colon (\ol{\dcbc} B)^{\otimes 2}\to \ol{\dcbc} B$ vanishes (for the same reason $\fp_{1,2,0}$ vanishes).
	The restriction of the operation given by \eqref{eq:twist-0} for $(\ell,g)\neq (1,0)$ to $\ol{\dcbc}B$ therefore vanishes for every strictly unital Maurer--Cartan element $\fm$ in $\dIBL(B)$.
The desired equality now follows from Lemma~\ref{lem:alg=KS}.
\end{proof}

The same proof works in the analytic case as well, and we have the following:

\begin{corollary}\label{cor:equality-in-the-reduced-case}
	Let $M$ be a closed oriented manifold such that its de Rham cohomology $\HdR\coloneqq \HdR(M)$ is $1$-connected, and let $P\colon\Om\coloneqq\Om(M)\to\Om$ be a special analytic propagator.
	Consider the following three Maurer--Cartan elements in $\dIBL(\HdR)$ associated to $P$: the analytic Maurer--Cartan element~$\fm^{\ana}_P$, the algebraic Maurer--Cartan element~$\fm^{\alg}_P$, and the Kontsevich-Soibelman Maurer--Cartan element~$\fm^\ks_P$.
	Then we have
	\begin{equation}\label{eq:reduced-equality}
		\ol{\dIBL^{\fm^{\ana}_P}}(\HdR) = \ol{\dIBL^{\fm^{\alg}_P}}(\HdR) = \ol{\dIBL^{\fm^\ks_P}}(\HdR),
	\end{equation}	
	which is the $\dIBL$ algebra
	\begin{equation*}
		\bigl(\ol{\dcbc} \HdR[2-n],\fq^{\fm^\ks_P}=\{\fq_{1,1,0}=\bb^*,\fq_{2,1,0},\fq_{1,2,0}\}\bigr),
	\end{equation*}
	where $\bb\colon\cbc \HdR\to \cbc \HdR$ is the Hochschild differential associated to the homotopy transferred $\Ainfty$ algebra $(\HdR,\{\m_i\})$ associated to $P$.
\end{corollary}

In the situation of Corollary~\ref{cor:equality-in-the-reduced-case}, we have moreover the following:
\begin{itemize}
	\item If $M$ is in addition geometrically formal, then Corollary~\ref{cor:ana-vanishing} implies that we can choose $P$ such that the $\dIBL$ algebra \eqref{eq:reduced-equality} equals the reduced canonical twisted $\dIBL$ algebra $\ol{\dIBL^{\fm^\can_{\HdR}}}(\HdR)$, so that $\bb$ is the Hochschild differential for $(\HdR,\wedge)$.
	\item If $M$ is formal and $\HdR$ is $2$-connected, then Corollary~\ref{cor:ibl-formality} implies that the $\dIBL$ algebra \eqref{eq:reduced-equality} is $\IBLinfty$ homotopy equivalent to $\ol{\dIBL^{\fm_{\HdR}^\can}}(\HdR)$.
\end{itemize}
Note that for $n=\dim(M)\ge 3$ we even have $\fm_P^\ana=\fm_P^\alg = \fm_P^{\ks}$ by Corollary~\ref{cor:ana-vanishing}. The only reason why we cannot strengthen the equality \eqref{eq:reduced-equality} of reduced twisted $\IBLinfty$ algebras to an equality of the nonreduced versions, or even to an equality of Maurer--Cartan elements, is the lack of vanishing results for $M=S^2$.

\emergencystretch=1em
\printbibliography

\end{document}